\documentclass[leqno,11pt]{amsart}
\usepackage{amsfonts}
\usepackage[backref=page]{hyperref}
\usepackage{amsmath,amssymb}
\usepackage{latexsym}
\numberwithin{equation}{section}
\newtheorem{theo}{Theorem}[section]

\newtheorem{corol}[theo]{Corollary}
\newtheorem{prop}[theo]{Proposition}
\theoremstyle{definition}
\newtheorem{remark}[theo]{Remark}
\newtheorem{remarks}[theo]{Remarks}
\newtheorem{example}[theo]{Example}
\newtheorem{examples}[theo]{Examples}
\newtheorem{defi}[theo]{Definition}
\newcommand{\clos}{\operatorname{cl}}

\newcommand{\dom}{\operatorname{dom}}
\newcommand{\conv}{\operatorname{co}}

\newcommand{\diam}{\operatorname{diam}}

 \newcommand{\domain}{\operatorname{dom}}

\newcommand{\graph}{\operatorname{Graph}}

\newcommand{\norm}{\|\cdot\|}

\newcommand{\bequ}{\begin{equation}} %
\newcommand{\eequ}{\end{equation}} %
\newcommand{\bequs}{\begin{equation*}} %
\newcommand{\eequs}{\end{equation*}}
\newcommand{\beqs}{\begin{equation*}} %
\newcommand{\eeqs}{\end{equation*}}
\newcommand{\mbx}{\mbox}
\newcommand{\vphi}{\varphi}
\newcommand{\epsic}{\varepsilon}

\newcommand{\Real}{\mathbb R}
\newcommand{\Nat}{\mathbb N}
\newcommand{\Lra}{\Longrightarrow}
\newcommand{\Ra}{\Rightarrow}

\newcommand{\uar}{\uparrow\!\!}

\newcommand{\sms}{\smallsetminus}
\newcommand{\ety}{\emptyset}

\newcommand{\sse}{\subseteq}

\newcommand{\rd}{\mathcal}

\newcommand{\ba}{\begin{array}}
\newcommand{\ea}{\end{array}}
\newcommand{\ben}{\begin{enumerate}}
\newcommand{\een}{\end{enumerate}}
\newcommand{\eite}{\end{itemize}}
\newcommand{\bite}{\begin{itemize}}
\newcommand{\bc}{\begin{center}}
\newcommand{\ec}{\end{center}}
\newcommand{\bfr}{\begin{flushright}}
\newcommand{\efr}{\end{flushright}}
\newcommand{\f}{\frac}
\newcommand{\ov}{\overline}

\newcommand{\eps}{\varepsilon}

\newcommand{\lbda}{\lambda}

\begin{document}

\title[Fixed points and completeness]{Fixed points and completeness in metric and in generalized metric spaces}
\author{S. Cobza\c s}
\date{February  08, 2018}
\address{\it Babe\c s-Bolyai University, Faculty of Mathematics
and Computer Science, 400 084 Cluj-Napoca, Romania}\;\;
\email{scobzas@math.ubbcluj.ro}

\begin{abstract}
The famous Banach Contraction Principle holds in complete metric spaces, but completeness is not a necessary condition -- there are incomplete metric spaces on which every contraction has a fixed point. The aim of this  paper  is to present various circumstances in which fixed point results imply completeness. For  metric spaces  this is the case of Ekeland variational principle and of its equivalent - Caristi fixed point theorem.    Other fixed point results having this property will be also presented in metric spaces, in quasi-metric spaces and in partial metric spaces. A discussion on topology and order and  on fixed points in ordered structures and their completeness properties is included as well.

\textbf{AMS 2010 MSC}:

Primary: 47H10

Secondary:  47H04 54C60 54E15 54E35 06B30 06F30

\end{abstract}
\maketitle

 \emph{Motto}:\\ \, \emph{All roads lead to Rome} -- Caesar Augustus. \footnote{see Wikipedia, \emph{Milliarium Aureum}.}\\
  \noindent \emph{All topologies come from generalized metrics} -- Ralph Kopperman\footnote{Amer. Math. Monthly 95 (1988), no. 2, 89--97, \cite{kopper88}.}.
\tableofcontents
\section*{Introduction}
The famous Banach Contraction Principle holds in complete metric spaces, but completeness is not a necessary condition -- there are incomplete metric spaces on which every contraction has a fixed point, see, e.g., \cite{elek09}. The aim of the present paper  is to present various circumstances in which fixed point results imply completeness. For  metric spaces  this is the case of Ekeland variational principle (and of its equivalent - Caristi fixed point theorem)  (see, for instance,  the surveys \cite{berinde10}, \cite{lee-choi14}, \cite{suliv81b}) but this is also true  in quasi-metric spaces (\cite{cobz11}, \cite{karap-romag15}) and in partial metric spaces (\cite{altun-romag13}, \cite{romag10}). Other fixed point results having this property will be also presented.

Various order completeness conditions of some ordered structures implied by fixed point properties  will be considered as well.

Concerning proofs -- in several cases we give proofs, mainly to the converse results, i.e. completeness  implied by fixed point results. In Sections \ref{S.top-ord} we give full proofs to results relating  topology and order as well as in Section \ref{S.pm-sp}  in what concerns the properties of partial metric spaces.

\section{Banach contraction principle in metric spaces}\label{Ban-CP}

Banach Contraction Principle was proved by S. Banach in his thesis from 1920, published in 1922, \cite{banach22}. Although the idea of successive approximations in some concrete situations (solving differential and integral equations) appears in some works of E. Picard, R. Caccioppoli, \emph{et al.}, it was Banach who placed it in the right abstract setting, making it suitable for a wide range of applications (see the expository paper \cite{kirk01}).

\subsection{Contractions and contractive mappings}

Let $(X,\rho)$ and $(Y,d)$ be  metric spaces. A mapping $f:X\to Y$
is called {\it Lipschitz} if there exists a number
$\alpha \geq 0$ such that
\begin{equation}\label{def.Lip}
\forall x,y\in X,\quad d(f(x),f(y))\leq \alpha \rho(x,y).
\end{equation}
The number $\alpha$ is called a {\it Lipschitz constant} for $f,$
and one says sometimes that the mapping $f$ is $\alpha$-{\it
Lipschitz}. If $\alpha =0$, then the mapping $f$ is constant
$f(x)=f(x_0) $ for some point $x_0\in X.$ If $\alpha =1,$
i.e.
\begin{equation}\label{def.nexp}
\forall x,y\in X,\quad d(f(x),f(y))\leq \rho (x,y),
\end{equation}
then the mapping $f$ is called {\it nonexpansive}. If
\begin{equation}\label{def.isomet}
\forall x,y\in X,\quad d(f(x),f(y))= \rho (x,y),
\end{equation}
then $f$ is called an {\it isometry.}

Suppose now $Y=X.$ An $\alpha$-Lipschitz mapping $f:X\to X$ with
$0\leq \alpha <1$ is called a {\it contraction}. A mapping
$f:X\to X$ satisfying the relation
\begin{equation}\label{def.contr}
\forall x,y\in X,\; x\neq y,\quad \rho(f(x),f(y))< \rho (x,y),
\end{equation}
is called {\it contractive.}

A point $x_0\in X$ such that $f(x_0)=x_0$ is called a {\it fixed
point} of the mapping $f:X\to X.$ The study of the fixed points
 of mappings is one of the most important branches of
mathematics, with numerous applications to the solution of various
kinds of equations (differential, integral, partial differential,
operator), optimization, game theory, etc.

The following theorem
is, perhaps, the most known fixed point result.
\begin{theo}[Banach's Contraction Principle]\label{t.fp-Ban}
Any contraction  on a complete metric space  has a fixed point.

More exactly, suppose that for some $\alpha,\, 0\leq \alpha <1,\,$
 $f$ is an $\alpha$-contraction on a complete metric space
$(X,\rho).$ Then, for an arbitrary point $x_1\in X,$
the  sequence $(x_n)$ defined by the recurrence relation
\begin{equation}\label{eq1.fp-Ban}
x_{n+1}=f(x_n), \; n\in \mathbb{N},
\end{equation}
converges to a fixed point $x_0$ of the mapping $f, $ and  the
following estimations hold:
\begin{equation}\label{eq2.fp-Ban}
\begin{aligned}
{\rm (a)}& \quad\forall n\in \mathbb{N},\quad \rho(x_n,x_{n+1})\leq
\alpha^{n-1}\rho(x_1,x_2);\\
{\rm (b)}&\quad \forall n\in \mathbb{N},\, \forall k\in \mathbb{N},\quad
\rho(x_n,x_{n+k})\leq\frac{1-\alpha^k}{1-\alpha}\alpha^{n-1}\rho(x_1,x_2);\\
{\rm (c)}&\quad  \forall n\in \mathbb{N},\quad \rho(x_n,x_{0})\leq
\frac{\alpha^{n-1}}{1-\alpha}\rho(x_1,x_2).
\end{aligned}
\end{equation}
\end{theo}

Under a supplementary  condition  contractive mappings also have fixed
points.
\begin{theo}[M. Edelstein (1962) \cite{edel61,edel62}]\label{t.fp-Edl}
Let  $(X,\rho)$ be a   metric space and   $f:X\to X$ a contractive
mapping.  If there exists $x\in X$ such that the sequence of iterates  $(f^n(x))$ has a limit point $\xi\in X,$ then $\xi$ is the unique fixed point of $f$.\end{theo}

Theorem \ref{t.fp-Edl} has the following important consequence.

\begin{corol}[Nemytski\v{\i} (1936)\, \cite{nemyts36}] \label{c.fp-Edl}
If the metric space  $(X,\rho)$ is  compact,   then every   contractive
mapping  $f:X\to X$  has a unique fixed point in $X$.

Moreover, for any $x_1\in X$ the sequence defined by
$x_{n+1}=f(x_n),\, n\in \mathbb{N},\, $ converges
to the fixed point of the mapping $f.$
\end{corol}

Fixed point results for isometries were proved by Edelstein in \cite{edel63}.

\subsection{Converses of   Banach's contraction principle} Supposing that a function $f$ acting on a metric space $(X,\rho)$ has a unique fixed point one looks for conditions ensuring the existence of a metric $\bar \rho$  on $X$, topologically equivalent to $\rho$ such that  $f$ is a contraction on $(X,\bar \rho)$. The first result of this kind was obtained by Bessaga \cite{besag59}. Good presentations of various aspects  of fixed points for contraction mappings and their  generalizations as well as converse-type results are contained in Ivanov \cite{ivanov76}, Lahiri \emph{et al.} \cite{lahiri09}, Kirk
\cite{kirk01}, Opoitsev \cite{opoitsev76},  Rus \cite{Rus01}, Rus \cite{Rus79}, Rus and  Petru\c sels  \cite{Rus-PP}.

We say that a metric $d$ on a set $X$ is \emph{complete} if $(X,d)$ is a complete metric space.

\begin{theo}[Cz. Bessaga (1959) \cite{besag59}]\label{t.Bessaga}
Let $X$ be a nonempty set, $f:X\to X$ and $\alpha\in (0,1).$
\begin{enumerate}
 \item If for every $n\in\mathbb{N},\, f^n$  has at most one fixed point, then there exists a metric $\rho$ on $X$ such that  $f$ is an $\alpha$-contraction with respect to  $\rho$.
  \item If, in addition,  some  $ f^n$  has a fixed point, then there exists a complete metric $\rho$ on $X$ such that  $f$ is an $\alpha$-contraction with respect to  $\rho$.
\end{enumerate}
\end{theo}

A different proof of Theorem \ref{t.Bessaga} was  given Wong \cite{wong66}, a version of which  is included   in Deimling's book on nonlinear functional analysis, \cite[p. 191-192]{Deimling}. Other proofs as well as some extensions were given by  Babu \cite{babu82}, Jachymski \cite{jachym00} (see also \cite{jachym94}),   Palczewski and  Miczko \cite{palczew86,palczew87}, Wang \emph{et al.} \cite{wang99} (cf. the  MR review). Angelov \cite{angel88,angel94} proved a converse result in the context of uniform spaces.

In the case of compact metric spaces Jano\v s \cite{janos67} proved the following result.
\begin{theo}\label{t.Janos}
  Let $(X,\rho)$ be a compact metric space and $f:X\to X$ be a continuous mapping such that, for some $\xi\in X,$
  \begin{equation}\label{eq1.Janos}
  \bigcap_{n=1}^\infty f^n(X)=\{\xi\}\,.
  \end{equation}

Then for every $\alpha \in (0,1)$ there exists a metric $\rho_\alpha$ on $X$, topologically equivalent to $\rho$, such that  $f$ is an $\alpha$-contraction with respect to  $\rho_\alpha$ (with $\xi$ as the unique fixed point).
\end{theo}

A mapping $f$ satisfying \eqref{eq1.Janos} is called \emph{squeezing}.

Another proof of Jano\v s' theorem was given by Edelstein \cite{edel69}.

Kasahara \cite{kasahara68} showed that compactness is also necessary for the validity of Jano\v s' result.

\begin{theo}\label{t.Kasah} Let $(X,\rho)$ be a metric space. If for every squeezing mapping $f:X\to X$ and every $\alpha\in (0,1)$ there exits a metric $\rho_\alpha$ on $X$, topologically equivalent to $\rho$, such that  $f$ is an $\alpha$-contraction with respect to  $\rho_\alpha$, then the space $X$ is compact.
\end{theo}

Jano\v s extended in \cite{janos70} this result to uniform spaces (more exactly, to completely regular spaces whose topology is generated by a family of semimetrics), see also Angelov \cite{angel87,angel88,angel94,angel04}. Rus \cite{rus93} extended Jano\v s' result to weakly Picard mappings.
An operator $f$ on a metric space $(X,\rho)$ is called \emph{weakly Picard} if,  for every $x\in X$, the sequence  $(f^n(x))_{n=1}^\infty$ of iterates  converges to a fixed point of $f$. If further, the limit is independent of $x$ (i.e. $f$ has a unique fixed point), then  $f$ is called a \emph{Picard operator} (see \cite{rus03} or \cite{Rus-PP}).

 Other extensions of Jano\v{s}' result were given by Leader \cite{leader77} (see also Leader \cite{leader82,leader83}), Meyers \cite{meyers65,meyers67},  Mukherjee and  Som \cite{mukherj84}. For a metric space $(X,\rho)$ and $\xi\in X$ consider the following properties:
 \begin{equation}\label{def.iter}
 \begin{aligned}
   &{\rm (i)}\quad f^n(x)\to\xi \quad\mbox{for every}  \; x\in X;\\
    &{\rm (ii)}\quad \mbox{the convergence in (i) is uniform on some neighborhood } U\;\mbox{of}\; \xi\,.
 \end{aligned}\end{equation}
The condition (ii) means that
  \begin{equation}\label{eq1.iter}
         \forall \varepsilon >0,\; \exists n_0=n_0(\varepsilon),\;\; \mbox{such that} \;\;\; \forall n\ge n_0,\; f^n(U)\subseteq B[\xi,\varepsilon]\,.
         \end{equation}

To designate the uniform convergence on a subset $A$ of $X$ of  the sequence $(f^n)$ to a point $\xi$, one uses the notation
$$
f^n(A)\to\xi\,.
$$

 Leader \cite{leader77} proved the following   results.
  \begin{theo}\label{t.Leader} Let $(X,\rho)$ be a metric space and $f:X\to X.$
 \begin{enumerate}
 \item
  There exists a metric $\bar \rho$  topologically equivalent to $\rho$ on $X$
such that $f$ is a Banach contraction under $\bar \rho$ with fixed point $\xi$  if,
and only if, $f$ is continuous and both (i) and (ii) from\eqref{def.iter} hold.
 \item
  There exists a bounded metric $\bar \rho$ topologically equivalent
to $\rho$ on $X$ such that $f$ is a Banach contraction under $\bar \rho$  with fixed
point $\xi$ if, and only if, $f$ is continuous and $f^n(X)\to \xi$.
 \item
 There exists a bounded metric $\bar \rho$ uniformly equivalent to
$\rho$ on $X$ such that $f$ is a Banach contraction under $\bar \rho$ if, and only
if, $f$ is uniformly continuous and
\begin{equation}\label{eq1.Leader}
\diam_\rho(f^n(X))\to 0\;\;\mbox{ as }\; n\to\infty\,.
\end{equation}
\end{enumerate}
   \end{theo}

   We mention also  a topological characterization of contractions given by Florinskij \cite{florinski87}.
   \begin{theo}Let $(X,\tau)$ be a metrizable topological space and $f:X\to X$ a continuous mapping. The following are equivalent.
   \ben
   \item There exists a metric $\rho$ on $X$ generating the topology $\tau$ such that $f$ is a Banach contraction with respect to $\rho$.
   \item There exists a nonempty open subset $G$ of $X$ such that
   \bite
   \item[\rm(i)] $f(G)\subseteq G$;
    \item[\rm(ii)] $\diam\left(\bigcap_{n=1}^\infty \ov{f^n(G)}\right) =0$\;\; (with the convention that $\diam \emptyset =0$);
     \item[\rm(iii)] for every $x\in X$ there exists $n_x\in\Nat$ such that $f^n(x)\in G$ for all $n\ge n_x$.
   \eite\een
   \end{theo}
   \begin{remark}
     Florinskij \cite{florinski87} gave also topological characterizations of mappings $f:X\to X$ that are:
     \bite\item nonexpansive, i.e.\;   $\rho(f(x),f(y))\le\rho(x,y),\, x,y\in X$;
     \item contractive, i.e.\;   $\rho(f(x),f(y))< \rho(x,y),\, x,y\in X$;
     \item separating, meaning that \; $\rho(f(x),f(y))\ge\alpha \rho(x,y),\, x,y\in X$,  for some $\alpha\in(0,1)$.\eite

   The paper \cite{florinski87} contains no proofs and I did not find them elsewhere.
    \end{remark}

  A general remetrization result  was obtained by Florinskij \cite{florinski97}, guaranteing that every contraction has a fixed point (without completeness).
  \begin{theo}
    On every  metric space  $(X,\rho)$  there exists a metric $\bar\rho $, topologically  equivalent to $\rho$, such that
the Banach Contraction  Principle holds in the space $(X,\bar\rho)$, i.e., every  contraction on $(X, \bar\rho)$ has a fixed point.
  \end{theo}

  The proof is based on the following result, proved in the same paper.
  \begin{theo}
 For every metric space $(X,\rho)$ there exists a metric space  $(\tilde X,\tilde\rho)$, containing $(X,\rho)$ as a dense subspace, such that
  the space $(\tilde X,\tilde \rho)$ is connected, locally connected and uniformly chainable.
 \end{theo}

 Let $(X,\rho)$ be a metric space and $\epsic>0.$ An $\epsic$-\emph{chain} joining two points $x,y\in X$ is a set $x=x_1,x_2,\dots,x_n=y$ of points in $X$ such that $d(x_i,x_{i+1})\le\epsic$ for $i=1,2,\dots,n-1.$  The metric space $(X,\rho)$ is called \emph{chainable} if for every $x,y\in X$ and $\epsic>0$ there exists an $\epsic$-chain connecting $x$ and $y$, and \emph{uniformly chainable} if for every $\epsic>0$ there exists $n(\epsic)\in\Nat$ such that any pair  $x,y$ of points in $X$ can be joined by an $\epsic$-chain containing at most $n(\epsic)$ elements.

 The following curious result was also obtained by Florinskij \cite{florinski91}.
 \begin{theo}
   Let $(X,\rho)$ be a separable and uniformly chainable metric space. Then there exists a metric $\tilde\rho$ on $X$, uniformly equivalent to  $\rho$, such that
   every contraction on $(X\tilde\rho)$ is a constant mapping, i.e. applies $X$ in a single  point.
 \end{theo}

 \textbf{   Ultrametric spaces}

  In the case of an ultrametric space, the situation is simpler. An \emph{ultrametric space} is a metric space $(X,\rho)$ such that  $\rho$ satisfies the so called \emph{strong triangle} (or \emph{ultrametric})  \emph{inequality}
 \begin{equation}\label{eq.ultram-ineq}
 \rho(x,z)\le\max\{\rho(x,y),\rho(y,z)\}\,,
 \end{equation}
 for all $x,y,z\in X.$

 Bellow we present some specific properties of these spaces.

 \begin{prop}\label{p1.ultram}
   Let $(X,\rho)$ be an ultrametric space. Then for all $x,y,z\in X$ and $r>0$,
   \begin{align*}
   &{\rm (i)}\;\;\rho(x,y)\ne\rho(y,z)\;\Longrightarrow\;  \rho(x,z)=\max\{\rho(x,y),\rho(y,z)\};\\
     &{\rm (ii)}\;\;y\in B[x,r]\;\Longrightarrow\;   B[x,r] = B[y,r];\\
     &{\rm (iii)}\;\;r_1\le r_2\;\;\mbox{and}\;\; B[x,r_1]\cap B[x,r_2]\ne\emptyset\;\Longrightarrow\; B[x,r_1]\subseteq  B[x,r_2]\,.
   \end{align*}

   Similar relations hold for the open balls $B(x,r)$.
 \end{prop}

 \begin{remark} These strange properties of non-Archimedean metric spaces remind us of the following Pascal's thought:
 \smallskip

   {\it Nature is an infinite sphere of which the center is everywhere and the circumference nowhere} (Blaise Pascal, \emph{Pens\'ees}).\smallskip

    Probably the famous French philosopher and mathematician had in mind  a  non-Archimedean world.  Initially he even wrote    ``A frightful (\emph{effroyable}) sphere".

     See also the essay  on this topic by J. L. Borges at

      http://www.filosofiaesoterica.com/pascals-sphere/
 \end{remark}

 An ultrametric space $(X,\rho)$ is called \emph{spherically complete} if for every collection $B_i=B[x_i,r_i],\, i\in I,$ of closed ball in $X$ such that  $B_i\cap B_j\ne\emptyset$ for all $i,j\in I,\,$  has nonempty intersection, \;$\bigcap_{i\in I}B_i\ne\emptyset$. It is obvious that a spherically complete ultrametric space is complete. In  an arbitrary  metric space this property is called the \emph{binary intersection property}.

 Priess-Crampe \cite{priess90} proved the following converse to Edelstein's theorem on contractive mappings.
 \begin{theo}\label{t.P-Crampe}
   An ultrametric space $(X,\rho)$ is spherically complete if and only if  every contractive mapping on $X$ has a (unique) fixed point.
 \end{theo}
 \begin{remark} In fact Priess-Crampe \cite{priess90} proved this result in the more general context of an ultrametric $\rho$ taking values in a totally  ordered set $\Gamma$ having a least element 0 such that $0<\gamma$ for all $\gamma \in\Gamma$.
   \end{remark}

   Fixed point theorems for contractive and for nonexpansive mappings on spherically complete non-Archimedean normed spaces were proved by Petalas and Vidalis \cite{petalas-vid93}.

 Concerning contractions we mention the following result obtained by Hitzler and Seda \cite{hitzler01}.
 \begin{theo}\label{t.Hitz}
 Let $(X,\tau)$ be a $T_1$ topological space and $f:X\to X$ a function on $X$. The following are equivalent:
 \begin{enumerate}
 \item\;{\rm (i)}\;\; The mapping $f$ has a unique fixed point $\xi \in X$,\, and\\
 {\rm (ii)}\;\; for every $x\in X$ the sequence $(f^n(x))$ converges to $\xi$ with respect to  the topology $\tau$.
 \item There exists a complete  ultrametric $\rho$ on $X$ such that $\rho(f(x),f(y))\le 2^{-1} d(x,y)$ for all $x.y\in X$.
     \end{enumerate}
 \end{theo}

 For applications of these fixed point results to logic programming see  the paper \cite{hitzler02}.
 \begin{remark} It is not sure that the metric $\rho$ from (2)  generates the topology $\tau$, but for every $x\in X$ the sequence $(f^n(x))$ converges to $\xi$ with respect to  the topology $\tau$ and  the metric $\rho$.
 \end{remark}

\subsection{Neither completeness nor compactness is necessary}
In this subsection we shall provide some examples of peculiar topological spaces having the fixed point property (FPP) for various classes of mappings.
\begin{examples}[Elekes \cite{elek09}]\label{ex.Elekes}\hfill

1.  The space $X=\{(x,\sin(1/x)) : x\in(0,1]\}$ is a non-closed (hence incomplete) subset of $\mathbb{R}^2$ having the FPP for contractions (Theorem 1.2).

2.    For every $n\in\mathbb{N}$ every open subset of $\mathbb{R}^n$ possessing the Banach
Fixed Point Property coincides with $\mathbb{R}^n$, hence is closed (Corollary 2.2).

3.   Every simultaneously $F_\sigma$ and $G_\delta$ subset of $\mathbb{R}$  with the Banach Fixed
Point Property is closed (Theorem 2.4).

4.   There exists a nonclosed $G_\delta$ set $X\subseteq  \mathbb{R}$ with the Banach Fixed
Point Property. Moreover, $X\subseteq  [0, 1]$ and every contraction mapping of $X $ into itself
is constant (Theorem 3.3).

5.   There exists a nonclosed $F_\sigma$ subset of $[0,1]$ with the Banach Fixed
Point Property (Theorem 3.4).

6.   There is a bounded Borel (even $F_\sigma$) subset of $\mathbb{R}$  with the Banach
Fixed Point Property that is not complete with respect to every equivalent metric (Corollary 3.5).

7.   For every integer $n > 0$ there exists a nonmeasurable set in $\mathbb{R}^n$ with
the Banach Fixed Point Property (Theorem 3.6).
  \end{examples}

  The example from 1 was   presented  at the Problem Session of the 34th Winter School in Abstract Analysis, Lhota nad Rohanovem, Czech Republic, 2006, by E. Behrends, classified by him as ``folklore", along with some questions concerning the subsets of $\mathbb{R}^n$ (in particular of $\mathbb{R}$) having Banach Fixed Point Property (i.e.  FPP for contractions).

  We   give the proof only for 1, following \cite{elek09}. A  proof based on some similar ideas was given by Borwein \cite{borw83}.
  \begin{proof}[Proof of the assertion 1]
 Let $X=\{(x,\sin(1/x)) : x\in(0,1]\}$ and $f:X\to X$ be a contraction with constant $0<\alpha<1$.
 For $H\subseteq (0,1]$ put  $X_H:=\{(x,y)\in X : x\in H\}$.

 Let $0<\varepsilon<1$ be such that  $\alpha\sqrt{\varepsilon^2+4}<2.$ Then for all $z=(x,y),\, z'=(x',y')$ in $X$ with $0<x,x'<\varepsilon$,
 $$
 \|f(z)-f(z')\|\le\alpha\sqrt{(x-x')^2+(y-y')^2}<\alpha\sqrt{\varepsilon^2+4}<2\,.
 $$

 Consequently, $X_{(0,\varepsilon)}$  does not contain both a local minimum and a local maximum of the graph. Since $X_{(0,\varepsilon)}$ is connected, it follows that it is contained in at most two consecutive monotone parts of the graph of $\sin(1/x).$ Therefore there exists $\delta_1>0$ such that
 $f\left(X_{(0,\varepsilon)}\right)\subseteq X_{[\delta_1,1]}$ for some $\delta_1>0$.  By compactness $f\left(X_{[\varepsilon,1]}\right)\subseteq X_{[\delta_2,1]}$ for some $\delta_2>0$.

Taking $\delta=\min\{\delta_1,\delta_2\}$ it follows  $f\left(X\right)\subseteq X_{[\delta,1]}$ and so $f\left(X_{[\delta,1]}\right)\subseteq X_{[\delta,1]}$. Applying Banach Fixed Point Theorem to $X_{[\delta,1]}$ it follows that $f$ has a fixed point.
\end{proof}

  Some examples of spaces having the FPP for continuous mappings were given by Connell \cite{conel59}. These examples show that, in author's words: ``in the general case, compactness
and the FPP are only vaguely related".

   We first mention the following result of Klee.
  \begin{theo}[Klee \cite{klee55}] A locally connected, locally compact metric space
with the FPP for continuous mappings is compact.\end{theo}

  \begin{examples}(Connell \cite{conel59})\hfill

    1.  There exists a Hausdorff topological space $X$ having the FPP for continuous mappings such that  the only compact subsets of $X$ are the finite ones.

    2. There exists a metric space $X$ having the FPP for continuous mappings  such that  $X^2$ does not have the  FPP for continuous mappings.

    3.  There exists a separable, locally contractible
metric space that has the FPP for continuous mappings, yet it is not compact.

4.  There exists a compact metric space $X$ that
does not have the FPP for continuous mappings, yet it contains a dense subset $Y $ that does have
the FPP for continuous mappings.
  \end{examples}

  \subsection{Completeness and other properties implied by FPP}

  We shall present some fixed point results that imply the completeness of the underlying space. The papers    \cite{berinde10}, \cite{lee-choi14} and \cite{suliv81b} contain surveys on this topic. A good analysis is given in the Master Thesis of  Nicolae \cite{Nicolae}.

  We first mention the following characterization of the field of real numbers among totally ordered fields.

Suppose $R$ is an ordered field. Call a continuous map $f:R\to R $ a contraction if there exists $r<1$ (in $R$) such that $|f(x)-f(y)|\le r|x-y|$ for all $x,y\in R$ (where $|x|:=\max\{x,-x\}$).

The following result is taken from \smallskip

http://mathoverflow.net/questions/65874/converse-to-banach-s-fixed-point-theorem-for-ordered-fields\smallskip

Asking a question posed by James Propp, George Lowther proved the following result.

\begin{theo} If $R$ is an ordered field such that every contraction on $R$ has a fixed point, then
$R\cong \mathbb{R}$.
\end{theo}

The proof is done in two steps:\vspace{2mm}

I. \emph{one shows first that the order of $R$ is Archimedean},\vspace{2mm}

\noindent and then\vspace{2mm}

II. \emph{one proves that every Cauchy sequence is convergent (i.e. the completeness of $R$)},\vspace{2mm}

\noindent two properties that characterize the field $\mathbb{R}$ among the ordered fields.

The first characterization of completeness in terms of contraction was done by Hu \cite{hu67}.
\begin{theo}\label{t.Hu}
A metric space $(X,\rho)$ is complete if and only if  for every nonempty closed subset $Y$ of $X$ every contraction on $Y$ has a fixed point in $Y$.
\end{theo}\begin{proof}
The idea of the proof is simple. One takes a Cauchy sequence $(x_n)$ in $X$. If it has a convergent subsequence, then it converges. Supposing that this is not the case, then $\beta(x_n):= \inf\{\rho(x_n,x_m): m>n\}>0$ for all $n\in\mathbb{N}$. For a given     $\alpha$ with $0<\alpha<1$, one constructs inductively a subsequence $(x_{n_k})$  such that
$\rho(x_i,x_j)\le\alpha \beta(x_{n_{k-1}})$ for all $i,j\ge n_k.$ Then  $Y=\{x_{n_k} : k\in \mathbb{N}\}$ is a closed subset of $X$ and the function   $f(x_{n_k})=x_{n_{k+1}},\, k\in\mathbb{N}$, is an $\alpha$-contraction on $Y$, because
$$
\rho(f(x_{n_{k}}),f(x_{n_{k+i}}))=\rho(x_{n_{k+1}},x_{n_{k+i+1}})\le\alpha \beta(x_{n_k})\le
\alpha \rho(x_{n_{k}},x_{n_{k+i}})\,,$$
for all $k,i\in\mathbb{N}$. It is obvious that $f$ has no fixed points.
\end{proof}

Subrahmanyam \cite{subrahm75} proved the following completeness result.
\begin{theo}\label{t.Subram}
A metric space $(X,\rho)$ in which every mapping $f:X\to X$ satisfying the conditions

{\rm (i)}\; there exists $\alpha>0$ such that  $\rho(f(x),f(y))\le \alpha \max\{\rho(x,f(x)),\rho(y,f(y))\}$ for all $ x,y\in X$;

{\rm (ii)}\; $f(X)$ is countable;\\
has a fixed point, is complete.
\end{theo}

The condition (i) in this theorem is related to Kannan  and Chatterjea conditions: there exists $\alpha\in(0,\frac12)$ such that  for all $x,y\in X,$
\begin{equation}\tag{K}\label{Kan}
\rho(f(x),f(y))\le \alpha \,\left[\rho(x,f(x))+\rho(y,f(y))\right]\,,
\end{equation}
respectively
\begin{equation}\tag{Ch}\label{Chat}
\rho(f(x),f(y))\le \alpha  \,\left[\rho(x,f(y))+\rho(y,f(x))\right]\,.
\end{equation}

Kannan and Chatterjea  proved that any mapping $f$ on a complete metric space satisfying  \eqref{Kan} or \eqref{Chat} has a fixed point (see, for instance, \cite{Rus-PP}).  As it is remarked in \cite{subrahm75} Theorem \ref{t.Subram} provides completeness of metric spaces on which every Kannan, or every Chatterjea map, has a fixed point.

Another case when the fixed point property for contractions implies completeness was discovered by Borwein \cite{borw83}.

A metric space $(X,\rho)$ is called \emph{uniformly Lipschitz connected} if there exists $L\ge 0$ such that  for any pair $x_0,x_1$ of points in $X$ there exists a mapping $g:[0,1]\to X$ such that  $g(0)=x_0,\, g(1)=x_1$ and
\begin{equation}\label{eq1.borw-contr}
\rho(g(s),g(t))\le L|s-t|\rho(g(0),g(1))\,,
\end{equation}
for all $s,t\in[0,1]$.

Obviously, a convex subset $C$ of a normed space $X$ is uniformly Lipschitz connected, the mapping
$g$ connecting $x_0,x_1\in C$ being given by $g(t)=(1-t)x_0+tx_1,\, t\in [0,1]$. In this case
$$
\|g(s)-g(t)\|=|s-t|\|x_1-x_0\|\,,
$$
for all $s,t\in [0,1]$.

From the following theorem it follows that a convex subset $C$ of a normed space $X$ is complete if and only if any contraction on $C$ has a fixed point. In particular this holds for the normed space $X$.
\begin{theo}\label{t.Borw-compl}
  Let $C$ be a uniformly Lipschitz connected subset of a complete metric space $(X,\rho)$. Then the following conditions are equivalent.
  \begin{enumerate}
    \item The set $C$ is closed.
    \item Every contraction on $C$ has a fixed point.
    \item Any contraction on $X$ which leaves $C$ invariant has a fixed point in $C$.
  \end{enumerate}
\end{theo}\begin{proof}
The implication (1) $\Rightarrow$ (2) is Banach Fixed Point Theorem and (2) $\Rightarrow$ (3) is obvious.

It remains to prove (3) $\Rightarrow$ (1). Supposing that $C$ not closed, there exists a point $\overline x\in\overline C\smallsetminus C.$ Let $(x_k)_{k\in\mathbb{N}_0}$ be a sequence of pairwise distinct points in $C$ such that
\begin{equation}\label{eq1a.Borw-compl}
\rho(x_k,\overline x)\le\min\big\{\frac1{2^{k+4}},\frac{L}{2^{k+4}}\big\}\,,
\end{equation}
for $k=0,1,\dots$, where $L>0$ is the constant given by the uniform Lipschitz connectedness of $C$.

It follows
\begin{equation}\label{eq2.Borw-compl}
\rho(x_k,x_{k+1})\le\min\big\{\frac1{2^{k+3}},\frac{L}{2^{k+3}}\big\}\,,
\end{equation}
for all $k\in \mathbb{N}_0$.

Let $g_k:[0,1]\to C$ be such that  $g_k(0)=x_k,\, g_k(1)=x_{k+1}$ and
\begin{equation}\label{eq3.Borw-compl}
\rho(g_k(s),g_k(t))\le L|s-t|\rho(x_k,x_{k+1})\,,
\end{equation}
for all $s,t\in [0,1]$. Define $g:(0,\infty)\to C$ by
\begin{equation}
g(t)=\begin{cases}
  x_0  &\mbox{for}\;\; 1<t<\infty,\\
g_k(2^{k+1}t-1) &\mbox{for}\;\; \frac1{2^{k+1}}<t\le\frac1{2^{k}}.
\end{cases}\end{equation}

It follows $g(2^{-k})=g_k(1)=x_{k+1}$.

Let $\Delta_k=(2^{-(k+1)},2^{-k}]$. Then for $s,t\in\Delta_k,$  taking into account  \eqref{eq3.Borw-compl} and \eqref{eq2.Borw-compl}, one obtains
\begin{align*}
\rho(g(s),g(t)) &\le L\cdot 2^{k+1} |s-t|\rho(x_k,x_{k+1}) \\
&\le L\cdot 2^{k+1}\cdot |s-t|\cdot\frac 1{2^{k+3}} =\frac{L}4\cdot|s-t|\le L\cdot  |s-t|\,.
\end{align*}

Since $|s-t|<\frac1{2^{k+1}}$, it follows also that
$$
\rho(g(s),g(t)) \le L\cdot 2^{k+1}\cdot\frac 1{2^{k+1}}\cdot\frac 1{2^{k+3}}=\frac L{2^{k+3}}\,,
$$
for all $s,t\in \Delta_k$.

If $s\in \Delta_k$ and $t\in \Delta_p$ with $k\le p$, then the above inequality and \eqref{eq1a.Borw-compl} yield
\begin{align*}
  \rho(g(s),g(2^{-k})) &\le \frac L{2^{k+3}};\\
  \rho(x_{k+1},x_{p+1}) &\le \rho(x_{k+1},\overline x)+\rho(\overline x,x_{p+1}) \le L\left(\frac1{2^{k+5}}+\frac1{2^{p+5}}\right);\\
   \rho(g(2^{-p}),g(t)) &\le \frac L{2^{p+3}}\,,
\end{align*}
so that
\begin{align*}
   \rho(g(s),g(t)) &\le  \rho(g(s),g(2^{-k}))+\rho(x_{k+1},x_{p+1})+\rho(g(2^{-p}),g(t))\\
   &\le L\cdot \left(\frac 1{2^{k+3}} +\frac1{2^{k+5}}+\frac1{2^{p+5}}+\frac 1{2^{p+3}}\right)\,.
\end{align*}

Observe that $s-t>\frac1{2^k}-\frac1{2^{p+1}}$, and so   if we show that
\begin{equation}\label{eq4.Borw-compl}
\frac 1{2^{k+3}} +\frac1{2^{k+5}}+\frac1{2^{p+5}}+\frac 1{2^{p+3}}\le \frac1{2^k}-\frac1{2^{p+1}}\,,
\end{equation}
then
\begin{equation}\label{eq5.Borw-compl}
\rho(g(s),g(t))\le L|s-t|\,.
\end{equation}

Since all the fractions with $p$ at the denominator are less or equal to the corresponding ones with $k$ at the denominator, it follows
\begin{align*}
 &\frac 1{2^{k+3}} +\frac1{2^{k+5}}+\frac1{2^{p+5}}+\frac 1{2^{p+3}} + \frac1{2^{p+1}}\\
 &\le \frac 1{2^{k+2}} +\frac1{2^{k+4}}+\frac1{2^{k+1}}  =\frac{13}{2^{k+4}}<\frac1{2^k}\,,
\end{align*}
so that \eqref{eq4.Borw-compl} holds.

Put now $g(0)=\overline x$. If $t\in\Delta_k$, then
\begin{align*}
\rho(g(0),g(t)) &\le \rho(\overline x,x_{k+1})+\rho(x_{k+1},g(t))\\
&\le L\,\left(\frac1{2^{k+5}}+ \frac1{2^{k+3}}\right)<L\cdot\frac1{2^{k+1}}<L\cdot t\,,
  \end{align*}
  showing that $g$ satisfies \eqref{eq5.Borw-compl} for all $s,t\in[0,\infty)$. Let $h:X\to[0,\infty)$ and $f:X\to X$ be defined for $x\in X $ by
  $$
  h(x):=(2L)^{-1}\rho(x,\overline x)\quad\mbox{ and  }\quad f(x):=(g\circ h)(x)\,,$$
   respectively. Then, for all $x,x'\in X$,
\begin{align*}
  \rho(f(x),f(x')) &=\rho\left(g\big(\frac1{2L}\rho(x,\overline x)\big), g\big(\frac1{2L}\rho(x',\overline x)\big)\right)\\
  &\le L\cdot\frac1{2L}\,|\rho(x,\overline x)-\rho(x',\overline x)|\le\frac12\cdot \rho(x,x')\,,
\end{align*}
that is $f$ is a $\frac12$-contraction on $X$. Because
$$
f(C)=g(h(C))\subseteq g((0,\infty))\subseteq C\,,
$$
it follows that $C$ is invariant for $f$. Since
$$
\overline x=g(0)=g(h(\overline x))=f(\overline x)\,,
$$
it follows that the only fixed point of $f$ is $\overline x$, which does not belong to $  C$, in contradiction to the   hypothesis.
  \end{proof}

  We mention the following consequences.
  \begin{corol}\label{c.Borw-compl}\hfill
    \begin{enumerate}
    \item A uniformly Lipschitz connected metric space $(X,\rho)$ is complete if and only if  it has the fixed point property for contractions.
\item       A convex subset $C$ of a normed space $X$ is complete if and only if
any contraction on $C$ has a fixed point. In particular this holds for the normed space $X$.
\end{enumerate}
  \end{corol}\begin{proof}
    For (1) consider $X$ as a uniformly Lipschitz connected subset of its completion $\tilde X$.
    The results in (2) were discussed before Theorem \ref{t.Borw-compl}.
  \end{proof}

\begin{example}[Borwein \cite{borw83}]\label{ex.Borw} There is a starshaped non-closed subset of $\mathbb{R}^2$ having the fixed point property for contractions, but not for continuous functions.
\end{example}

One takes
$$L_k=\conv\big(\big\{(0,0),(1,\frac1{2^k})\big\}\big),\; k\in \mathbb{N}\,,$$
and $C=\bigcup\{L_k : k\in \mathbb{N}\}$.  Then $C$ is starshaped with respect to $(0,0)$ and non-closed, because $\conv(\{(0,0),(1,0)\})\subseteq \overline C\smallsetminus C$. One shows that $C$ has the required properties, see \cite{borw83} for details.

Xiang \cite{xiang07} completed and extended  Borwein's  results.  Let $(X,\rho)$ be a metric space. By an arc we mean a continuous function $g:\Delta\to X$, where $\Delta$ is an interval in $ \mathbb{R}$. An arc $g:(0,1]\to X$ is called \emph{semi-closed} if
\begin{equation}\label{def.semi-cl}
\forall \varepsilon>0,\; \exists\delta>0,\; \mbox{ such that }\; \rho(g(s),g(t))<\varepsilon\; \mbox{ for all }\; s,t\in(0,\delta)\,.
\end{equation}

The arc $g$ is called \emph{Lipschitz semi-closed }if  the mapping  $g$ is Lipschitz and satisfies \eqref{def.semi-cl}.

The metric space $(X,\rho)$ is called \emph{arcwise complete} if for every semi-closed arc $g:(0,1]\to X$ there exists the limit $\lim_{t\searrow 0}g(t)$. If this holds for every Lipschitz semi-closed arc $g:(0,1]\to X$, then $X$ is called \emph{Lipschitz complete}.

 Some examples,  \cite[Examples 1.1, 1.2 and 2.3]{xiang07}, show that the arcwise completeness is weaker than the usual completeness even   in an
arcwise connected space, and so is Lipschitz completeness.  It is obvious from the definitions that Lipschitz completeness is weaker than arcwise completeness.

A metric space $(X,\rho)$ is called \emph{locally arcwise connected} (respectively, \emph{locally Lipschitz connected}) if there exists $\delta>0$ such that  any pair $x_0,x_1$ of points in $X$ with $\rho(x_0,x_1)\le\delta$ can be linked by an arc (respectively by a Lipschitz arc).

\begin{theo}[\cite{xiang07}, Theorems 3.1 and 3.2] Let $(X,\rho)$ be a metric space.
\begin{enumerate}
\item If $(X,\rho)$ has the fixed point property for contractions, then $X$ is Lipschitz complete.
\item If $(X,\rho)$ is locally Lipschitz connected, then $X$ has the fixed point property for contractions if and only if  it is Lipschitz complete.
\end{enumerate}\end{theo}

The above result have the following consequence (compare with Corollary \ref{c.Borw-compl} and Example \ref{ex.Borw}).
\begin{corol}[\cite{xiang07}]
  A starshaped subset of a normed space has the fixed point property for contractions if and only if  it is Lipschitz complete.
\end{corol}

This implies  that the starshaped set considered in Example \ref{ex.Borw} is Lipschitz complete, in spite of the fact that it is not closed. This furnishes a further example of a non complete starshaped set that is Lipschitz complete.

One says that the metric space $(X,\rho)$ has the \emph{strong contraction property} if every mapping $f:X\to X$ which is a contraction with respect to some metric $\overline\rho$ on $X$, uniformly equivalent to $\rho$, has a fixed point.
\begin{theo}[\cite{xiang07}, Theorems 4.1 and 4.4] Let $(X,\rho)$ be a metric space.
\begin{enumerate}
\item If $(X,\rho)$ has the strong  contraction property, then $X$ is arcwise  complete.
\item If $(X,\rho)$ is locally arcwise connected, then $X$ has the strong  contraction property if and only if  it is arcwise complete.
\end{enumerate}\end{theo}

Suzuki \cite{suzuki08} found an extension of Banach contraction principle that implies completeness.
He considered the function $\theta:[0,1)\to(1/2,1]$
\begin{equation}\label{def.Suzuki-fcs}
\theta(r)=\begin{cases}
1 &\mbox{if}\quad 0\le r\le(\sqrt 5-1)/2\\
(1-r)r^{-2}\quad &\mbox{if}\quad  (\sqrt 5-1)/2\le r\le 2^{-1/2}\\
(1+r)^{-1}\quad &\mbox{if}\quad    2^{-1/2}\le r<1\,.
\end{cases}\end{equation}
and proved the following fixed point result.
\begin{theo}\label{t1.Suzuki} Let $(X,\rho)$ be a complete metric space and $f:X\to X$.
\begin{enumerate}
\item If there exists $r\in [0,1)$ such that
\begin{equation}\label{eq1.Suzuki}
\theta(r) d(x,f(x))\le d(x,y)\;\Longrightarrow\; d(f(x),f(y))\le r d(x,y)\,,\end{equation}
for all $x,y\in X$, then $f$ has a fixed point $\overline x$ in $X$ and  $\lim_nf^n(x)=\overline x$ for every point $x\in X$.
\item Moreover, $\theta(r)$ is the best constant in \eqref{eq1.Suzuki} for which the result holds, in the sense that for every $r\in [0,1)$ there exist a complete metric space $(X,\rho)$ and a function $f:X\to X$ without fixed points and such that
    \begin{equation}\label{eq2.Suzuki}
\theta(r) d(x,f(x))< d(x,y)\;\Longrightarrow\; d(f(x),f(y))\le r d(x,y)\,,\end{equation}
for all $x,y\in X$.\end{enumerate}
 \end{theo}

 Extensions of Suzuki fixed point theorem to partial metric spaces and to partially ordered metric spaces were given by Paesano and Vetro \cite{vetro12}.

 The converse result is the following one.
 \begin{theo}[\cite{suzuki08}, Corollary 1]\label{t2.Suzuki} For a metric space $(X,\rho)$ the following are equivalent.
 \begin{enumerate}
 \item The space $(X,\rho)$ is complete.
 \item There exists $r\in(0,1)$ such that  every mapping $f:X\to X$ satisfying
 \begin{equation}\label{eq3.Suzuki}
\frac1{10000}\, d(x,f(x))\le d(x,y)\;\Longrightarrow\; d(f(x),f(y))\le r d(x,y)\,,\end{equation}
for all $x,y\in X$, has a fixed point.\end{enumerate}
    \end{theo}

It is clear that the function $\theta(r)$ given by \eqref{def.Suzuki-fcs} satisfies the equality
$\lim_{r\nearrow 1}\theta(r)=1/2$.  The critical case of functions acting on a subset $X$ of a Banach space $E$ satisfying the condition
\begin{equation}\label{def.Suzuki-nexp}
\frac12   \|x-f(x)\|\le \|x-y\|\;\Longrightarrow\;   \|f(x)-f(y)\|\le \|x-y\|\,,
\end{equation}
for all $x,y\in X$ was examined by Suzuki \cite{suzuki08b}. Condition \eqref{def.Suzuki-nexp} was called condition (C) and the functions satisfying this condition are called \emph{generalized nonexpansive}. It is clear that every nonexpansive mapping satisfies \eqref{def.Suzuki-nexp}, but there are discontinuous functions satisfying \eqref{def.Suzuki-nexp}, so that the class of generalized nonexpansive mappings is strictly larger than that of nonexpansive ones. The term generalized nonexpansive is justified by the fact that the generalized nonexpansive mappings share with nonexpansive mappings several properties concerning fixed points - in some   Banach spaces $E$ they have fixed points on every weakly compact convex subset of $E$, and for every closed bounded convex subset $X$ of $E$ and every generalized nonexpansive mapping $f$ on $X$ there exists   an almost fixed point sequence, i.e. a sequence $(x_n)$ in $X$ such that  $\|x_n-f(x_n)\|\to 0$ as $n\to \infty$, see \cite{suzuki08b}. Also a generalized nonexpansive mapping $f$ is quasi-nonexpansive, in the sense that $\|f(x)-y\|\le \|x-y\|$ for all  $x\in X$ and $y\in $\,Fix$(f)$ (the set of fixed points of $f$). It is known that every nonexpansive mapping having a fixed point is quasi-nonexpansive
(for fixed points of nonexpansive mappings and other fixed point  results see   \cite{Goebel-Kirk} and \cite{Hdb-MFP}).

For further results and extensions, see  \cite{dhomp09c}, \cite{dhomp09b}, \cite{dhomp09}, \cite{liu02}, \cite{liu03}  and \cite{suzuki11}.

 Amato \cite{amato84,amato86,amato93} proposed another approach to study the connections between fixed points and completeness in metric spaces. For a metric space $(E,d)$ he considers a pair
$(Y,\Psi)$, where $Y$ is a subset of $X$ and $\Psi$ is a class of mappings on $Y$. The pair $(Y,\Psi)$ is said to be a completion class for $E$ if $\Psi/\rho$ is a completion of $(E,d)$, where $\rho$ is a semimetric on $\Psi$ (defined in a concrete manner) and $\Psi/\rho$ is the quotient space with respect to  the equivalence relation $f\equiv g\iff \rho(f,g)=0$. Among other results, he proves that if $E$ is an infinite dimensional normed space and $K$ is a compact subset of $E$, then it is possible to take $Y=E\smallsetminus K$ and $\Psi$ the class of all compact contractions of $Y$.

We mention also the following characterization of completeness in terms of fixed points of set-valued mappings.  For a metric space  $(X,\rho)$ denote by $\mathcal P_{cl}(X)$ the family of all nonempty closed subsets of $X$.

For a mapping $F:X\to \mathcal P_{cl}(X)$ consider the following two properties:
\begin{itemize}
  \item[{\rm (J1)}]\;\; $F(F(x))\subseteq F(x)$ for every $x\in X$;
 \item[{\rm (J2)}]\;\;$\forall x\in X,\, \forall \varepsilon >0,\exists y\in F(x)$\;\;\mbox{with}\;\;
 $\diam F(y)<\varepsilon$.
\end{itemize}

For $F:X\to 2^X$ a point $\overline x\in X$ is called

\textbullet \; a \emph{fixed point} of $F$ if $\overline x\in F(\overline x)$;

\textbullet \; a \emph{stationary}  point of $F$ if $ F(\overline x)=\{\overline x\}$.

\begin{theo}[\cite{jachym11b}, Corollary  1] For any metric space  $(X,\rho)$    the following conditions are equivalent.
\begin{enumerate}
\item The space $(X,\rho)$ is complete.
\item Every  set-valued mapping $F:X\to \mathcal P_{cl}(X)$  satisfying  (J1) and (J2) has a fixed point.
\item Every  set-valued mapping $F:X\to \mathcal P_{cl}(X)$  satisfying  (J1) and (J2) has a stationary point.
    \end{enumerate}
\end{theo}

Characterizations of the completeness of a metric space in terms of the existence of fixed points for various classes of set-valued mappings acting on them were done by Jiang \cite{jiang00} and Liu \cite{liu96}.

We present the results from Jiang \cite{jiang00}. Let $(X,\rho)$ be a metric space.  For a bounded subset $Y$ of $X$ denote by $\alpha(Y)$ the Kuratowski measure of noncompactness of the set $Y$ defined by
\begin{equation}\label{def.K-meas-nonc}\begin{aligned}
 \alpha(Y):=\inf\{\varepsilon>0 : &\, Y \;\mbox{can be covered by the union of a finite family}\\ &\mbox{of subsets of  } X, \mbox{ each of diameter } \le \varepsilon\}\,.
\end{aligned}\end{equation}

For a set-valued mapping $F:X\to\mathcal P_{cl}(X)$ one considers the following conditions:
\begin{itemize}
 \item[{\rm (a)}]\;\; $F(F(x))\subseteq F(x)$ for every $x\in X$;
 \item[{\rm (b)}]  there exists a sequence $(x_n)$ in $X$ such that  $x_{n+1}\in F(x_n),\, \forall n\in\mathbb{N},$ and \\ $\lim_n\diam\left(F(x_n)\right)=0$;
\item[{\rm (c)}]  there exists a sequence $(x_n)$ in $X$ such that  $x_{n+1}\in F(x_n),\, \forall n\in\mathbb{N},$ and \\ $\lim_n\alpha\left(F(x_n)\right)=0$;
\item[{\rm (d)}]     $\lim\rho(x_n,x_{n+1})=0$ for each sequence $(x_n)$ in $X$ such that  $x_{n+1}\in F(x_n),\, \forall n\in\mathbb{N}$.
   \end{itemize}

\begin{remark}
  Condition (a) is identic to (J1) and it is easy to check that (J2) implies (b). Condition (d) is condition (iv) from Theorem \ref{t1.DHM}.
\end{remark}

   One considers also the following classes of set-valued mappings
   $F:X\to\mathcal P_{cl}(X)$:
   \begin{align*}
AB(X):=&\{ F : F \mbox{ satisfies (a) and (b)}\};\\
AC(X):=&\{ F : F \mbox{ satisfies (a) and (c)}\};\\
AD(X):=&\{ F : F \mbox{ satisfies (a) and (d)}\}\,.
   \end{align*}

   \begin{theo}[Jiang \cite{jiang00}, Theorems 3.1 and 3.2]
    For any metric space $(X,\rho)$ the following conditions are equivalent.
    \begin{enumerate}
    \item  The metric space $(X,\rho)$ is complete.
    \item Every $F$ in  $AB(X)$ has a fixed point.
    \item Every $F$ in  $AC(X)$ has a fixed point.
    \item Every $F$ in  $AD(X)$ has a fixed point.
    \item Every $F$ in  $AB(X)$ has a stationary point.
    \item Every $F$ in  $AD(X)$ has a stationary point.
 \end{enumerate}  \end{theo}

\section{Variational principles and completeness}\label{S.Ekeland}

This section is concerned with Ekeland Variational Principle (EkVP) in metric and in quasi-metric spaces and its relations to the completeness of these spaces.
\subsection{Ekeland Variational Principle}
The general form of EkVP  is the following.
\begin{theo}[Ekeland Variational Principle -- EkVP]\label{t.EkVP}
Let $(X,\rho)$ be a complete metric space and $f:X\to
\Real\cup\{+\infty\}$ a lsc bounded below function. Let $\epsic
> 0$ and $ x_0\in \dom f.$

Then given $\lambda > 0$ there exists $\,z=z_{\epsic,\lambda}\in
X\,$ such that
\bequ \label{eq2.EkVP}
\begin{aligned}
{\rm (a)}&\; \quad f(z) +\frac{\epsic}{\lambda} \rho(z,x_0)
\leq f(x_0);\\
{\rm (b)}& \;\quad \;\forall x\in X,\, x\neq z,\; \;  f(z) < f(x) +\frac{\epsic}{\lambda} \rho(z,x).
\end{aligned}
\eequ

If further, $x_0$ satisfies the condition
\bequ \label{eq1.EkVP}
f(x_0) \leq \inf f(X) + \epsic,
\eequ
then
$$
{\rm (c)}\; \quad  \rho(z,x_0)\leq \lambda.\qquad\qquad\qquad\qquad\qquad\qquad$$
\end{theo}

 The Ekeland Variational Principle is sometimes written in the following way (see, for instance, \cite{penot86} or  \cite[Lemma 3.13]{Phe93}).

 \begin{theo}[Ekeland Variational Principle-version b]\label{t.EkVPb}
Let $(X,\rho)$ be a complete metric space and $f:X\to
\Real\cup\{+\infty\}$ a lsc bounded below function. Let $\epsic
> 0$ and $ x_0\in \dom f.$

Then given $\lambda' > 0$ there exists $\,z=z_{\epsic,\lambda'}\in
X\,$ such that
\bequ \label{eq2.EkVPb}
\begin{aligned}
{\rm (a')}&\; \quad f(z) + \lambda' \rho(z,x_0) \leq f(x_0);\\
{\rm (b')}& \;\quad \;\forall x\in X,\, x\neq z,\; \;  f(z) < f(x) + \lambda' \rho(z,x).
\end{aligned}
\eequ
If further, $x_0$ satisfies the condition
\bequ \label{eq1.EkVPb}
f(x_0) \leq \inf f(X) + \epsic,
\eequ
then
$$
{\rm (c')}\; \quad   \rho(z,x_0)\leq \frac{\epsic}{\lambda'}\,.\qquad\qquad\qquad\qquad\qquad\qquad$$
\end{theo}
\begin{proof}
The equivalence between  Theorem \ref{t.EkVP} and  Theorem \ref{t.EkVPb} follows by the substitution
\bequ\label{EkVP-a-b}
\lambda'=\frac\epsic\lambda \iff \lambda=\frac\epsic{\lambda'}\eequ.
  \end{proof}

An important consequence is obtained by taking $\lambda =\sqrt
\varepsilon$ in Theorem \ref{t.EkVP}.
\begin{corol}\label{c.EkVP1}
Under the hypotheses of Theorem \ref{t.EkVP}, for every $\varepsilon
> 0$ and $x_0\in X$ with $ f(x_0) \leq \inf f(X) + \epsic$
there exists $y_\varepsilon\in X$  such that
\begin{equation} \label{eq5.EkVP1}
\begin{aligned}
{\rm (a)}&\; \; f(y_\varepsilon)
+\sqrt{\varepsilon}\rho(y_{\varepsilon},x_0) \leq f(x_0);\\
{\rm (b)}&\; \;\forall x\in X,\, x\neq y_{\varepsilon},\; \;
f(y_{\varepsilon}) < f(x)+\sqrt{\varepsilon} \rho(y_{\varepsilon},x);\\
{\rm (c)}& \;\;  \rho(y_{\varepsilon},x_0)\leq \sqrt \varepsilon.
\end{aligned}
\end{equation}
\end{corol}

Taking $\lambda =1$ in Theorem \ref{t.EkVP}, one obtains the following form of the
Ekeland Variational Principle, known as the {\it weak form of the
Ekeland Variational Principle}.

\begin{corol}[Ekeland's Variational Principle - weak form (wEkVP)]
\label{c.wEkVP}
Let $(X,\rho)$  be a complete metric space and $f:X\to
\mathbb{R}\cup\{+\infty\}$ a lsc  and bounded from below proper function.
Then for every $\varepsilon >0$ there exists an element $y_\varepsilon\in X$
such that
\begin{equation}\label{eq1.wEkVP}
f(y_\varepsilon)\leq \inf f(X) + \varepsilon,
\end{equation}
and
\begin{equation}\label{eq2.wEkVP}
f(y_\varepsilon) < f(y)+ \varepsilon \rho(y,y_\varepsilon),\quad \forall
y\in X\setminus \{y_\varepsilon\}.
\end{equation}
\end{corol}

Note that the validity of Ekeland Variational Principle (in its weak
form) implies the completeness of the metric space $X.$ This was discovered by Weston \cite{weston77} in 1977 and rediscovered by Sullivan \cite{suliv81a} in 1981 (see also the survey \cite{suliv81b}).

More exactly, the following result holds.

\begin{prop}\label{p.EkVP-compl}
Let $(X,\rho)$ be a metric space.
If for every      Lipschitz function $f:X\to[0;\infty)$ there exists a point $z\in X$ such that
    \bequ\label{eq1.EkVP-compl}
    f(z)\le f(x)+\rho(z,x)\quad\mbox{for all}\;\; x\in X\,,
    \eequ
then $X$ is complete.
\end{prop}
\begin{proof}
   To prove the completeness of $X,$ let  $(x_n)$  be a Cauchy sequence in $X.$
The inequalities $\,|\rho(x_n,x)-\rho(x_m,x)|\le\rho(x_n,x_m),\, m,n\in\Nat,\,$ show that $(\rho(x_n,x))$ is a Cauchy sequence in $\Real,$ for every $x\in X$. Consequently, we can
 define a function $f:X\to\Real$ by $f(x)=2\lim_{n\to\infty}\rho(x_n,x),$
  $ x\in X.$  The inequalities
 $|2\rho(x_n,x)-2\rho(x_n,x')|\leq2\rho(x,x'), \, n\in \Nat,$ yield
 for $n\to \infty, \, |f(x)-f(x')|\leq2\rho(x,x'), $ showing that
 $f$ is Lipschitz.  For every $\epsic>0$ there exists $n_0$ such
 that $2\rho(x_n,x_{n+k})< \epsic,$ for all $n\geq n_0$ and all $k\in
 \Nat.\,$ Letting $k\to \infty,$ one obtains $0\le f(x_n)\leq \epsic,
 \forall n\geq n_0,$  implying  $\lim_{n\to \infty}f(x_n)=0.$
  By hypothesis,  there exists $z\in X$ such that
 \bequ\label{eq1.wEkVP-compl}
 f(z)\le f(x)+ \rho(z,x),
 \eequ
 for all $x\in X$. Putting $x=x_n$ in \eqref{eq1.wEkVP-compl}
 and letting $n\to \infty,$ one obtains $f(z)\leq \frac12 f(z), $
 implying $f(z)=0, $ which is equivalent to $\lim_{n\to
 \infty}\rho(x_n,z) =0, $ i.e., $(x_n)$ converges to $z$.
 \end{proof}
 \begin{remark}\label{re.str-EkVP-compl}
   The validity of strong EkVP also implies the completeness of the underlying metric space $X$. Indeed, the fulfilment  of \eqref{eq1.str-EkVP1}.(b) implies that \eqref{eq1.EkVP-compl} holds, so that, by Proposition \ref{p.EkVP-compl}, the space $X$ is complete.

   This also  shows that the strong EkVP is equivalent to  EkVP.
   Indeed, the proof of Theorem \ref{t.str-EkVP1} shows that Theorem  \ref{t.EkVPb} implies Theorem \ref{t.str-EkVP1}. Since the validity of Theorem \ref{t.str-EkVP1} implies the completeness  of $X$ and,  in its turn, this implies the validity of EkVP,  the equivalence of these two principles follows.
 \end{remark}
Ekeland Variational Principle is equivalent to many important fixed point and geometric results
(drop property, Caristi's fixed point theorem, the flower petal theorem, etc, see \cite{penot86}).
We mention here only Caristi's fixed point theorems - for both single-valued and set-valued mappings.

  \begin{theo}[Caristi-Kirk Fixed Point Theorem]\label{t.caristi1}
Let $(X,\rho)$ be a complete metric space and $\varphi:X\to \mathbb{R}$ a
bounded from below lsc  function.  If the  mapping $f:X\to X$
satisfies the condition
\begin{equation} \label{eq1.fp-Car}
\rho(x,f(x))\leq \varphi(x)-\varphi(f(x)),\; x\in X,
\end{equation}
then $f$ has a fixed point in $X.$
\end{theo}

Another consequence of EkVP is a set-valued version of Caristi's
 fixed point theorem.
 \begin{theo}\label{t.caristi2}
 Let $(X,\rho)$ be a complete metric space, $\,\varphi:X\to\mathbb{R}
 \cup\{+\infty\}\,$ a lsc and bounded from below proper function, and
 $F:X\rightrightarrows X$  a set-valued mapping.  If the mapping
 $F$ satisfies the condition
 \begin{equation}\label{eq1.caristi2}
 \rho(x,y)\leq \varphi(x)-\varphi(y), \quad \forall x\in X,\;\;
 \forall y\in F(x),
 \end{equation}
 then $F$ has a fixed point, i.e. there exists $x_0\in X$ such that
 $\, x_0\in F(x_0).$
 \end{theo}

 It follows that the validity of Caristi's FPT also implies the completeness of the underlying metric space.

\begin{corol}\label{c.Car-compl}
Let $(X,\rho)$ be a complete metric space. If every function $f:X\to
X,$ satisfying the hypotheses of Caristi Fixed Point Theorem for
some lsc function $\varphi :X\to \mathbb{R}, $ has a fixed
point in $X,$ then the metric space $X$ is complete.
\end{corol}

\begin{remark}
   Replacing in both Theorems \ref{t.caristi1} and \ref{t.caristi2} and in Corollary \ref{c.Car-compl}  the function $\varphi$ by $\varphi-\inf\varphi(X)$, one can consider, without restricting the generality, that the function $\varphi$ is lsc and takes values in $\mathbb{R}_+\,$.
 \end{remark}

 \subsection{The strong Ekeland principle, compactness and reflexivity }

 Let $X$ be a Banach space and $f:X\to\Real\cup\{\infty\}$ a function. A point $x_0\in \dom(f)$ is called

 \textbullet \;\;a \emph{minimum point} for $f$ if $f(x_0)\le f(x)$ for all $x\in X$;

\textbullet \;\;a \emph{strict minimum point} for $f$ if $f(x_0)<f(x)$ for all $x\in X\setminus\{x_0\}$;

\textbullet\;\;a \emph{strong minimum point} for if $f(x_0)=\inf f(X)$ and every sequence $(x_n)$ in $X$ such that
$\lim_nf(x_n)=\inf_X f$ is norm-convergent to $x_0$.

A sequence $(x_n)$ satisfying $\lim_nf(x_n)=\inf f(X)$ is called a \emph{minimizing sequence} for $f$.

\begin{remark}\label{re.str-min}
  A strong minimum point is a strict minimum point, but the converse is not true.
 \end{remark}

 Indeed, if there exist $z\ne z'$ such that $f(z)=m=f(z')$, where $m=\inf f(X),$ then the sequence $x_{2k-1}=z,\, x_{2k}=z', k\in \Nat,$ satisfies $\lim_nf(x_n)=m,$ but it is not convergent. Also,  the function $f:\Real \to\Real,\, f(x)=x^2e^{-x},$ has a strict minimum at 0, $f(0)=0,\, f(n)\to 0,$ but the sequence $(n)_{n\in\Nat}$ does not converge to 0.

 Condition (b) in Theorem \ref{t.EkVP} (as well as the corresponding condition in other variants) asserts, in fact, that  $z$ is  strict minimum   point for the perturbed function $\tilde f:=f+\f{\epsic}{\lbda}\rho(z,\cdot)$. Georgiev \cite{georgiev88} proved a stronger variant of Ekeland variational principle, guaranteeing the existence of a strong minimum point $z$ for $\tilde f$. Later Turinici \cite{turin05} has shown that this strong form can be deduced from Theorem \ref{t.EkVPb}.

 \begin{theo}[Strong Ekeland Variational Principle-amended form] \label{t.str-EkVP1} Let $(X,\rho)$ be a complete metric space and $f:X\to\Real\cup\{+\infty\}$ a lsc function bounded from below on $X$. Then for every $\gamma,\delta > 0$ and $x_0\in \dom(f)$ there exists $z\in X$ such that
\bequ\label{eq1.str-EkVP1}\begin{aligned}
 & {\rm (a)}\;\; f(z)+\gamma \rho(x_0,z)<f(x_0)+\delta;\\
  &{\rm (b)}\;\; f(z)<f(x)+\lbda \rho(z,x)\quad\mbox{for all}\quad x\in X\setminus\{z\}\,;\\
  &{\rm (c)}\;\;  f(x_n)+\lbda \rho(z,x_n)\to f(z)\;\Ra\; x_n\to z,\;\;\mbox{for every  sequence } (x_n)\mbx{  in }  X.
  \end{aligned} \eequ\end{theo}

  Georgiev \cite{georgiev88} deduced from Theorem \ref{t.str-EkVP1} (called by him the amended strong Ekeland variational principle) the following result.
 \begin{theo}[Strong Ekeland Variational Principle]\label{t.str-EkVP2} Let $(X,\rho)$ and $f$ be as in Theorem \ref{t.str-EkVP1}.For $\epsic >0$ suppose that $x_0\in X$ satisfies
  \bequ\label{eq1.str-EkVP2}
  f(x_0)\le\epsic+\inf f(X).\eequ

  Then for every $\gamma,\delta_1,\delta_2 > 0$ with $\gamma\delta_1\ge\epsic,$ there exists $z\in X$ such that
\bequ\label{eq2.str-EkVP2}\begin{aligned}
 & {\rm (a)}\;\; f(z)+\gamma \rho(x_0,z)<f(x_0)+\delta_2;\\
 & {\rm (b)}\;\; z\;\;\mbox{is a strong minimum point for}\;\; \tilde f:=f+\gamma\rho(z,\cdot);\\
  & {\rm (c)}\;\;\rho(z,x_0)\le\delta_1+\delta_2.\end{aligned}
 \eequ\end{theo}
 \begin{proof}
  Taking $\delta:=\delta_2$ in Theorem \ref{t.str-EkVP1}, \eqref{eq1.str-EkVP1}.(a) yields \eqref{eq2.str-EkVP2}.(a), which, in its turn, implies
  $$
  \rho(z,x_0)\le\frac1\gamma\left[f(x_0)-f(z)\right]+\delta_2\le\frac\epsic\gamma+\delta_2\le\delta_1+\delta_2.$$
 \end{proof}
Observe that there is a discrepancy between the conditions (a$'$) in Theorem \ref{t.EkVPb} and condition (a) in Theorem \ref{t.str-EkVP1}, condition (a$'$) being stronger than (a). As was remarked by Suzuki \cite{suzuki06,suzuki10b}, a strong version of the Ekeland variational principle with condition (a$'$)
instead of (a) can be proved by imposing supplementary conditions on the underlying metric (or normed) space $X$, which are, in some sense, also necessary.

Let  $f\colon X\to(-\infty,+\infty]$ be  a proper function defined on a metric space $(X,\rho)$. For $x_0\in \dom f$ and $\lambda>0$ consider an element  $z=z_{x_0,\lambda} $ satisfying  the following conditions:
\bequ\label{eq1.str-EkVP-comp}\begin{aligned}
  &{\rm (i)}\;\; f(z)+\lbda \rho(z,x_0)\le f(x_0)\,;\\
  &{\rm (ii)}\;\; f(z)<f(x)+\lbda \rho(z,x)\quad\mbox{for all}\quad x\in X\setminus\{z\}\,;\\
  &{\rm (iii)}\;\;  f(x_n)+\lbda \rho(z,x_n)\to f(z)\;\Ra\; x_n\to z,\;\;\mbox{for every  sequence } (x_n)\mbx{  in }  X.
  \end{aligned}\eequ

  If $(X,\norm)$ is a normed space, then $\rho(x,y)$ is replaced by $\|y-x\|.$

  A metric space $(X,\rho)$ is called \emph{boundedly compact} if every bounded closed subset of $X$ is compact. It is obvious that a boundedly compact metric space is complete, and that a normed space is boundedly compact if and only if it is finite dimensional.

  \begin{prop}\label{p1.str-EkVP-comp} Let $(X,\rho)$ be a boundedly compact metric space, $f\colon X\to(-\infty,+\infty]$ a lsc bounded from below function, $x_0\in \dom f$ and $\lbda >0.$

  Then there exists a point $z\in X$ satisfying the conditions  \eqref{eq1.str-EkVP-comp}.
   \end{prop}\begin{proof}
    By Theorem \ref{t.EkVPb} there exists $z\in X$ satisfying the conditions (i) and (ii) from \eqref{eq1.str-EkVP-comp}.

   Let  now  $(x_n)$ be a sequence in $X$ such that $\lim_n [f(x_n)+\lbda \rho(z,x_n)]= f(z)$ and suppose, by contradiction, that $(x_n)$ does not converge to $z$.  Then there exist $\gamma>0$ and a subsequence $y_k=x_{n_k},\, k\in\Nat,$ of $(x_n)$  such that $\rho(y_k,z)\ge\gamma $ for all $k\in\Nat.$

    Let $k_0\in\Nat$ be such that $f(y_k)+\lbda \rho(z,y_k)\le  f(z)+1$ for all $k>k_0$. The inequalities
    \begin{align*}
      \lbda \rho(z,y_k)=& f(y_k)+\lbda \rho(z,y_k) -f(y_k)\\
             \le &  f(z)+1-\inf f(X)\,,
    \end{align*}
    valid for all $k>k_0,$ show that the sequence $(y_k)$ is bounded. Consequently, it contains a subsequence
$    (y_{k_i})$ converging to some $y\in X\setminus\{z\}.$ Taking into account the lsc of the function $f,$ one obtains
$$
f(y)+\lbda\rho(z,y)\le\liminf_i[f(y_{k_i})+\lbda \rho(z,y_{k_i})]=f(z),$$
in contradiction to \eqref{eq1.str-EkVP-comp}.(ii).
  \end{proof}

\begin{remark}\label{re.q-cv} 1.  Let $X$ be a vector space. A function $f:X\to\Real\cup\{\infty\}$ is called \emph{quasi-convex} if
 $$
 f((1-t)x+ty)\le \max\{f(x),f(y)\}\,,$$
 for all $x,y\in X$ and $t\in[0,1].$ This is equivalent to the fact that the sublevel  sets $\{x\in X:f(x)\le\alpha\}$ are convex for all $\alpha\in\Real$  (see \cite{Niculescu}).

 2. One says that a Banach space $X$ is a \emph{dual Banach space} if there exists a Banach space $Y$ such that $Y^*=X.$  Obviously, a reflexive Banach space is a dual Banach space with $X=\left(X^*\right)^*$  and, in this case, the weak topology  and the weak$^*$ topology on $X$  agree.
 \end{remark}

In the Banach space case   the following results can be proved.

\begin{prop}[\cite{suzuki06}]\label{p1.str-EkVP-refl} Let $X$ be a   Banach space,   $f\colon X\to(-\infty,+\infty]$ a bounded from below function, $x_0\in \dom f$ and $\lbda >0.$
\ben
\item[\rm 1.] If $X$ is a dual Banach space and $f$ is  $w^*$-lsc,
  then there exists a point $z\in X$ satisfying   \eqref{eq1.str-EkVP-comp}  with $x_n\xrightarrow{w^*}x$ in the condition (iii).
\item[\rm 2.]
Suppose that  the  Banach space $X$ is  reflexive. If  $f$ is  weakly  lsc,  then there exists a point $z\in X$ satisfying the conditions  \eqref{eq1.str-EkVP-comp}.
The same is true if $f$ is quasi-convex and norm-lsc.
   \een\end{prop}
   \begin{proof}  1.\; Since
   $$
    x_n\xrightarrow{\norm}x\;\Ra\;x_n\xrightarrow{w^*}x\;\Ra\;f(x)\le\liminf_n f(x_n)\,,$$
    it follows that $f$ is also lsc on $(X,\norm)$. By Theorem \ref{t.EkVPb}, there exists $z\in X$ satisfying the conditions (i) and (ii) from  \eqref{eq1.str-EkVP-comp}.

    Let $(x_n)$ be a sequence in $X$ such that
    $$
    \lim_{n\to\infty}[f(x_n)+\lbda\|x_n-z\|]=f(z)\,.$$

    The inequalities
    \begin{align*}
    \|x_n\|&\le \|z\|+\|x_n-z\|\\
    &\le \|z\|+\f{f(x_n)+\lbda\|x_n-z\|}{\lbda}-\f1{\lbda}\inf f(X)
    \end{align*}
    show that the sequence $(x_n)$ is norm-bounded. If $(x_n)$ is not $w^*$-convergent to $z$, then, by the Alaoglu-Bourbaki theorem, there exist
    a subnet $(y_i:i\in I)$ of the sequence $(x_n)$ and $y\ne z$ in $X$ such that $y_i\xrightarrow{w^*}y.$ Since  the norm $\norm$ is also $w^*$-lsc we have
    $$
    f(y)+\lbda \|y-z\|\le \liminf_i[f(y_i)+\lbda\|y_i-z\|]=f(z)\,,$$
    in contradiction to \eqref{eq1.str-EkVP-comp}.(ii).
   \smallskip

  2.\;   Since a reflexive Banach space $X$ is a dual Banach space and the  weak and weak$^*$ topologies agree on $X$,
  Theorem \ref{t.EkVPb} and the first statement of the proposition show that, in both cases, there exists $z\in X$ satisfying the conditions (i) and (ii) from \eqref{eq1.str-EkVP-comp}.

   Let $(x_n)$ be a sequence in $X$ such that
   $$\lim _n[f(x_{n})+\lbda \|z-x_{n}\|]=f(z)\,.
   $$

   I. \; \emph{The sequence $(x_n)$ converges weakly to} $z$\,.

   If $f$ is weakly lsc, then I follows from  statement 1 of the proposition.

   Suppose that $f$ is quasi-convex and norm-lsc. The function $f$ is quasi-convex if and only if the set $\{x\in X : f(x)\le \alpha\}$ is convex for every $\alpha\in \Real,$ and lsc if and only if  the set $\{x\in X : f(x)\le \alpha\}$ is closed for every $\alpha\in \Real.$ Since norm and weak closed convex sets are the same, it follows that  $f$ is also weakly lsc. Hence, by the first part of the proposition the sequence $(x_n)$ converges weakly to $z.$
    \vspace{2mm}

II.\; \emph{The sequence $(x_n)$ converges strongly to} $z$\,.

Suppose again  that the sequence   $(x_n)$ does not converge strongly to $z$. Then there exists $\epsic >0$ and a subsequence $(x_{n_k})$ of $(x_n)$ such that
$$
\|z-x_{n_k}\|>2\epsic\quad\mbox{and}\quad f(x_{n_k})+\lbda \|z-x_{n_k}\|<f(z)+\lbda\epsic\,,$$
   for all $k\in \Nat.$

   The inequalities
   $$
   f(x_{n_k})+2\lbda \epsic <f(x_{n_k})+\lbda \|z-x_{n_k}\|<f(z)+\lbda\epsic,$$
   yield
   $$
   f(x_{n_k})<f(z)-\lbda\epsic\,,$$
   for all $k\in\Nat.$

   Since, by I, $(x_{n_k})$ converges weakly to $z$ and $f$ is weakly lsc, the  above inequalities lead to the contradiction
   $$
   f(z)\le\liminf_ kf(x_{n_k}) \le f(z)-\lbda\epsic\,.$$
 \end{proof}

 As it was shown by Suzuki \cite{suzuki10b}, in some sense, the results from Propositions \ref{p1.str-EkVP-comp} and \ref{p1.str-EkVP-refl}  are the best that can be expected.

 \begin{theo}\label{t.str-EkVP-comp}
 For a   metric space $(X,\rho)$ the following are equivalent.
 \begin{enumerate}
   \item[{\rm 1.}] The metric space $X$ is boundedly compact.
   \item[{\rm 2.}] For every  proper  lsc bounded from below function $f\colon X\to(-\infty,+\infty]$,     $x_0\in \dom f$ and $\lbda >0$
there exists a point $z\in X$ satisfying the conditions   \eqref{eq1.str-EkVP-comp}.
 \item[{\rm 3.}] For every  Lipschitz function $f\colon X\to[0,+\infty)$, \,    $x_0\in \dom f$ and $\lbda >0$
there exists a point $z\in X$ satisfying the conditions  \eqref{eq1.str-EkVP-comp}.
 \end{enumerate}
 \end{theo}\begin{proof} The implication 1 $\Rightarrow$ 2 was proved in Proposition \ref{p1.str-EkVP-comp}, while the implication 2 $\Rightarrow$ 3 is obvious.

 It remains to prove  the implication  3 $\Rightarrow$ 1.

 Observe that, by Proposition \ref{p.EkVP-compl}, the metric space $X$ is complete.
 Suppose  that it is not boundedly compact. Then there exists a closed bounded subset $Y$ of $X$ which is not compact. Since $X$ is complete it follows that $Y$ is not totally bounded. Consequently there exist $\epsic >0$ and a subset $\{y_n : n\in\Nat\}$ of $Y$ such that
 $$
 \rho(y_n,y_m)>3\epsic\quad\mbox{for all}\quad m\ne n.$$

 Let
 $$
 \delta=\sup_m\rho(y_1,y_m)\quad \mbox{and}\quad B_n=B[y_n,\epsic],\; n\in \Nat\,.$$

 It follows
 $$
\delta>3\epsic,\quad 1-\frac{\rho(y_1,y_n)}{\delta}\ge 0\; \qquad\mbox{and}\qquad B_n\cap B_m=\emptyset\,,$$
 for all $n$ and all $ m\ne n$.

 For $n\in \Nat$ put
 $$
 f_n(x):=1-\frac{\rho(y_1,y_n)}{\delta}+\frac1n+\frac{\rho(x,y_n)}{\epsic},\; x\in X\,.$$

 It follows that the function $f_n$ is $\frac1\epsic$-Lipschitz on $X$.

 Defining   $f:X\to[0;\infty)$ by
 $$
 f(x):=\min\left\{1,\inf_nf_n(x)\right\},\; x\in X\,,$$
  the function $f$ will be  $\frac1\epsic$-Lipschitz on $X$ too.

 We intend to show that
 \bequ\label{eq1.t1.str-EkVP-comp}
 f(x)=\begin{cases}
   \min\{1,f_n(x)\}\quad\mbox{if}\quad x\in B_n\;\;\mbox{for some   } n\ge 2,\\
   1\qquad\qquad \qquad\qquad\mbox{otherwise}.
 \end{cases}\eequ

 Let $x\in B_n$ for some $n\ge 2.$ If $m\ne n$ then
 $$
  f_m(x)=1-\frac{\rho(y_1,y_m)}{\delta}+\frac1m+\frac{\rho(x,y_m)}{\epsic}>2+\frac1m\,,$$
  because $\, \rho(x,y_m)\ge \rho(y_n,y_m)-\rho(y_n,x)>2\epsic.$

  Consequently
  $$
  f(x)=\min\left\{1,\inf_nf_n(x)\right\}=\min\big\{1,f_n(x),\inf_{m\ne n}f_m(x)\big\}=\min\left\{1,f_n(x)\right\}\,.$$

  Let $x\in B_1.$ Then
  $$
  f_1(x)=2+\frac{\rho(x,y_1)}\epsic\ge 2\quad\mbox{and}\quad f_m(x)>2+\frac1m\quad\mbox{for}\;\; m\ge 2\,,$$
  implying $f(x)=1.$ In particular $f(y_1)=1.$

  Let now $x\in X\setminus\bigcup_k B_k.$  Then for every $n\in\Nat,$
  $$
  f_n(x)=1-\frac{\rho(y_1,y_n)}{\delta}+\frac1n+\frac{\rho(x,y_n)}{\epsic}>1+\frac 1n\,,$$
  because $\rho(x,y_n)>\epsic.$ Consequently  $f(x) =1$  and \eqref{eq1.t1.str-EkVP-comp} holds.

 Let  $x_0=y_1$ and $\lbda=\delta^{-1}. $ We show that if  $z\in X$ satisfies the conditions (i) and (ii) from \eqref{eq1.str-EkVP-comp}, then the condition (iii) is not satisfied.

 We show first that $z=y_1$.  If $f(z)<1,$ then $z\in B_n$ for some $n\ge 2,$ implying
\begin{align*}
  f_n(z)+\frac{\rho(z,y_1)}\delta=&1+\frac1n-\frac{\rho(y_1,y_n)}{\delta}+\frac{\rho(z,y_1)}\delta
  + \frac{\rho(z,y_n)}{\epsic}
  \\
  &\ge 1+\frac1n-\frac{\rho(z,y_n)}{\delta}+\frac{\rho(z,y_n)}{\epsic}\ge 1+\frac1n>1=f(y_1)\,,
\end{align*}
in contradiction to \eqref{eq1.str-EkVP-comp}.(i).

Hence  $f(z)=1$ and, by \eqref{eq1.str-EkVP-comp}.(i),
$$
1+\frac{\rho(z,y_1)}\delta\le f(y_1)=1\,,$$
which implies $\rho(z,y_1)=0$ and $z=y_1.$

Consider now the  sequence $(y_n)$ in $X$. Then, by \eqref{eq1.t1.str-EkVP-comp},
$$
\alpha_n:=f(y_n)+\frac{\rho(y_1,y_n)}\delta=\min\left\{1,1+\frac 1n-\frac{\rho(y_1,y_n)}\delta\right\} +\frac{\rho(y_1,y_n)}\delta\,,$$
for every $n\ge 2.$

If
$$
1+\frac 1n-\frac{\rho(y_1,y_n)}\delta>1\,,$$
then
$$
\alpha_n=1+\frac{\rho(y_1,y_n)}\delta < 1+\frac 1n\,,
$$
and
$$
\alpha_n= 1+\frac 1n\,$$
if
$$
1+\frac 1n-\frac{\rho(y_1,y_n)}\delta<1\,.$$

 Taking into account  \eqref{eq1.str-EkVP-comp}.(ii) and these relations, one obtains
$$
1=f(y_1)<\alpha_n\le 1+\frac 1n\,,$$
for all $n\ge 2.$ It follows $\lim_{n\to\infty}\alpha_n=1,$ but the sequence $(y_n)$ does not converge.
\end{proof}

A similar result holds in the case of normed spaces.

 \begin{theo}\label{t.str-EkVP-refl}
 For a   normed  space $(X,\norm)$ the following are equivalent.
 \begin{enumerate}
   \item[{\rm 1.}]  $X$ is a reflexive Banach space.
   \item[{\rm 2.}] For every  proper  lsc bounded from below quasi-convex function $f\colon X\to(-\infty,+\infty]$,     $x_0\in \dom f$ and $\lbda >0$
there exists a point $z\in X$ satisfying the conditions   \eqref{eq1.str-EkVP-comp}.
 \item[{\rm 3.}] For every  Lipschitz convex function $f\colon X\to[0,+\infty)$, \,    $x_0\in \dom f$ and $\lbda >0$
there exists a point $z\in X$ satisfying the conditions     \eqref{eq1.str-EkVP-comp}.
 \end{enumerate}
 \end{theo}\begin{proof}
 The implication 1 $\Rightarrow$ 2 was proved in Proposition \ref{p1.str-EkVP-refl}, while the implication 2 $\Rightarrow$ 3 is obvious.

 It remains to prove  the implication  3 $\Rightarrow$ 1.

 Observe that, by Proposition \ref{p.EkVP-compl}, the normed space $X$ is complete. By James' theorem on the characterization of reflexivity, the space $X$ will be reflexive if every $x^*\in S_{X^*}$ (the unit sphere of the dual space $X^*$) attains its norm on $S_X$ (the  unit sphere of the   space $X$).

For $x^*\in S_{X^*}$ consider the convex 1-Lipschitz function $f(x):=|x^*(x)-1|,\, x\in X.$  Then for  $x_0=0$ and $\lbda =1$ there exists $z\in X$ satisfying the conditions     \eqref{eq1.str-EkVP-comp}.

 Suppose that $z=0.$ Let $(x_n)$ be sequence in $S_X$ such that $\lim_{n\to\infty}x^*(x_n)=\|x^*\|=1.$  Then
$$
\lim_{n\to \infty}[f(x_n)+\|0-x_n\|]=\lim_{n\to\infty}[|x^*(x_n)-1|+1]=1=f(0)\,,$$
but the sequence $(x_n)$ does not converge to 0, in contradiction to \eqref{eq1.str-EkVP-comp}.(iii).

Consequently $z\ne 0.$ Taking into account \eqref{eq1.str-EkVP-comp}.(i), one obtains
$$
|x^*(z)-1|+\|z\|\le f(0)=1\;\Lra \; \|z\| \le 1-|x^*(z)-1|\le |x^*(z)|\,.$$
Since $|x^*(z)|\le \|z\|$ it follows  $|x^*(z)|= \|z\|.$ This last equality implies $\|z\|=1,$ so that  $x^*$ attains its norm at $z\in S_X.$ The reflexivity of $X$ is proved.
\end{proof}

It is known that every lsc function attains its minimum on a compact space, a property which is actually equivalent to compactness.
\begin{theo}[\cite{suzuki10b}] For  a metric space    $(X,\rho)$ the following are equivalent.
\ben
\item[\rm 1.] The metric space $X$ is compact.
\item[\rm 2.] Every lsc function $f:X\to \Real\cup\{\infty\} $ attains its minimum on $X$.
\item[\rm 3.] Every Lipschitz  function $f:X\to \Real $ attains its minimum on $X$.
\een\end{theo}\begin{proof}
  The implication 1\;$\Ra$\;2 is well known, but, for convenience, we include the simple proof. Suppose that $f $ is not identically equal to $+\infty$ and let
  $m=\inf f(X)$ and let $(x_n)$ be a sequence in $X$ such that $f(x_n)\to m.$ By the compactness of $X$ there exist a subsequence $(x_{n_k})$ of $(x_n)$ and $x\in X$ such that $x_{n_k}\to x$ as $k\to\infty.$  But then
  $$
  m\le f(x)\le\liminf_kf(x_{n_k})=m\,,$$
  implying $f(x)=m\in\Real.$

  The implication 2\;$\Ra$\;3 is obvious.

  3\;$\Ra$\;1.\;  The compactness of $X$ is equivalent to its completeness and total boundedness.\smallskip

    $X$ \emph{is complete}.

    Let $(x_n)$ be  a Cauchy sequence in $X$. Then, as in the proof of Proposition \ref{p.EkVP-compl}, the function $f(x)=\lim_n\rho(x_n,x),\, x\in X,$  is well defined, 1- Lipschitz and $\inf f(X)=0.$  By hypothesis, there exists $x\in X$ such that $0=f(x)=\lim_n\rho(x_n,x)$, showing that the sequence $(x_n)$ converges to $x.$

\smallskip
    $X$ \emph{is totally bounded}.

    If contrary, then there exist $\eps>0$ and a sequence $(x_n)$ in $X$ such that
    \bequ\label{eq1.lsc-comp}
    \rho(x_n,x_m)\ge\eps\,,
    \eequ
    for all $m\ne n$ in $\Nat.$ Let $Y=\{x_n:n\in\Nat\}\sse X$ and $g:Y\to \Real$ be defined by $g(x_n)=1/n,\, n\in     \Nat.$ Then, by \eqref{eq1.lsc-comp},
    $$
    |g(x_n)-g(x_m)|=\left|\f1n-\f1m\right|\le 1\le\f1{\eps}\, \rho(x_n,x_m)\,,$$
    for all $m,n\in\Nat,$ that is, $g$ is $1/\eps$-Lipschitz. The function $f:X\to \Real$ given by
    $$
    f(x)=\inf_n[g(x_n)+K\rho(x_n,x)],\, x\in X,$$
    where $K=1/\eps,$ is a $K$-Lipschitz extension of $g$ (see, for instance, \cite[Th. 4.1.1]{CMN}). We have
    $$0\le   \inf f(X)\le \inf_n g(x_n)=0\,,$$
    so that $\inf f(X)=0.$ If $f(x)=0$ for some $x\in X$, then, by the definition of $f$, there exists a subsequence $(x_{n_k})$ of $(x_n)$ such that
    $$
    \lim_k\left[\f1{n_k}+K\rho(x_{n_k},x)\right] =0\,,$$
    implying $\lim_kx_{n_k}=x.$ But this is impossible, because,  by \eqref{eq1.lsc-comp}, the sequence $(x_n)$ has no Cauchy subsequences.
\end{proof}

\begin{remark}  The proof given here is different from that in \cite{suzuki10b}.
\end{remark}

\section{Other principles}
In this subsection we shall present some results equivalent to Ekeland Variational Principle.
\medskip

\subsection{Takahashi minimization principle}

The first one  is that of Takahashi \cite{taka91} (see also \cite{suzuki-taka96} and \cite[T. 2.1.1]{Taka}).
\begin{theo}[Takahashi Principle]\label{t1.Taka} Let $(X,\rho)$ be a complete metric space and  $f:X\to\mathbb{R}\cup\{\infty\}$ a  lsc bounded from below proper function. If the following condition
\begin{equation}\label{eq1.Taka}
\forall x\in X,\; \inf f(X)<f(x)\;\Rightarrow\; \exists y_x\in X\sms\{x\},\; f(y_x)+\rho(x,y_x)\le f(x)\,,
\end{equation}
holds, then there exists $x_0\in X$ such that  $f(x_0)=\inf f(X)$.
\end{theo}\begin{proof} In fact, Takahashi's theorem is an immediate consequence of EkVP. Indeed, by Theorem \ref{t.EkVP}, given $\eps=\lbda=1$,  there exists $z\in X$  such that
 \bequ\label{Ek-Taka} f(z) < f(x) +  \rho(x,z)\,,\eequ
  for all $x\in X\sms\{z\}.$ If $f(z)>\inf f(X),$ then, by \eqref{eq1.Taka}, there exists $y\in X\sms\{z\}$ such that
  $$  f(y) +  \rho(z,y)\le f(z)\,,$$
  in contradiction to \eqref{Ek-Taka}, so that $f(z)=\inf f(X).$\end{proof}

The validity of Takahashi's minimum principle   also implies the completeness of the underlying metric space.
 \begin{prop}\label{p.Taka-compl}
 Let $(X,\rho)$ be  a metric space. If every bounded below Lipschitz function $f:X\to\Real$ satisfying the condition \eqref{eq1.Taka} attains its minimum on $X$, then the metric space $X$ is complete.
 \end{prop}\begin{proof} We proceed by contradiction: suppose that  the metric space $X$ is not complete and show that  there exists a Lipschitz function $f:X\to\Real$ satisfying \eqref{eq1.Taka} and which does not attain its minimum on $X$.

If $X$ is not complete, then $ X $ contains  a Cauchy sequence  $(x_n)$ which has no limit.
 We shall use the function
$$f(x)=2\lim_{n\to\infty}\rho(x,x_n),\quad x\in X\,,$$
similar to that considered  in the proof of Proposition \ref{p.EkVP-compl}. The function $f$ is 2-Lipschitz and $f(x)>0$ for all $x\in X$. Indeed, if $f(x)=0$ for some $x\in X$, then $\lim_{n\to\infty}\rho(x,x_n)=0$, i.e.  the sequence $(x_n)$ would converge to $x$, in contradiction to the hypothesis.

As it was shown in the proof of Proposition \ref{p.EkVP-compl},
\bequs
\lim_{n\to\infty}f(x_n)=0\,.\eequs

Consequently

$$
0\le\inf f(X)\le\inf\{f(x_n) : n\in\Nat\}=0\,,$$
so that $\inf f(X)=0$. Since $f(x)>0$ for all $x\in X$, it follows that $f$ does not attains its minimum value  on $X$.

It remains to show that $f$ satisfies \eqref{eq1.Taka}. If  $x\in X$, then $f(x)>0=\inf f(X)$.

As   $\lim_n\rho(x,x_n)=2^{-1}f(x)$, there exists $n_1\in\Nat$ such that
\bequ\label{eq1.Taka-compl}
\rho(x,x_n)\le \frac23\cdot f(x)\,,\eequ
for all $n\ge n_1$.

Since the sequence $(x_n)$ is Cauchy, there exists  $n_2\in\Nat$ such that
$$
\rho(x_n,x_{n+k})\le \frac16\cdot f(x)\,,$$
for all $n\ge n_2$ and all $k\in\Nat$. Keeping $n\ge n_2$ fixed and letting $k\to \infty$, one obtains
\bequ\label{eq2.Taka-compl}
f(x_n)\le \frac13\cdot f(x)\,,\eequ
for all $n\ge n_2$. Then, taking $n_0\ge\max\{n_1,n_2\}$ such that $x_{n_0}\ne x$  (such an $n_0$ exists because the sequence $(x_n)$ is not convergent), the inequalities \eqref{eq1.Taka-compl} and \eqref{eq2.Taka-compl} imply
$$
\rho(x,x_{n_0})+f(x_{n_0})\le \f 56f(x)<f(x)\,,$$
i.e. \eqref{eq1.Taka} is satisfied by $y_x=x_{n_0}$.
 \end{proof}

\subsection{Dancs-Heged{\H{u}}s-Medvegyev principle}

 Another result, also equivalent to Ekeland Variational Principle, was proved by Dancs, Heged{\H{u}}s and Medvegyev \cite{dancs-heged83}.
 \begin{theo}\label{t1.DHM} Let $(X,\rho)$ be a complete metric space and  $F:X\rightrightarrows X$ a
 set-valued function satisfying the conditions:
 \begin{itemize}
  \item[{\rm (i)}] \; $F(x)$ is closed  for every $x\in X$;
  \item[{\rm (ii)}]\; $x\in F(x)$   for every $x\in X$;
  \item[{\rm (iii)}]\; $x_2\in F(x_1)\; \Rightarrow\; F(x_2)\subseteq F(x_1)$\, for all $x_1,x_2\in X$;
  \item[{\rm (iv)}] $\lim_n\rho(x_n,x_{n+1})=0$ for every sequence $(x_n)$ in $X$ such that  $x_{n+1}\in F(x_n),\, \forall n\in\mathbb{N}$.
 \end{itemize}

 Then there exists $x_0\in X$ such that  $F(x_0)=\{x_0\}$. Moreover, for every $\overline x\in X$, there exists such a point in $F(\overline x)$.
    \end{theo}

    This result admits an equivalent formulation in terms of an order on $X$.
    \begin{theo}\label{t2.DHM} Let $(X,\rho)$ be a complete metric space and $\preceq$ a closed partial order on $X$.  If $\lim_n\rho(x_n,x_{n+1})=0$ for every increasing sequence $x_1\preceq x_2\preceq \dots$ in $X$, then there is a maximal element in $X$.  In fact, for every $  x\in X$ the set $\{y\in X :   x\preceq y\}$ contains  a maximal element of $(X,\preceq)$.\end{theo}
    \begin{proof} The equivalence is proved by considering the sets $\Phi(x)=\{y\in X :   x\preceq y\}$.
    \end{proof}

     An order $\preceq$ on a metric space is said to be \emph{closed} if $x_n\preceq y_n$, for all $n\in\mathbb{N},$
implies $\lim_nx_n\preceq\lim_n y_n,$ provided both limits exist. This is equivalent to the fact that the graph of $\preceq$, $\graph(\preceq):=\{(x,y)\in X\times X : x\preceq y\}$ is closed in $X\times X$ with respect to the product topology.

\begin{remark}
  If $F:X\rightrightarrows X$ is a set-valued mapping, then for every $x_0\in X$, a sequence $(x_n)$
  satisfying $x_n\in F(x_{n-1}),\, n\in\mathbb{N}$, is called a generalized Picard sequence. For the properties of set-valued Picard operators, defined in terms of the convergence of generalized Picard sequences, see the surveys  \cite{petrusel04,petrusel12}.
\end{remark}

\begin{remark}
   As it is remarked in  \cite[Th. 3.3]{dancs-heged83}, the converse completeness property  holds in this case too, i.e. the validity of Theorem \ref{t1.DHM} for every set-valued mapping $F$ on an arbitrary metric space $X$ implies the completeness of $X$.
\end{remark}

\subsection{Arutyunov's principle}

 A   result similar   to  Takahashi's  minimization principle  (Theorem \ref{t1.Taka}), under a slightly relaxed condition on the function $f$, was found by Arutyunov and  Gel'man \cite{arut-gel09} and Arutyunov \cite{arut15}.  For  recent further developments see \cite{arut-zhuk19a} and \cite{arut-zhuk19b}.

Let $(X,\rho)$ be a metric space and $f:X\to[0,+\infty]$ a function.  One considers the following condition on $f$:
\bequ\label{eq1.Arut}
   x_n\to \bar x\;\mbx{ and }\; f(x_n)\to 0\;\Ra\;  f(\bar x)=0\,,
\eequ
for any sequence $(x_n)$  in $ X $ and any $\bar x\in X.$
\begin{remark} It is obvious that any  lsc or with closed graph  bounded below proper function $f:X\to [0,+\infty]$  satisfies \eqref{eq1.Arut}.
  \end{remark}

  The case when $f$ has closed graph is clear. If $f$ is lsc, then for any sequence $(x_n)$ as in \eqref{eq1.Arut},
  $$
 0\le  f(\bar x)\le\liminf_{n\to\infty} f(x_n) =0\,.$$

\begin{theo}[\cite{arut-gel09}]\label{t1.Arut-Gel} Let $(X,\rho)$ be a complete metric space and  $f:X\to[0,+\infty]$ a function satisfying the condition \eqref{eq1.Arut}. Suppose that  there exist $k_1>0$ and $0<k_2<1$ such that for every $x\in X$ with $f(x)>0$ there exists $x'\in X\sms\{x\}$  satisfying
\bequ\label{eq2.Arut}
{\rm (i) }\quad \rho(x,x')\le k_1f(x)\quad\mbx{ and }\quad {\rm (ii) } \quad f(x')\le k_2 f(x)\,.
\eequ

Then for every $x_0\in \domain f$ there exists $\bar x\in X$ such that
\bequ\label{eq3.Arut}
{\rm (a) }\quad  f(\bar x)=0=\inf f(X)\quad\mbx{ and }\quad {\rm (b) } \quad \rho(x_0,\bar x)\le \f{k_1}{1-k_2} f(x_0)\,.
\eequ\end{theo}

We shall give a proof of an extension of this theorem  given by Arutyunov \cite{arut15}.

\begin{theo}[\cite{arut15}]\label{t2.Arut} Let $(X,\rho)$ be a complete metric space and  $f:X\to[0,+\infty]$ a function satisfying   \eqref{eq1.Arut}.  Suppose that there exist a usc function $\vphi:\Real_+\to\Real_+$ such that
 \bequ\label{eq4.Arut} \vphi(t)<t\;\mbx{ for every }\; t>0\,,\eequ
 and $\gamma>0$ such that   for every $x\in X$ with $f(x)>0$ there exists $x'\in X\sms\{x\}$  satisfying
\bequ\label{eq5.Arut}
{\rm (i) }\quad f(x')+\gamma \rho(x,x')\le f(x)\quad\mbx{ and }\quad {\rm (ii) } \quad f(x')\le \vphi(f(x))\,.
\eequ

Then for every $x_0\in \domain f$ there exists $\bar x\in X$ such that
\bequ\label{eq6.Arut}
{\rm (a) }\quad  f(\bar x)=0=\inf f(X)\quad\mbx{ and }\quad {\rm (b) } \quad \rho(x_0,\bar x)\le \f1{\gamma}f(x_0)\,.
\eequ\end{theo}
\begin{remark}
  Replacing the metric $\rho$ with the equivalent one $\tilde \rho=\gamma\rho,$ the condition \eqref{eq5.Arut}.(i) becomes Takahashi's condition \eqref{eq1.Taka}.
  Arutyunov \cite{arut15} calls it  a\emph{ Caristi type condition}.
\end{remark}

Let us show first that Theorem \ref{t2.Arut}  implies Theorem \ref{t1.Arut-Gel}.  Let  $x\in X$. Then  there exists $x'\ne x$ in $X$ satisfying \eqref{eq2.Arut}.  But \eqref{eq2.Arut}.(ii) means that $f(x')\le \vphi(f(x))$, where   $\vphi(t)=k_2t,\;t\in\Real_+.$ Also, multiplying the first inequality in \eqref{eq2.Arut} by $1-k_2$, the second by $k_1$, adding them and dividing the result by $k_1$, one obtains
$$
f(x')+\f{1-k_2}{k_1}\rho(x,x')\le f(x)\,,$$
i.e., \eqref{eq5.Arut}.(i) holds with $\gamma =(1-k_2)/k_1\,.$

\begin{proof}[Proof of Theorem \ref{t2.Arut}] It is obvious that $\inf f(X)=0.$

Indeed, $f(x)\ge 0$ for all $x\in X$ and, by \eqref{eq5.Arut}.(i), for every $x\in X$ with $f(x)>0$ there exists $x'\in X$ with $f(x')<f(x).$

For  $x_0\in\domain f$ consider a sequence $(x_k)_{k=0}^\infty$ satisfying
\bequ\label{eq7.Arut}
{\rm (i) }\quad f(x_{k+1})+\gamma \rho(x_{k+1},x_k)\le f(x_k)\quad\mbx{ and }\quad {\rm (ii) } \quad f(x_{k+1})\le \vphi(f(x_k))\,,
 \eequ
for all $k=0,1,\dots . $ This is possible  by \eqref{eq5.Arut}. Indeed, if $x_0,\dots,x_n$ satisfy \eqref{eq7.Arut} for $k=0,\dots, n-1,$ then, by  \eqref{eq5.Arut}, there exists $x_{n+1}\ne x_n$ such that
\bequs
{\rm (i) }\quad f(x_{n+1})+\gamma \rho(x_{n+1},x_n)\le f(x_n)\quad\mbx{ and }\quad {\rm (ii) } \quad f(x_{n+1})\le \vphi(f(x_n))\,.
 \eequs

  By \eqref{eq7.Arut}.(i), $f(x_{k+1})<f(x_k),\, k\in\Nat_0,$ so there exists $\alpha:=\lim_{k\to\infty}f(x_k)\ge 0.$
  By the same inequality,
  \bequ\label{eq8.Arut}\begin{aligned}
    \gamma\rho(x_n,x_{n+k})&\le \gamma\rho(x_n,x_{n+1})+\dots+\gamma\rho(x_{n+k-1},x_{n+k})\\
    &\le f(x_n)-f(x_{n+k})\,,
  \end{aligned}\eequ
  for all $n\in\Nat_0$ and $k\in\Nat.$

  By the convergence of the sequence $ (f(x_n))$  this implies that the sequence $(x_n)$ is Cauchy, so, by the completeness of the metric space $X$, it converges to some $\bar x\in X.$  By \eqref{eq7.Arut}.(ii) and the usc of the  function $\vphi$,
  $$
  0\le \alpha=\lim_{k\to\infty}f(x_{k+1})\le\limsup_{k\to\infty}\vphi(f(x_k))\le\vphi(\alpha)  \,.$$

  By \eqref{eq4.Arut}, this implies $\alpha =0$, so that, by\eqref{eq1.Arut}, $f(\bar x)=0=\inf f(X).$

  Let us prove now \eqref{eq6.Arut}.(b). By \eqref{eq8.Arut} for $n=0$
  $$
  \gamma\rho(x_0,x_k)\le f(x_0)-f(x_k)\,,$$
  for all $k\in\Nat.$

  Letting $k\to \infty,$ one obtains  $\gamma\rho(x_0,\bar x)\le f(x_0).$
\end{proof}
\begin{remark} If the function $f$ is lsc and satisfies the condition \eqref{eq5.Arut}.(i), then the conclusions of Theorem \ref{t2.Arut} can be obtained directly from Ekeland's Variational Principle (Theorem \ref{t.EkVP}). As it is noticed in \cite{arut15}, this remark is due to B. D. Gel'man.
\end{remark}

 As above, we have $\inf f(X)=0$, so that   there is nothing to prove   if $f(x_0)=0$.

Suppose  $f(x_0)>0$ and put $\eps:=f(x_0)$ and $\lbda :=\eps/\gamma.$  Then $\eps/\lbda=\gamma, $ so that, by Theorem \ref{t.EkVP}, there exists $\bar x\in X$ such that
\bequ\label{eq9.Arut}\begin{aligned}
{\rm (a)}& \;\quad \;\forall x\in X,\, x\neq \bar x,\; \;  f(\bar x) < f(x)+\gamma\rho(\bar x,x);\\
{\rm (b)}&\; \quad  \rho(\bar x,x_0)\leq\lbda=\f{\eps}{\gamma}=\f1{\gamma}f(x_0)\,.
\end{aligned}\eequ
Since \eqref{eq9.Arut}.(b) agrees with \eqref{eq6.Arut}.(b), we have only to show that $f(\bar x)= 0$. If $f(\bar x)>0,$ then, by \eqref{eq5.Arut}.(i), there exists $x'\ne \bar x$ in $X$ such that
$$ f(x')+\gamma \rho(\bar x,x')\le f(\bar x)\,,
$$
in contradiction to \eqref{eq9.Arut}.(a).

\subsection{Weak sharp minima and completeness}  Let $(X,\rho)$ be a metric  space and  $f:X\to\Real\cup\{\infty\}$  bounded below proper function with $$\mu:=\inf f(X)>-\infty.$$

For $\alpha\in\Real$  let
\bequs
S_\alpha(f)=\{x\in X:f(x)\le \alpha\}\,.
\eequs
It follows that
\bequs
S_\mu(f)=\{x\in X:f(x)=\mu\}\,,
\eequs
and
\bequs
S_\alpha(f)\sse S_\beta(f)\quad\mbx{for }\; \alpha\le\beta\,.
\eequs

One says that $f$ has \emph{global weak sharp minima with constant} $\lbda>0$ if $S_\mu\ne\ety$ and
\bequ\label{def.w-sh-min}
f(x)-\mu\ge\lbda d(x,S_\mu)\quad\mbx{for all }\; x\in X,
\eequ
where, for $\ety\ne Y\sse X$ and $x\in X$,
$$
d(x,Y):=\inf\{\rho(x,y): y\in Y\}\,$$
denotes  the \emph{distance} from  $x$ to $Y$.

This is an important notion with  many applications in optimization theory, see, for instance, \cite{burke1,burke2,burke3}.
\begin{remark}
  Adopting the convention $\inf \ety=+\infty,$ the validity of \eqref{def.w-sh-min} implies $S_\mu(f)\ne \ety$, i.e., $f$ attains its minimum on $X$.
\end{remark}

The implication 1\;$\Ra$\;2 from the following theorem was proved by Ng \cite{ng03} supposing that $X$ is a Banach space. The completeness result, i.e. the implication 3\;$\Ra$\;1, was  obtained by Huang \cite{huang05}.

\begin{theo} Let $ (X,\rho)$ be a metric space. The following are equivalent.
\ben
\item[\rm 1.] The metric space $X$ is complete.
\item[\rm 2.]
Every bounded below  lsc proper function $f:X\to \Real\cup\{\infty\}$ with $\mu=\inf f(X)$ for which there exist
a sequence $(\mu_n)$ in $\Real$ and $\lbda>0$ satisfying the conditions
\bequ\label{eq0.w-sh-min}\begin{aligned}
{\rm (i)}\quad &\mu<\mu_{n+1}<\mu_n \; \mbx{for all }\; n\in\Nat;\\
{\rm (ii)}\quad &\lim_{n\to\infty}\mu_n=\mu;\\
{\rm (iii)}\quad& \lim_{n\to\infty}\lbda d(x,S_{\mu_n})\le f(x)-\mu\;\mbx{ for all }x\in X\sms S_\mu,
\end{aligned}\eequ
  has global weak sharp minima with constant $\lbda>0.$
 \item[\rm 3.]  Every Lipschitz function $f:X\to \Real$ satisfying \eqref{eq0.w-sh-min} with $\mu =0$
 attains its minimum on $X$.
\een\end{theo}
\begin{proof}  1\;$\Ra$\; 2.\;
 Let $\lbda>0$ and  $(\mu_n)$  a sequence  satisfying the conditions  \eqref{eq0.w-sh-min}.

  We show that $S_\mu\ne\ety\,$ and that, for every $0<\gamma<\lbda$,
\bequ\label{eq1.w-sh-gamma}
\gamma d(x,S_\mu)\le f(x)-\mu\;\;\mbx{for all }\; x\in X,\eequ
which will imply \eqref{def.w-sh-min}.

  Let $0<\gamma<\lbda$ and  $x\in X$. By EkVP (Theorem \ref{t.EkVPb}) there exists $z\in X$ such that
 \bequ\label{eq1.w-sh-min}\begin{aligned}
   {\rm (i)}\quad &f(z)+\gamma \rho(x,z)\le f(x);\\
   {\rm (ii)}\quad &f(z)<f(x')+\gamma \rho(x',z)\;\:\mbx{ for all }\; x'\in X\sms\{z\}.
 \end{aligned}\eequ

 The set  $S_{\mu_n}$ is obviously nonempty and  closed (because $f$ is lsc). By \eqref{eq0.w-sh-min}.(i),
 \bequ\label{eq2.w-sh-min}\begin{aligned}
   {\rm (i)}\quad& S_{\mu_{n+1}}\sse  S_{\mu_{n}}\quad\mbx{and}\\
   {\rm (ii)}\quad& 0\le d(x',S_{\mu_{n}})\le  d(x',S_{\mu_{n+1}})\quad\mbx{for all }\; x'\in X,
 \end{aligned}\eequ
 for all $n\in\Nat.$

 If $f(z)=\mu,\, $ i.e. $z\in S_\mu$, then, by \eqref{eq1.w-sh-min}.(i),

 \bequ\label{eq2.w=sh-gamma}
 \gamma d(x,S_\mu)\le\gamma\rho(x,z)\le f(x)-f(z)=f(x)-\mu,\eequ
 i.e. \eqref{eq1.w-sh-gamma} holds.

 If $\lim_{n\to\infty}d(z,S_{\mu_{n}})=0,$ then $d(z,S_{\mu_{n}})=0$ for all $n$  (by \eqref{eq2.w-sh-min}.(ii)). Since $S_{\mu_{n}}$ is closed this implies  $z\in  S_{\mu_{n}}$, that is,
 $$
 \mu\le f(z)\le\mu_n \quad\mbx{for all }\; n\in\Nat,$$
and so, by \eqref{eq0.w-sh-min}.(ii), $f(z)=\mu$  and \eqref{eq1.w-sh-gamma} holds.

 If $f(z)>\mu$, then

  $$\lim_{n\to\infty}d(z,S_{\mu_n})>0\,,$$ and so

  $$
  d(z,S_{\mu_n})>0\iff z\notin S_{\mu_n}
  $$
  for all sufficiently large $n$  (i.e., for all $n\ge n_0$ for some  $n_0\in\Nat$).

 By \eqref{eq1.w-sh-min}.(ii),
 \begin{align*}
   f(z)&<f(y)+\gamma \rho(z,y)\\
   &\le\mu_n+\gamma \rho(z,y)\,,
   \end{align*}
  for all $y\in S_{\mu_n},$ which implies
 \bequ\label{equ.w-sh-min} f(z)-\mu_n\le\gamma d(z,S_{\mu_n})\,,\eequ
   for all sufficiently large $n$.

   But then, by \eqref{eq0.w-sh-min}.(iii),
   \begin{align*}
 f(z)-\mu&=\lim_{n\to\infty}(f(z)-\mu_n)
 \le\gamma \lim_{n\to\infty}d(z,S_{\mu_n})\\
 &<\lbda\lim_{n\to\infty}d(z,S_{\mu_n})
 \le f(z)-\mu\,,
    \end{align*}
    a contradiction.

The implication 2\;$\Ra$\;3 is obvious.

  3\;$\Ra$\;1.\;  Let $(x_n)$ be a Cauchy sequence in $X$. As in the proof of Proposition \ref{p.EkVP-compl}, the function $f:X\to \Real$
 given by $f(x)=\lim_{n\to\infty}\rho(x,x_n),\, x\in X,$ is well-defined, 1-Lipschitz, nenegative and $\lim_{n\to\infty}f(x_n)=0,$   so that
 $\inf f(X)=0.$

 Let us show that $f$ satisfies the conditions \eqref{eq0.w-sh-min} with $\mu =0.$  Put $\lbda_n=1/n,\, n\in\Nat,$  so that
 $$
 S_{1/n}(f)=\{x\in X:f(x)\le 1/n\},\quad n\in\Nat\,.$$

 The sets $S_{1/n}(f)$ have the following properties:
 \bequ\label{equ3.w-sh-min}\begin{aligned}
 {\rm (i)}\quad &S_{1/n}(f)\ne\ety ; \\ {\rm (ii)}\quad &S_{1/(n+1)}(f)\sse S_{1/n}(f);\\
  {\rm (iii)}\quad &\diam(S_{1/n}(f))\le 4/n\,,
 \end{aligned}\eequ
 for all $n\in \Nat.$

It is obvious that   $S_{1/n}(f)\ne\ety$ (because $\mu=\inf f(X)=0$)  and that the inclusion from     \eqref{equ3.w-sh-min}.(ii) holds.

  For $x,y\in  S_{1/n}(f)$ let $k\in\Nat$ be such that
\begin{align*}
\rho(x,x_k)&\le f(x)+\f1n\le\f2n\;\mbx{ and}\\
\rho(y,x_k)&\le f(y)+\f1n\le\f2n\,.
\end{align*}

Then,
$$
\rho(x,y)\le \rho(x,x_k)+\rho(x_k,y)\le\f4n\,,$$
showing that \eqref{equ3.w-sh-min}.(iii) holds too.\smallskip

\emph{The condition} \eqref{eq0.w-sh-min}.(iii) \emph{with} $\mu =0.$

For $x\in X$  and  $n\in\Nat$  let $k\in\Nat$ be such that
\begin{align*}
  {\rm (a)}\quad &\rho(x,x_k)\le f(x)+\f1n;\\
  {\rm (b)}\quad &\rho(x_k,x_{k+i})\le  \f1n\;\mbx{ for all }\; i\in\Nat\,.
\end{align*}

This is possible by the definition of the function $f$ and the fact that the sequence $(x_n)$ is Cauchy.
Letting $i\to\infty$ in (b) one obtains $f(x_k)\le 1/n,$ that is, $x_k\in S_{1/n}(f).$ But then, by (a),
$$
d(x, S_{1/n}(f))\le \rho(x,x_k)\le f(x)+\f1n\,,$$
for all $n\in\Nat,$ so that
$$
\lim_{n\to\infty}d(x, S_{1/n}(f))\le f(x)\,.$$

By hypothesis $S_0(f)\ne \ety$. Since    $S_0(f)\sse S_{1/n}(f)$ for all $n\in \Nat,$ the condition \eqref{equ3.w-sh-min}.(iii) implies $\diam (S_0(f))=0,$ and so
$S_0(f)=\{z\}$ for some $z\in X$.

 This implies
$$
0=f(z)=\lim_{n\to\infty}\rho(z,x_n)\,,$$
that is, $(x_n)$ converges to $z$, proving the completeness of $X$.
\end{proof}

\begin{remark}
   The proof of 1\;$\Ra$\;2 given in \cite{ng03} contains a flaw. Namely,  one affirms  that
   $$  f(z)-\mu_n>\gamma d(z,S_{\mu_n})$$
   for sufficiently large  $n$, which,  by \eqref{equ.w-sh-min},  is impossible.  The proof given here follows the ideas from  \cite{ng03}.
\end{remark}

\section{EkVP in quasi-metric spaces}
This section  is concerned with  Ekeland Variational Principle and Caristi's fixed point theorem in the context of quasi-metric spaces.\\

\subsection{Quasi-metric spaces}

We shall briefly present the fundamental properties of quasi-metric spaces. Details and references can be found in the book \cite{Cobz}.
\begin{defi}\label{def.qm}
  A {\it quasi-semimetric} on an arbitrary set $X$ is a mapping $\rho: X\times X\to
[0;\infty)$ satisfying the following conditions:
\begin{align*}
\mbox{(QM1)}&\qquad \rho(x,y)\geq 0, \quad and  \quad \rho(x,x)=0;\\
\mbox{(QM2)}&\qquad \rho(x,z)\leq\rho(x,y)+\rho(y,z),
\end{align*}
for all $x,y,z\in X.$ If, further,
$$
\mbox{(QM3)}\qquad \rho(x,y)=\rho(y,x)=0\Rightarrow x=y,
$$
for all $x,y\in X,$ then $\rho$ is called a {\it quasi-metric}. The pair $(X,\rho)$ is called a {\it
quasi-semimetric space}, respectively  a {\it quasi-metric space}. The conjugate of the quasi-semimetric
$\rho$ is the quasi-semimetric $\bar \rho(x,y)=\rho(y,x),\, x,y\in X.$ The mapping $
\rho^s(x,y)=\max\{\rho(x,y),\bar \rho(x,y)\},\,$ $ x,y\in X,$ is a semimetric on $X$ which is a metric if and
only if $\rho$ is a quasi-metric.\end{defi}

If $(X,\rho)$ is a quasi-semimmetric space, then for $x\in X$ and $r>0$ we define the balls in $X$ by the formulae
\begin{align*}
B_\rho(x,r)=&\{y\in X : \rho(x,y)<r\} \; \mbox{-\; the open ball, and }\\
B_\rho[x,r]=&\{y\in X : \rho(x,y)\leq r\} \; \mbox{-\; the closed ball. }
\end{align*}

 The topology $\tau_\rho$ of a quasi-semimetric  space $(X,\rho)$ can be defined starting from the family
$\mathcal V_\rho(x)$ of neighborhoods  of an arbitrary  point $x\in X$:
\begin{equation}
\begin{aligned}
V\in \mathcal V_\rho(x)\;&\iff \; \exists r>0\;\mbox{such that}\; B_\rho(x,r)\subseteq V\\
                             &\iff \; \exists r'>0\;\mbox{such that}\; B_\rho[x,r']\subseteq V.
\end{aligned}
\end{equation}

The convergence of a sequence $(x_n)$ to $x$ with respect to $\tau_\rho,$ called $\rho$-convergence and
denoted by
$x_n\xrightarrow{\rho}x,$ can be characterized in the following way
\begin{equation}\label{char-rho-conv1}
         x_n\xrightarrow{\rho}x\;\iff\; \rho(x,x_n)\to 0.
\end{equation}

Also
\begin{equation}\label{char-rho-conv2}
         x_n\xrightarrow{\bar\rho}x\;\iff\;\bar\rho(x,x_n)\to 0\; \iff\; \rho(x_n,x)\to 0.
\end{equation}

As a space equipped with two topologies, $\tau_\rho$ and   $\tau_{\bar\rho}$, a quasi-metric space can be viewed as a bitopological space in the sense of Kelly \cite{kelly63}. The problem of quasi-metrizability of topologies is discussed in \cite{kopper93}.

An important example of quasi-metric space is the following.
\begin{example}\label{ex.R-qm} On $X=\mathbb{R}$ let $q(x,y)=(y-x)^+$, where $\alpha^+$ stands for the positive part of a real number $\alpha$. Then $\bar q(x,y)=(x=y)^+$ and $q^s(x,y)=|y-x|$. The balls are given by
$$B_q(x,r)=(-\infty,x+r)\quad\mbox{ and }\quad B_{\bar q}(x,r)=(x-r,\infty)\,.$$
  \end{example}

The following   topological properties are true for  quasi-semimetric spaces.
    \begin{prop}[see \cite{Cobz}]\label{p.top-qsm1}
   If $(X,\rho)$ is a quasi-semimetric space, then
   \begin{enumerate}
   \item\; The ball $B_\rho(x,r)$ is $\tau_\rho$-open and  the ball $B_\rho[x,r]$ is
       $\tau_{\bar\rho}$-closed. The ball    $B_\rho[x,r]$ need not be $\tau_\rho$-closed.
     \item
   If $\rho $ is a quasi-metric, then the topology $\tau_\rho$ is $T_0,$ but not necessarily $T_1$ (and so
   nor $T_2$ as in the case of metric spaces).  \\ The topology $\tau_\rho$ is $T_1$ if and only if
   $\rho(x,y)>0$ whenever $x\neq y.$
      \item \; For every fixed $x\in X,$ the mapping $\rho(x,\cdot):X\to (\mathbb{R},|\cdot|)$ is
   $\tau_\rho$-usc and $\tau_{\bar \rho}$-lsc. \\
   For every fixed $y\in X,$ the mapping $\rho(\cdot,y):X\to (\mathbb{R},|\cdot|)$ is $\tau_\rho$-lsc and
   $\tau_{\bar \rho}$-usc.
   \item   If  the mapping $\rho(x,\cdot):X\to (\mathbb{R},|\cdot|)$ is $\tau_\rho$-continuous for
   every $x\in X,$ then the topology $\tau_\rho$ is regular. \\
       If $\rho(x,\cdot):X\to (\mathbb{R},|\cdot|)$ is $\tau_{\bar\rho}$-continuous for every $x\in X,$ then the
       topology $\tau_{\bar\rho}$ is semi-metrizable.
  \end{enumerate}
     \end{prop}

     \subsection{Completeness in quasi-metric spaces}

   The lack of symmetry in the definition of quasi-metric and
   quasi-uniform spaces causes a lot of troubles, mainly concerning
   completeness, compactness and total boundedness in such spaces.
   There are a lot of completeness notions in quasi-metric and in
   quasi-uniform spaces, all agreeing with the usual notion of
   completeness in the case of metric or uniform spaces,  each of
   them having its advantages and weaknesses.

   As in what follows we shall work only with one of these  notions, we shall present only it,
   referring to   \cite{Cobz}) for    other notions of Cauchy  sequence and for   their properties.

   A sequence $(x_n)$ in $(X,\rho)$ is called

   (a) \;  {\it left $\rho$-$K$-Cauchy} if  for every $\varepsilon >0$
   there exists $n_\varepsilon\in \mathbb{N}$ such that
\begin{equation}\label{def.l-Cauchy}\begin{aligned}
 &\forall n,m, \;\;\mbox{with}\;\; n_\varepsilon\leq n < m ,\quad\rho(x_n,x_m)<\varepsilon\\
 \iff &\forall n \geq n_\varepsilon,\; \forall k\in\mathbb{N},\quad \rho(x_n,x_{n+k})<\varepsilon.
 \end{aligned}\end{equation}

Similarly, a sequence $(x_n)$ in $(X,\rho)$ is called

(a$'$) \;  {\it right $\rho$-$K$-Cauchy} if  for every $\varepsilon >0$
   there exists $n_\varepsilon\in \mathbb{N}$ such that
\begin{equation}\label{def.r-Cauchy}\begin{aligned}
 &\forall n,m, \;\;\mbox{with}\;\; n_\varepsilon\leq n < m ,\quad\rho(x_m,x_n)<\varepsilon\\
 \iff &\forall n \geq n_\varepsilon,\; \forall k\in\mathbb{N},\quad \rho(x_{n+k},x_n)<\varepsilon.
 \end{aligned}\end{equation}

\begin{remarks}\label{re.compl-qm}Let $(X,\rho)$ be a quasi-semimetric space.
\begin{itemize}
\item[{\rm 1.}]
  Obviously, a sequence is left $\rho$-$K$-Cauchy
   if and only if it is right $\bar\rho$-$K$-Cauchy.
\item[{\rm 2.}] Let $(x_n)$ be a left    $\rho$-$K$-Cauchy sequence. If $(x_n)$  contains a subsequence which is $\tau(\rho)$ ($\tau(\bar\rho$))-convergent to some $x\in X$, then the sequence $(x_n)$ is $\tau(\rho)$ (resp. $\tau(\bar\rho$))-convergent to $x$ (\cite[P. 1.2.4]{Cobz}).
\item[{\rm 3.}]   If a sequence   $(x_n)$ in $X$ satisfies $\sum_{n=1}^\infty\rho(x_{n},x_{n+1})<\infty\;$ ($\sum_{n=1}^\infty\rho(x_{n+1},x_{n})<\infty$), then it is left (right)-$\rho$-$K$-Cauchy.
\item[{\rm 4.}] There are examples showing that a $\rho$-convergent sequence need not  be
left $\rho$-$K$-Cauchy, showing that in the asymmetric case the situation is far more complicated than in the
symmetric one (see   \cite[Section 1.2]{Cobz}).
\item[{\rm 5.}]
  If each convergent sequence in a regular quasi-metric space $(X,\rho)$
   admits a left $K$-Cauchy subsequence, then $X$ is metrizable
   (\cite[P. 1.2.1]{Cobz}).
   \end{itemize} \end{remarks}

A quasi-metric space $(X,\rho)$ is called
 {\it left $\rho$-$K$-complete} if every left $\rho$-$K$-Cauchy sequence  is
$\rho$-convergent, with the corresponding definition of the {\it right $\rho$-$K$-completeness}. The quasi-metric space $(X,\rho)$ is called \emph{left (right) Smyth complete} if every left (right) $\rho$-$K$-Cauchy sequence  is $\rho^s$-convergent and \emph{bicomplete} if the associated metric space $(X,\rho^s)$ is complete.

\begin{remark}
   In spite of the obvious fact that left $\rho$-$K$-Cauchy is equivalent
   to right $\bar \rho$-$K$-Cauchy, left $\rho$-$K$- and right
   $\bar\rho$-$K$-completeness do not agree, due to the fact that right
   $\bar \rho$-completeness means that every left $\rho$-Cauchy
   sequence converges in $(X,\bar \rho),$ while left
   $\rho$-completeness means the convergence of such sequences in the
   space $(X,\rho).$

   Also, it is easy to check that Smyth completeness (left or right) of a quasi-metric space $(X,\rho)$ implies the completeness of the associated metric space $(X,\rho^s)$ (i.e. the bicompleteness of the quasi-metric space $(X,\rho)$).
\end{remark}
\begin{example}\label{ex.R-compl-qm} The spaces $(\mathbb{R},q)$ and $(\mathbb{R},\bar q)$ from Example \ref{ex.R-qm} are not right $K$-complete. The sequence $x_n=n,\,n\in \mathbb{N},$ is right $q$-$K$-Cauchy and not  convergent in $(\mathbb{R},q)$ and the sequence $y_n=-n,\,n\in\mathbb{N},$ is right $\bar q$-$K$-Cauchy and not  convergent in $(\mathbb{R},\bar q)$.\end{example}

Indeed, $q(x_{n+k},x_n)=(n-n-k)^+=0$ for all $n,k\in\mathbb{N}$. For $x\in \mathbb{R}$ let $n_x\in\mathbb{N}$ be such that $n_x>x$. Then $q(x,x_n)=n-x\ge n_x-x>0$ for all $n\ge n_x$. The case of the space $(\mathbb{R},\bar q)$  and of the sequence  $y_n=-n,\, n\in\mathbb{N},$  can be treated similarly.

\subsection{Variational principles, fixed points and completeness}

   The following version of EkVP in quasi-metric spaces was proved in \cite{cobz11}.

    \begin{theo}[Ekeland Variational Principle]\label{t.EkVP-qm}
Suppose that $(X,\rho)$ is a  $T_1$ quasi-metric space
and $f:X\to\mathbb{R}\cup\{\infty\}$ is a  bounded below  proper function.
 For given $\varepsilon >0$ let $x_\varepsilon\in X$ be such that
  \begin{equation}\label{eq1.EkVP-qm}
  f(x_\varepsilon)\leq \inf f(X)+\varepsilon.\end{equation}
\begin{enumerate}
  \item If $(X,\rho)$ is right $\rho$-$K$-complete and $f$ is $\rho$-lsc,  then for
  every $\lambda >0$ there exists $z=z_{\varepsilon,\lambda}\in X$  such that

  {\rm(a)}\quad $f(z)+\frac{\varepsilon}{\lambda}\rho(z,x_\varepsilon)\leq f(x_\varepsilon);$

  {\rm(b)}\quad $\rho(z,x_\varepsilon)\leq \lambda;$

  {\rm(c)}\quad $\forall x\in X\setminus\{z\},\quad f(z)<f(x)+\frac{\varepsilon}{\lambda}\rho(x,z).$
 \item If $(X,\rho)$ is right $\bar\rho$-$K$-complete and $f$ is $\bar\rho$-lsc,  then for
 every $\lambda >0$ there exists $z=z_{\varepsilon,\lambda}\in X$  such that

   {\rm(a$'$)}\quad $f(z)+\frac{\varepsilon}{\lambda}\rho(x_\varepsilon,z)\leq f(x_\varepsilon);$

   {\rm(b$'$)}\quad $\rho(x_\varepsilon,z)\leq \lambda;$

   {\rm(c$'$)}\quad $\forall x\in X\setminus\{z\},\quad f(z)<f(x)+\frac{\varepsilon}{\lambda}\rho(z,x).$
\end{enumerate}
\end{theo}

Again, taking $\lambda =1$ in Theorem \ref{t.EkVP-qm}, one obtains the weak form of EkVP in quasi-metric spaces.

\begin{corol}[Ekeland's Variational Principle - weak form]\label{c.wEkVP-qm}
 Suppose that $(X,\rho)$ is a  $T_1$ quasi-metric space
 and $f:X\to\mathbb{R}\cup\{\infty\}$ is a  bounded below proper function.
 \begin{enumerate}
 \item
  If $X$ is right $\rho$-K-complete  and $f$ is  $\rho$-lsc,
 then for every $\varepsilon >0$ there exists an element $y_\varepsilon\in X$  such that

  \begin{equation}\label{eq1.wEkVP-qm}
    \begin{aligned}
    {\rm (i)} &\quad   f(y_\varepsilon)\leq \inf f(X) + \varepsilon, \\
   {\rm (ii)} &\quad\forall  x\in X\setminus \{y_\varepsilon\},\;\quad f(y_\varepsilon) < f(x)+ \varepsilon \rho(x,y_\varepsilon).
    \end{aligned}
    \end{equation}
\item  If $X$ is right $\bar\rho$-K-complete  and $f$ is  $\bar\rho$-lsc,
then for every $\varepsilon >0$ there exists an element $y_\varepsilon\in X$  such that
  \begin{equation}\label{eq2.wEkVP-qm}
  \begin{aligned}
  {\rm (i)} &\quad   f(y_\varepsilon)\leq \inf f(X) + \varepsilon, \\
 {\rm (ii)} &\quad\forall  x\in X\setminus \{y_\varepsilon\},\;\quad f(y_\varepsilon) < f(x)+ \varepsilon \rho(y_\varepsilon,x).
 \end{aligned}
   \end{equation}
   \end{enumerate}
 \end{corol}

   Caristi's fixed point theorem version in  quasi-metric spaces is the following.

  \begin{theo}[Caristi-Kirk Fixed Point Theorem, \cite{cobz11}]\label{t.Caristi1-qm}
   Let $(X,\rho)$ be a $T_1$  quasi-metric space, $f:X\to X$  and $\varphi:X\to\mathbb{R}. $
   \begin{enumerate}
     \item If $X$ is right $\rho$-$K$-complete, $\varphi$ is
    bounded below and $\rho$-lsc   and the  mapping $f $      satisfies the
   condition
   \begin{equation} \label{eq1.Caristi1-qm}
   \rho(f(x),x)\leq \varphi(x)-\varphi(f(x)),\; x\in X,
   \end{equation}
   then $f$ has a fixed point in $X.$
    \item If $X$ is right $\bar\rho$-$K$-complete, $\varphi$ is
     bounded below and $\bar\rho$-lsc    and the  mapping $f $  satisfies the condition
    \begin{equation} \label{eq2.Caristi1-qm}
    \rho(x,f(x))\leq \varphi(x)-\varphi(f(x)),\; x\in X,
    \end{equation}
    then $f$ has a fixed point in $X.$
   \end{enumerate}
   \end{theo}

   In this case we have also a set-valued version.

    \begin{theo}[Caristi-Kirk Fixed Point Theorem - set-valued version, \cite{cobz11}] \label{t.Caristi2-qm}
   Let $(X,\rho)$ be a $T_1$  quasi-metric space, $F:X\rightrightarrows X$ a set-valued mapping
   such that $F(x)\neq\emptyset$ for every $x\in X,$  and $\varphi:X\to\mathbb{R}. $
   \begin{enumerate}
   \item If $X$ is right $\rho$-$K$-complete, $\varphi$ is
    bounded below and $\rho$-lsc   and the  mapping $F $      satisfies the
   condition
   \begin{equation} \label{eq1.Caristi2-qm}
   \forall  x\in X,\; \forall y\in F(x),\quad \rho(y,x)\leq \varphi(x)-\varphi(y),
   \end{equation}
   then $F$ has a fixed point in $X.$
   \item If $X$ is right $\bar\rho$-$K$-complete, $\varphi$ is
    bounded below and $\bar\rho$-lsc    and the  mapping $F$ satisfies the condition
   \begin{equation} \label{eq2.Caristi2-qm}
  \forall  x\in X,\; \forall y\in F(x),\quad \rho(x,y)\leq \varphi(x)-\varphi(y),
   \end{equation}
   then $F$ has a fixed point in $X.$
     \end{enumerate}
   \end{theo}

     As in the symmetric case, the weak form of Ekeland Variational Principle is equivalent to  Caristi's fixed point theorem, \cite{cobz11}.
   \begin{prop}\label{Caristi-wEk-qm}
   Let $(X,\rho)$   be a $T_1$ quasi-metric space. Consider the following assertions.
   \begin{itemize}
     \item[{\rm (wEk)}] For any $\rho$-closed subset $Y$ of $X$, for every   bounded below $\rho$-lsc
    proper  function $f:Y\to\mathbb{R}\cup\{\infty\}$ and for every $\varepsilon > 0$  there exists $x_\varepsilon\in Y$ such that
\begin{equation}\label{eq1.Caristi-wEk-qm}
\forall y\in Y\smallsetminus\{x_\varepsilon\},\qquad f(x_\varepsilon)<f(y)+\varepsilon \rho(y,x_\varepsilon).\end{equation}
\item[{\rm (C)}] For every $\rho$-closed subset $Y$ of $X$ and for any $\rho$-lsc function $\varphi:Y\to\mathbb{R},$ any function $g:Y\to Y$ satisfying \eqref{eq1.Caristi1-qm} on $Y$ has a fixed point.
   \end{itemize}

   Then {\rm (wEk)} $\iff$ {\rm (C)}.
    \end{prop}

   As we have seen, in the case of a metric space $X,$ the validity of the weak form of Ekeland Variational Principle  implies the completeness of $X$ (Proposition \ref{p.EkVP-compl}).
The following proposition contains some partial converse results   in the quasi-metric case.
\begin{prop}[\cite{cobz11}]\label{p.EkVP-compl1-qm}
Let $(X,\rho)$ be a $T_1$ quasi-metric space.
\begin{enumerate}
\item If for every $\rho$-lsc function $f:X\to \mathbb{R}$ and for every $\varepsilon>0$ there exists
$y_\varepsilon\in X$ such that
\begin{equation}\label{eq1.EkVP-compl-qm}
\forall x\in X,\quad f(y_\varepsilon)\leq f(x)+\varepsilon\rho(y_\varepsilon,x),\end{equation}
then the quasi-metric space $X$ is left $\rho$-$K$-complete.
\item If for every $\bar\rho$-lsc function $f:X\to \mathbb{R}$ and for every $\varepsilon>0$ there exists
$y_\varepsilon\in X$ such that
\begin{equation}\label{eq2.EkVP-compl-qm}
\forall x\in X,\quad f(y_\varepsilon)\leq f(x)+\varepsilon\rho(x,y_\varepsilon),\end{equation}
then the quasi-metric space $X$ is left $\bar\rho$-$K$-complete.
\end{enumerate}
\end{prop}\begin{proof} The proof is  similar  to that of Proposition \ref{p.EkVP-compl}, taking  care of the fact  that a quasi-metric has weaker continuity properties than a metric  (see Proposition \ref{p.top-qsm1}).

To prove (1), suppose that  $(x_n)$ is a left $\rho$-$K$-Cauchy sequence in $X$. We show first that, for every $n\in\mathbb{N}$,  the sequence $(\rho(x,x_n))$ is bounded. Indeed, if  $n_1\in\mathbb{N}$ is such that
$\rho(x_{n_1},x_{n_1+k})\le 1$ for all $k\in\mathbb{N},$ then
$$
\rho(x,x_{n_1+k})\le\rho(x,x_{n_1})+\rho(x_{n_1},x_{n_1+k})\le \rho(x,x_{n_1}+1\,,$$
for all $k\in\mathbb{N}$, proving the boundedness  of the sequence $(\rho(x,x_n))$.

Consequently, the function $f:X\to [0,\infty)$ given by
$$
f(x)=\limsup_{n\to\infty}\rho(x,x_{n}),\;\; x\in X\,,
$$
is well defined.

For $x,x'\in X,$
$$
\rho(x,x_{n})\le\rho(x,x')+\rho(x',x_{n}) \,,$$
for all $n\in \mathbb{N}.$  Passing to $\limsup$ in  both sides of this inequality one obtains
$$
f(x')\ge f(x)-\rho(x,x')\,.
$$

Then for every $\varepsilon>0,\, \rho(x,x')<\varepsilon$ implies $f(x')>f(x)-\varepsilon$, proving that $f$ is  $\rho$-lsc at every $x\in X$.

Similarly,

$$
 \rho(x',x_n)\leq \rho(x',x)+\rho(x,x_n),\quad n\in \mathbb{N},$$
 implies
 $$
 f(x')\leq  f(x)+\rho(x',x),$$
 from which follows  that   the function $f$  is $\bar \rho$-usc at every $x.$

 We show now that
\begin{equation}\label{eq3.EkVP-compl-qm}
 \lim_{n\to\infty}f(x_n)=0\,.\end{equation}

 Indeed, for every $\varepsilon >0$ there exists $n_\varepsilon\in\mathbb{N}$ such that
 $$
 \forall n\geq n_\varepsilon,\; \forall k\in \mathbb{N},\quad \rho(x_n,x_{n+k})<\varepsilon,$$
 implying
 $$
 \forall n\geq n_\varepsilon,\quad 0\leq f(x_n)=\limsup_k\rho(x_n,x_{n+k})\leq \varepsilon\,,
 $$
  that is $\lim_nf(x_n)=0.$

  Let now $y\in X$ satisfying \eqref{eq1.EkVP-compl-qm}  for $\varepsilon=1/2. $ Taking $x=x_n$ it follows
  $$
 \forall n\in \mathbb{N},\quad f(y)\leq f(x_n)+\frac12\,\rho(y,x_n)\,.$$

 Passing to $\limsup$ and taking into account \eqref{eq3.EkVP-compl-qm} one obtains
 $$
  f(y)=\frac 12\, f(y)\,,$$
  that implies $f(y)=0$. Since
 $$
   f(y)=0\iff \limsup_n\rho(y,x_n)=0\iff\lim_n\rho(y,x_n)=0\,,
   $$
it follows  that the sequence $(x_n)$ is $\rho$-convergent to $y$, proving
   the left $\rho$-$K$-completeness of the quasi-metric space $X.$

    The proof of (2) is similar, working with the function $g:X\to [0,\infty)$ given by
   $$
   g(x)=\limsup_n\rho(x_n,x),\quad x\in X\,, $$
   which is $\bar\rho$-lsc and $\rho$-usc.
\end{proof}

 \begin{remark}
 Note that    Proposition \ref{p.EkVP-compl1-qm}  does not contain a proper converse (in the sense of completeness) of the weak Ekeland Principle. We have in fact a kind of ``cross" converse, as can be seen from the following explanations.

From  Corollary \ref{c.wEkVP-qm}.2 it follows  that if the quasi-metric space $(X,\rho)$ is right $\bar\rho$-$K$-complete, then  for every $\bar\rho$-lsc function $f:X\to \mathbb{R}$ and every $\varepsilon>0,$ there exists  a point $y_\varepsilon\in X$  satisfying \eqref{eq1.EkVP-compl-qm}.

On the other side,  the fulfillment of \eqref{eq1.EkVP-compl-qm} for any $\rho$-lsc function implies the left $\rho$-$K$-completeness of the quasi-metric space $(X,\rho).$

Of course that, in the metric case, both of these conditions reduce to the
 completeness of $X$.

Taking into account the fact that a sequence $(x_n)$ in $X$ is right $\bar\rho$-$K$-Cauchy if and only if  it is left $\rho$-$K$-Cauchy one obtains the following completeness results:
\begin{align*}
  &(X,\rho) \;\mbox{is right }\bar\rho\mbox{-}K\mbox{-complete}\iff\\
  &\forall  (x_n)\;\mbox{a left }\; \rho\mbox{-}K\mbox{-Cauchy sequence } \;\mbox{in}\; X,\;\;\exists x\in X \;\mbox{such that}\; x_n\xrightarrow{\bar \rho} x\,,
  \end{align*}
  while
  \begin{align*}
  &(X,\rho) \;\mbox{is left } \rho\mbox{-}K\mbox{-complete}\iff\\
  &\forall  (x_n)\;\mbox{a left }\; \rho\mbox{-}K\mbox{-Cauchy sequence } \;\mbox{in}\; X,\;\;\exists x\in X \;\mbox{such that}\; x_n\xrightarrow{\rho} x\,.
  \end{align*}
\end{remark}

  The right converse was given by Karapinar and Romaguera \cite{karap-romag15}. To do this they need to  slightly modify the notion of lsc function.

Let $(X,\rho)$ be a quasi-metric space. A proper function $f:X\to \mathbb{R}\cup\{\infty\}$, is called \emph{nearly} $\rho$-lsc at $x\in X$ if $f(x)\le\liminf_nf(x_n)$ for every sequence $(x_n)$ of distinct points in $X$ which is $\rho$-convergent to $x$.

It is clear that a $\rho$-lsc function is nearly $\rho$-lsc and if the topology $\tau_\rho$  is $T_1$ (equivalent to $\rho(x,y)>0$ for all distinct points $x,y\in X$), then the converse is also true. The following simple example shows that these notions are different in $T_0$ quasi-metric spaces.
\begin{example} Let $X=\{0,1\}$,   $\rho(0,0)=\rho(0,1)=\rho(1,1)=0$ and $\rho(1,0)=1.$  Then every function $f:X\to\mathbb{R}\cup\{\infty\}$ is nearly $\rho$-lsc (there are no sequences formed of distinct points), but the function $f(0)=1,\, f(1)=0$  is not $\rho$-lsc at $x=0$.
  \end{example}

  Indeed, $x_n=1$ satisfies $\rho(0,x_n)=0\to 0,\, f(x_n)=0$ and $f(0)=1>0=\liminf_n f(x_n)$.

  \begin{theo}\label{t1.Karapin-Romag}
  For   a quasi-semimetric space  $(X,\rho)$ the following conditions are equivalent.
\begin{enumerate}
\item $(X,\rho)$ is right $K$-sequentially complete.
\item For every self mapping T of X and every  bounded below and nearly $\rho$-lsc proper function
$\varphi: X \to \mathbb{R}\cup\{\infty\}$  satisfying the inequality
\begin{equation}\label{eq0-qm}
\rho(T(x),x) + \varphi(T(x)) \le\varphi(x)\,,
\end{equation}
 for all $x \in X$, there exists
$z = z_{T,\varphi}\in X$ such that $\varphi(z) = \varphi(T(z))$.
\item For every bounded below and nearly $\rho$-lsc proper function $f : X \to\mathbb{R}\cup\{\infty\}$  and for
every $\varepsilon> 0 $ there exists $y_\varepsilon \in X$ such that
\begin{align*}
{\rm (i)}\;\; &f(y_\varepsilon)\le \inf f(X) + \varepsilon;\\
{\rm (ii)}\; &f(y_\varepsilon) < f(x) + \varepsilon \rho(x,y_\varepsilon) \;\mbox{for all  } x \in X\setminus\overline{\{y_\varepsilon\}}\; \mbox{and  } \\
{\rm (iii)}\; &f(y_\varepsilon) \le f(x)\; \mbox{ for all }
x\in\overline{\{y_\varepsilon\}}.
\end{align*}
  \end{enumerate}
\end{theo}
\begin{proof} We shall present only the proof of the implication (3)\;\;$\Rightarrow$\; (1).

  We   proceed   by contradiction. Suppose that the space $(X,\rho)$ is not right $K$-complete. Then there exists a right $K$-Cauchy sequence $(x_n)$ in $X$ which has no limit. This implies that $(x_n)$  has no convergent subsequences, see Remarks \ref{re.compl-qm}.

    We shall distinct two situations.

   Suppose that
 \bequ\label{eq0a.KR-Ek}
  \exists m,\; \forall k\ge m,\;\exists n_k>k,\;\; \rho(x_{n_k},x_k)>0\,.
\eequ

Then, for $n_1=m$ there exists $n_2>n_1$ such that  $\rho(x_{n_2},x_{n_1})>0$. Taking $k=n_2$ it follows the existence of $n_3>n_2$ such that  $\rho(x_{n_3},x_{n_2})>0$. Continuing in this manner we obtain a
sequence $n_1<n_2<\dots$ such that $\rho(x_{n_{k+1}},x_{n_k})>0$ for all $k\in\mathbb{N}$.

Passing to a further subsequence, if necessary, and relabeling, we can suppose   that

 \begin{equation}\label{eq1a.KR}
 0<\rho(x_{n+1},x_n)< \frac1{2^{n+1}}\,,
 \end{equation}
 for all $n\in\mathbb{N}.$

 If \eqref{eq0a.KR-Ek} does not hold, then

  \bequ\label{eq0b.KR-Ek}
  \forall m,\; \exists k\ge m,\;\mbox{ such that }\;\forall n>k,\;\; \rho(x_{n},x_k)=0\,.
\eequ

For $m=1$ let $k=n_1\ge 1$ be such that  $\rho(x_{n},x_{n_1})=0$ for all $n>n_1$. Now, for $m=1+n_1$ let $n_2>n_1$ be such that  $\rho(x_{n},x_{n_2})=0$ for all $n>n_2$. It follows $\rho(x_{n_2},x_{n_1})=0$.

Continuing in this manner we obtain a
sequence $n_1<n_2<\dots$ such that $\rho(x_{n_{k+1}},x_{n_k})=0$ for all $k\in\mathbb{N}$.

Relabeling, if necessary, we can suppose   that the sequence $(x_n)$ satisfies
 \begin{equation}\label{eq1b.KR}
 \rho(x_{n+1},x_n)=0 \,,
 \end{equation}
 for all $n\in\mathbb{N}.$

Put
 $$
 B:=\{x_n : n\in\mathbb{N}\}\,,$$
 and define $f:X\to \mathbb{R}$ by
\bequ\label{def.f-KR}
 f(x)=\begin{cases}
   \frac1{2^{n-1}}\quad&\mbox{\;\,if } x=x_n \;\;\mbox{for some}\quad n\in\mathbb{N},  \\
2 \qquad\quad&\mbox{  for}\quad x\in X\setminus B.
 \end{cases}\eequ

 The function $f$ is nearly $\rho$-lsc. Indeed, let $x\in X$ and $(y_n)$  a sequence of   distinct points in $X$ converging to $x.$ If the set $\{n\in\mathbb{N} : y_n\in B\}$ would be infinite, then there would  exist the natural numbers   $m_1<m_2<\dots$ and  $n_1<n_2<\dots$ such that $y_{m_k}=x_{n_k},\, k\in\mathbb{N}$. But, this would imply that $(x_n)$ has a subsequence $(x_{n_k})$ convergent to $x$, in contradiction to the hypothesis.
Consequently, $(y_n)$ must be eventually in $X\setminus B$, and so $f(x)\le 2=\lim_nf(y_n).$

For  $\varepsilon =1$ let $y\in X$  satisfying the conditions (i)-(iii). Since
$$
 \{x\in X : f(x)\le\inf f(X)+1\}=\{x\in X : f(x)\le 1\}=B\,,
 $$
it follows  $y=x_m\in B$ for some $m\in\mathbb{N}$.

If \eqref{eq1a.KR} holds,  then
\begin{equation}\label{eq.KR-case1}
f(x_{m+1})+\rho(x_{m+1},x_m)<\frac1{2^m}+\frac1{2^{m+1}}=\frac3{2^{m+1}}<\frac1{2^{m-1}}=f(x_m)\,,
\end{equation}
showing that condition (ii) from (3) is not satisfied, that is (3) does not hold.

If \eqref{eq1b.KR} holds,  then, by the triangle inequality, $$\rho(x_{m+k},x_{m})\le\sum_{i=1}^{k}\rho(x_{m+i},x_{m+i-1})=0\,,
$$
i.e. $x_n\in\overline{\{x_m\}}$ for all $n\ge m$. By (iii),
$$
0< f(x_m)\le f(x_n)=\frac1{2^{n-1}}\,,$$
for all $n\ge m$. Since $2^{n-1}<f(x_m)$ for sufficiently large $n$, this leads to a contradiction.
\end{proof}
\begin{remark} In the proof of Theorem 2 in \cite{karap-romag15} the possibility that $\rho(x_{n+1},x_n)=0$ for all $n\in\mathbb{N}$ (when one can not use  $x_{m+1}$ to obtain the contradiction from \eqref{eq.KR-case1}) is not discussed. So the proof given above fills in this gap.
  \end{remark}

\subsection{Smyth completeness}

We present  now   some results on Caristi FPT and Smyth completeness in quasi-metric spaces obtained by Romaguera and Tirado \cite{romag-tirado15}.

Let $(X,\rho)$ be a quasi-metric space, $\varphi:X\to [0,\infty)$ and $T:X\to X$ such that
\begin{equation}
 \rho(x,Tx)\le\varphi(x)-\varphi(Tx)\,,
\end{equation}
for all $x\in X$.

The mapping $T$ is called $\bar\rho$-\emph{Caristi} if $\varphi $ is $\bar\rho$-lsc and  $\rho^s$-\emph{Caristi} if $\varphi $ is $\rho^s$-lsc.

\begin{theo}[Romaguera and Tirado \cite{romag-tirado15}] Let $(X,\rho)$ be a quasi-metric space.
 \begin{enumerate}
 \item  If  $(X,\rho)$ is right $\bar\rho$-K-complete, then every $\bar \rho$-Caristi map on $X$ has a fixed point.
      \item  If  $(X,\rho)$ is right $\rho$-K-complete, then every $\rho$-Caristi map on $X$ has a fixed point.
  \item        The quasi-metric space   $(X,\rho)$ is right $\bar \rho$-Smyth complete if and only if every $\rho^s$-Caristi map has a fixed point.
 \end{enumerate}
\end{theo}

\begin{remark}
 Some versions of Ekeland Variational Principle in asymmetric locally convex spaces were proved in
 \cite{cobz12}. The equivalence of completeness of a quasi-semimetric space with  Takahashi, Ekeland and Caristi principles is discussed in \cite{cobz19}. Other characterizations of completeness of quasi-metric spaces are given by Romaguera and Valero \cite{romag10b}.
\end{remark}

\subsection{Some results of Bao, Cobza\c s, Mordukhovich  and Soubeyran}

Inspired by the results of  Dancs,  Heged{\H{u}}s and Medvegyev  \cite{dancs-heged83} (see Theorem \ref{t1.DHM}), Bao, Cobza\c s and Soubeyran \cite{bao-cs-soub16}, Bao and Soubeyran  \cite{bao-soubey15a}, and Bao and Th\'era \cite{bao-thera15} proved  versions of Ekeland principle in quasi-semimetric spaces and obtained characterizations of completeness.

They consider a set-valued mapping attached to a function $\varphi$ and to a number $\lambda>0$, as in the following proposition.

\begin{prop}\label{p1.Bao-Soub}
  Let $(X,\rho)$ be a quasi-semimetric space, $\varphi:X\to\mathbb{R}\cup\{+\infty\}$ a proper function and $S_\lambda:X\rightrightarrows X$ the set-valued mapping defined by
  \begin{equation}
  \label{def.Slambda}
  S_\lambda(x)=\{y\in X : \lambda \rho(x,y)\le \varphi(x)-\varphi(y)\}\,.
    \end{equation}

  Then $S_\lambda$ enjoys the following properties:
  \begin{itemize}
    \item[{\rm (i)}] (nonemptiness)\quad $x\in S_\lambda(x)$ for all $x\in\dom(\varphi)$;
     \item[{\rm (ii)}]  (monotonicity) \quad $y\in S_\lambda(x)\;\Longrightarrow\; \varphi(y)\le\varphi(x)$ \; and\; $S_\lambda(y)\subseteq S_\lambda(x)$.
  \end{itemize}
  \end{prop}

 Recall that  a generalized Picard
sequence corresponding for a set-valued mapping $F:X\rightrightarrows X$
  is  a sequence $(x_n)$  in $X$  such that $x_{n+1}\in F(x_n)$ for all $n$.

  \begin{theo} \label{t1.Bao-Soub}
  Let $(X,\rho)$ be a  quasi-semimetric space, and let $\varphi:X\to\mathbb{R}\cup\{+\infty\}$  be   proper.
  Given $x_0\in\dom(f)$    and $\lambda> 0$, consider the set-valued mapping $S_\lambda:X\rightrightarrows X$ defined by  \eqref{def.Slambda}. Assume that
\begin{itemize}
\item[{\rm (C1)}] (boundedness from below)\;  $\varphi$ is bounded from below on $S_\lambda(x_0)$;
\item[{\rm (C2)}](nonempty intersection) \; for every    generalized Picard sequence $(x_n)_{n\in\mathbb{N}_0}$ of $S_\lambda$  (starting with $x_0$),   such that $\varphi(x_n)>\varphi(x_{n+1}),\,\forall n\in\mathbb{N}_0,\,$ and $\sum_{n=0}^\infty \rho(x_n,x_{n+1})<\infty$, there exists $y\in X$ such that $S_\lambda(y)\subseteq S_\lambda(x_n)$ for all $n\in\mathbb{N}_0$, where $\mathbb{N}_0=\mathbb{N}\cup\{0\}$.
    \end{itemize}

Then, there is a  generalized Picard sequence $(x_n)_{n\in\mathbb{N}_0}\,$ (i.e. $x_{n+1}\in S_\lambda(x_n),\, \forall n\in\mathbb{N}_0$) satisfying $\sum_{n=0}^\infty\rho(x_n,x_{n+1})<\infty$,\, $\bar \rho$-convergent to  some   $\bar x\in X$ such that for every $\bar y\in S_\lambda(\bar x)$  the following conditions hold:
\begin{equation}\label{eq1.t1.Bao}\begin{aligned}
{\rm (i)}&\quad \lambda \rho(x_0,\bar y)\le \vphi(x_0)-\vphi(\bar y)\,;\\
{\rm (ii)}&\quad \vphi(\bar y)< \vphi(x)+\lambda \rho(\bar y,x)\;\;\mbox{ for every }\; x\in X\smallsetminus S_\lambda(\bar y)\,;\\
{\rm (iii)}&\quad \rho(\bar x,\bar y)=0,\; \vphi(\bar y)=\vphi(\bar x)\;\;\mbox{ and }\; S_\lambda(\bar y)\subseteq\overline{\{\bar y\}}^{\bar \rho}\,.
\end{aligned}\end{equation} \end{theo}
\begin{proof} We give a proof following the ideas from \cite{bao-cs-soub16}.
Replacing, if necessary, $\rho$ by $\lambda \rho$,  we can suppose $\lambda =1$. Put also $S(x)=S_1(x),\, x\in X$.

We shall define inductively a  left $\rho$-$K$-Cauchy generalized Picard sequence which will satisfy all the requirements of the theorem.\vspace{2mm}

\emph{Case} I. \; Start with $x_0$ and suppose that $\alpha_0:=\inf\vphi(S(x_0))<\vphi(x_0)$. Choose $x_1\in S(x_0)$ such that
$$
\alpha_0\le\vphi(x_1)<\alpha_0+\frac12(\vphi(x_0)-\alpha_0)=\frac12(\alpha_0+\vphi(x_0))<\vphi(x_0)\,.$$

Suppose that we have found $x_0,x_1,\dots,x_n$ satisfying $x_{k+1}\in S(x_k),\,k=0,1,\dots,n-1,\,  \vphi(x_k)>\alpha_k:=\inf\vphi(S(x_k)),\, k=0,1,\dots,n,$ and
\begin{equation*}
\alpha_k\le\vphi(x_{k+1})<\alpha_k+\frac12(\vphi(x_k)-\alpha_k)=\frac12(\alpha_k+\vphi(x_k))<\vphi(x_k)\,,
\end{equation*}
for $k=0,1,\dots,n-1.$

Pick then $x_{n+1}\in S(x_n)$ such that
\begin{equation*}
\alpha_n\le\vphi(x_{n+1})<\alpha_n+\frac12(\vphi(x_n)-\alpha_n)=\frac12(\alpha_n+\vphi(x_n))<\vphi(x_n)\,.
\end{equation*}

 Supposing that we can do indefinitely this procedure we find a generalized Picard sequence $x_{n+1}\in S(x_n),\, n\in\mathbb{N}_0$.  Let us show that the  conditions from \eqref{eq1.t1.Bao} are satisfied by this sequence.

 Since the sequence $(\vphi(x_n))$ is strictly decreasing and bounded from below   (by condition (C1)), there exists $\bar \alpha:=\lim_n\vphi(x_n)=\inf_n\vphi(x_n)$.

 By Proposition \ref{p1.Bao-Soub}, $x_{n+1}\in S(x_n)$ implies $S(x_{n+1})\subseteq S(x_{n})\subseteq S(x_{0})$, so that $\alpha_{n+1}\ge\alpha_n$, implying the existence of $\beta=\lim_n\alpha_n$.
 The inequalities
 $$
 \vphi(x_{n+1})< \frac12(\alpha_n+\vphi(x_n))<\vphi(x_n)$$
 yield for $n\to \infty,\; \bar\alpha\le\frac12(\beta+\bar\alpha)\le\bar \alpha$, implying $\beta=\bar\alpha$. Consequently
 $$
 \lim_n\alpha_n=\lim_n\vphi(x_n)=\bar \alpha\,.$$

The inequalities

$$
  q(x_k,x_{k+1})\le\vphi(x_k)-\vphi(x_{k+1})  \quad(\iff x_{k+1}\in S(x_k)) $$
   yield by summation and taking into account condition (C1),
\bequ\label{eq2b.SC2-Ek-forward}\begin{aligned}
 &\sum_{k=0}^n \rho(x_k,x_{k+1})\le\vphi(x_0)-\vphi(x_{n+1})\le\\&\le \vphi(x_0)-\inf \vphi(S(x_{n+1}))\le \vphi(x_0)-\inf \vphi(S(x_{0}))\,,
\end{aligned}\eequ
showing that  $\sum_{k=0}^\infty\rho(x_n,x_{n+1})<\infty$

 Condition  (C2) implies the existence of $\bar x\in X$ such that
$$
S(\bar x)\subseteq\bigcap_{n=0}^\infty S(x_n)\,.
$$

Since $\bar x\in S(\bar x)\subseteq S(x_n)$, it follows that
\begin{equation}\label{eq3.t1.Bao}
 0\le \rho(x_n,\bar x)\le \vphi(x_n)-\vphi(\bar x)\,,\end{equation}
  for all $n\in\mathbb{N}_0:=\mathbb{N}\cup\{0\}$.

Consequently,
 $$
 \alpha_n\le\vphi(\bar x)\le \vphi(x_n)\,,$$
 for all $n\in\mathbb{N}$, yielding for $n\to \infty,\, \vphi(\bar x)=\bar\alpha$. But then, the inequalities
 \eqref{eq3.t1.Bao} imply $\lim_n\rho(x_n,\bar x)=0$, that is the sequence $(x_n)$ is  $\bar\rho$-convergent to $\bar x$.

We show now that the conditions (iii) from \eqref{eq1.t1.Bao} are satisfied.

The relations $\bar y\in S(\bar x)\subseteq S(x_n)$ imply $\alpha_n\le \vphi(\bar y)\le\vphi(\bar x)=\bar\alpha,\, \forall n\in\mathbb{N}_0$. Letting $n\to \infty$, one obtains $\vphi(\bar y)=\bar\alpha=\vphi(\bar x)$.

Also
$$\bar y\in S(\bar x)\iff 0\le \rho(\bar x,\bar y)\le \vphi(\bar x)-\vphi(\bar y)=0\,,$$
so that $\rho(\bar x,\bar y)=0$.

Finally, if $z\in S(\bar y)\subseteq S(\bar x)$, then, as above, it follows $\vphi(z)=\vphi(\bar x)=\vphi(\bar y)$ and  $0\le \rho(\bar y,z)\le\vphi(\bar y)-\vphi(z)=0,\,$ so that $\,  \rho(\bar y,z)=0,$ that is $z\in \overline{\{\bar y\}}^{\bar \rho}.$

The condition (i) is equivalent to $\bar y\in S(x_0)$ which is true because $\bar y\in S(\bar x) \subseteq S(x_0)$.

Condition (ii) follows from the definition of the set $S(\bar y)$.\vspace{2mm}

\emph{Case} II. Suppose that, for some $n_0\in \mathbb{N},\, \vphi(x_{n_0})=\alpha_{n_0}=\inf\vphi (S(x_{n_0}))$.
Then $x_{n_0+1}=x_{n_0}$ and, by induction, $x_{n_0+k}=x_{n_0}$ for all $k\in\mathbb{N}$.

Then the sequence $(x_n)$ is left $\rho$-$K$-Cauchy and  $\bar\rho$-convergent to $x_{n_0}$.  Also, for $x\in S(x_{n_0}),\, \vphi(x)\ge\alpha_{n_0}=\vphi(x_{n_0})$, so that the inequalities
$$
0\le \rho(x_{n_0},x)\le\vphi(x_{n_0})-\vphi(x)\le 0\,,$$
imply $\rho(x_{n_0},x)=0$ and $\vphi(x_{n_0})=\vphi(x)$. It follows also that $x\in \overline{\{x_{n_0}\}}^{\bar \rho}$, that is $S(x_{n_0})\subseteq \overline{\{x_{n_0}\}}^{\bar \rho}$.

These show that the condition (iii) is satisfied. The validity of the conditions (i) and (ii) follows as in Case I.
   \end{proof}

\begin{remark}\label{re1.Bao-Soub}
 Under the hypotheses of  Theorem \ref{t1.Bao-Soub} it follows
\begin{itemize}
\item[{\rm (ii$'$)}]\quad $\varphi(\bar y)<\varphi(x) +\lambda\rho(\bar y,x)$ \;\;  for all \;\;
 $ x\in X\smallsetminus \overline{\{\bar y\}}^{\bar\rho}$\,.
    \item[{\rm (iii$'$)}]\quad $\varphi(\bar y)\le\vphi(x)\;\;$ for all\;\;
    $x\in \overline{\{\bar y\}}^{\bar \rho}$.
\end{itemize}

   Similar conditions appear in Theorem \ref{t1.Karapin-Romag} too.
\end{remark}

The inequality from (ii$'$) follows from the inclusion $S_{\lambda,\rho}(\bar y)\subseteq\overline{\{\bar y\}}^{\bar \rho}$.

To prove (iii$'$), observe that $\,x\in\overline{\{\bar y\}}^{\bar \rho}\iff \rho(\bar y,x)=0$. If $x\in S_{\lambda,\rho}(\bar y)$, then as it is shown in the proof of Theorem~\ref{t1.Bao-Soub}, $\vphi(x)=\vphi(y_*)=\vphi(x_*)$.  If $x\in\overline{\{\bar y\}}^{\bar \rho}\setminus S_{\lambda,\rho}(\bar y)$, then, by (ii),
$$
\vphi(y_*)<\vphi(x)+\lambda \rho(y_*,x)=\vphi(x)\,.$$

To obtain a characterization of completeness, one needs a weaker notion of lower semicontinuity.
\begin{defi}\label{def.decr-lsc}
 A function $\varphi:X\to\mathbb{R}\cup\{\infty\}$, where $(X,\rho)$ is a quasi-semimetric space, is called  \emph{strictly-decreasing-$\rho$-lsc} if for every  $\rho$-convergent sequence $(x_n)$ in $X$ such that the sequence $(\varphi(x_n))$ is strictly decreasing one has
$$
\varphi(y)\le \lim_n\varphi(x_n)\,$$
 for every   $\rho$-limit $y$ of the sequence  $(x_n)$.   A sequence $(x_n)$ such that the sequence $(\vphi(x_n))$ is strictly decreasing is called \emph{strictly $\vphi$-decreasing}.
\end{defi}

\begin{remark}
 In \cite{bao-cs-soub16} it is shown, by an example, that this notion is strictly  weaker than that of $\rho$-lsc, i.e. there exists a function that is strictly-decreasing $\rho$-lsc  but not $\rho$-lsc. In fact the everywhere discontinuous function $f(x)=0$ for $x\in\mathbb{Q}$ and $f(x)=1$ for $x\in\mathbb{R}\setminus\mathbb{Q}$, defined on $(\mathbb{R},|\cdot|)$,  is strictly-decreasing lsc (because there do not exist a   sequence $(x_n)$ in $\mathbb{R}$ such that the sequence $(\varphi(x_n))$ be strictly decreasing). The function $f$ is not lsc because $f(x)=1>0=\liminf_{x'\to x}\varphi(x)$ for every $x\in \mathbb{R}\setminus\mathbb{Q}$. Also, it is not usc  at every $x\in\mathbb{Q}$.
\end{remark}

\begin{remark}
  The notion of  a function $\varphi:X\to\mathbb{R}\cup\{\infty\}$,  such that  $\varphi(x)\le \lim_n\varphi(x_n)$ for every sequence $(x_n)$ in $X$ converging to $x$  such that
  $\varphi(x_{n+1})\le\varphi(x_{n})$ for all $n\in\mathbb{N}$, appears also in the paper \cite{kirk01b} of Kirk and Saliga, called \emph{lower semicontinuity from above}, in connection with Ekeland Variational Principle and Caristi Fixed Point Theorem .
  \end{remark}

   \begin{remark}\label{re.ser-compl-qm} In \cite{bao-cs-soub16} the following sufficient condition for the fulfillment of  condition (C2) from Theorem \ref{t1.Bao-Soub} was given.

  \begin{enumerate}
  \item
 Let $(X,\rho)$ be a  quasi-semimetric space and $\vphi:X\to\mathbb{R}\cup\{+\infty\}$ a proper function.\\
   If every sequence $(x_n)$ in the  space $(X,\rho)$ such that $\vphi(x_{n+1})<\vphi(x_{n+1}),\, n\in\mathbb{N}$, and $\sum_{n+1}^\infty\rho(x_n,x_{n+1})<\infty$ is $\bar\rho$-convergent to some $x\in X$, and the function $\varphi$ is strictly-decreasing-$\bar\rho$-lsc on $\dom \varphi$, then   condition (C2) is satisfied.
 \end{enumerate}
  \end{remark}
  We mention also the following results relating series completeness and completeness in quasi-metric spaces.
  \begin{prop}[\cite{bao-cs-soub16}, Proposition 3.10]\label{p1.ser-compl} Let $(X,\rho)$ be a quasi-semimetric space.
  \begin{enumerate}
\item If a sequence $(x_n) $ in $X$ satisfies  $\sum_{n=1}^\infty\rho(x_n,x_{n+1})<\infty$, then it is left $\rho$-$K$-Cauchy (or, equivalently, right  $\bar\rho$-$K$-Cauchy).
    \item The space $X$ is left $\rho$-$K$-complete if and only if  every sequence $(x_n)$ in $X$ satisfying $\sum_{n=1}^\infty\rho(x_n,x_{n+1})<\infty$ is $\rho$-convergent to some $x\in X$.
 \item The space $X$ is right $\bar\rho$-$K$-complete if and only if  every sequence $(x_n)$ in $X$ satisfying $\sum_{n=1}^\infty\rho(x_n,x_{n+1})<\infty$ is $\bar\rho$-convergent to some $x\in X$.
  \end{enumerate}\end{prop}
  \begin{proof} (Sketch) The assertion (1) follows from the triangle inequality and the Cauchy criterion of convergence applied to the series $\sum_{n=1}^\infty\rho(x_n,x_{n+1})$:
  $$
  \rho(x_n,x_{n+k})\le\sum_{i=0}^{k-1}\rho(x_{n+i},x_{n+i+1})<\varepsilon\,.$$

(2) and (3).\;     If $(x_n)$ is left $\rho$-$K$-Cauchy, then there exist the numbers $n_1<n_2<\dots$ such that $\rho(x_{n_k},x_{n_{k+1}})<1/2^k,\, k\in\mathbb{N}$. Then $\sum_{k=1}^\infty\rho(x_{n_k},x_{n_{k+1}})<\infty$ so that, by hypothesis, there exists $x\in X$ with $\lim_k\rho(x,x_{n_k})=0$  ($\lim_k\bar\rho(x,x_{n_k})=0$). By Remarks \ref{re.compl-qm}.(2)  $\lim_n\rho(x,x_{n})=0\, $ (resp. $\lim_n\bar\rho(x,x_{n})=0$).
  \end{proof}

  We will present now    a characterization of  completeness of   quasi-semimetric spaces in terms of the Ekeland Variational Principle (Theorem \ref{t1.Bao-Soub}).

\begin{theo}[A characterization of  completeness]\label{t2.Bao-Soub}
For  any quasi-semimetric space $(X,\rho)$    the following conditions are  equivalent.
\begin{enumerate}
\item The space $X$ is right $\bar\rho$-$K$-complete.
  \item For every proper, bounded from below, and  strictly-decreasing-$\bar\rho$-lsc
function $\varphi:X\to\mathbb{R}\cup\{+\infty\}$  and for any $x_0\in\dom(\varphi)$ there is $\bar x\in X$ such that for every
$\bar y\in  S_1(\bar x) $ one has
\begin{itemize}
\item[{\rm (i)}]\;\; $  \rho(x_0,\bar y)\le \varphi(x_0)-\varphi(\bar y)$;
\item[{\rm (ii)}]\;\; $\varphi(\bar y)<\varphi(x)+ \rho(\bar y,x),\;\;\mbox{for all   } \; x\in X\smallsetminus\overline{\{\bar y\}}^{\bar\rho}$;
\item[{\rm (iii)}]\; $ \vphi(\bar y)\le\vphi(x) $ for all  $x\in\overline{\{\bar y\}}^{\bar \rho}\,.$
 \end{itemize}\end{enumerate}\end{theo}
 \begin{proof}
 (1) $\Longrightarrow$ (2) This implication follows by  Theorem \ref{t1.Bao-Soub} (see also Proposition \ref{p1.ser-compl} and the Remarks \ref{re1.Bao-Soub},\;     \ref{re.ser-compl-qm}).

 The proof of the implication (2) $\Longrightarrow$ (1) is similar to that of the implication (3) $\Longrightarrow$ (1) in Theorem \ref{t1.Karapin-Romag}, working with the conjugate metric $\bar \rho$ instead of $\rho$.\end{proof}

The following example shows that the completeness could not hold if we suppose that only the conditions (i) and (ii) from Theorem \ref{t2.Bao-Soub} hold.

\begin{example}
Let $x_n=-n,\, n\in\mathbb{N}_0,$ and $X=\{x_n : n\in\mathbb{N}_0\}$ with the metric
$q(x_n,x_m)=(x_m-x_n)^+=(-m+n)^+ =n-m$ if $n>m$ and $=0$ if  $n\le m$ (see Example \ref{ex.R-qm}).  Then $q(x_n,x_{n+1})=0$ for all $n\in\mathbb{N}_0$ and the space $X$ is not right $\bar q$-$K$-complete (see Example \ref{ex.R-compl-qm}).  Let $\vphi:X\to[0,\infty)$ be an arbitrary function. For  $\bar x=x_0$,
$$
 \vphi(\bar x)=\vphi(x_0)\le \vphi(x_0)= \vphi(x_0)+q(\bar x,x_0)\,.$$

  Since $q(x_0,x_n) =0$ for all $n\in\mathbb{N}_0$ the condition
$$
\vphi(x_0)<\vphi(x_n)+q(x_0,x_n)\quad \mbox{for all}\quad n\in \mathbb{N}_0\quad\mbox{ with}\quad  q(x_0,x_n)>0\,,$$
is trivially satisfied.
\end{example}

\begin{remark}  Bao and Soubeyran use the notions of forward and backward in  a quasi-semimetric space $(X,\rho)$,
where forward means with respect to $\bar\rho$, while backward  means with respect to $\rho$.

  For instance, a sequence $(x_n)$ in $X$ is:
\begin{itemize}
\item
forward convergent to $x$  if \quad $\bar\rho(x,x_n)=\rho(x_n,x)\to 0$, i.e. it is $\bar\rho$-convergent to $x$;
\item
forward Cauchy  if   for every $\varepsilon >0$ there exists $n_\varepsilon\in\mathbb{N}$ such that $\rho(x_n,x_{n+k})<\varepsilon$, for all $n\ge n_\varepsilon$ and all $k\in\mathbb{N}$,  i.e. if it is left $\rho$-$K$-Cauchy, or equivalently, right $\bar\rho$-$K$-Cauchy.
\item
backward convergent to $x\,$  if $\,\rho(x,x_n) \to 0$, i.e. it is $\rho$-convergent to $x$;
\item
backward Cauchy  if   for every $\varepsilon >0$ there exists $n_\varepsilon\in\mathbb{N}$ such that $\rho(x_{n+k},x_{n})<\varepsilon$, for all $n\ge n_\varepsilon$ and all $k\in\mathbb{N}$,  i.e. if it is right $\rho$-$K$-Cauchy, or equivalently, left $\bar\rho$-$K$-Cauchy.
\end{itemize}

The space $X$ is called
forward-forward complete if  every forward Cauchy sequence is forward convergent, i.e. if it is   right $\bar\rho$-$K$-complete.

 The backward notion of completeness  is defined on an analogous way: the quasi-metric space $(X,\rho)$  is backward-backward complete if every backward Cauchy sequence is backward convergent, i.e. if it is right $\rho$-$K$-complete. One can define also combined notions of completeness: backward-forward (meaning left $\bar\rho$-$K$-completeness)  and  forward-backward (meaning left $\rho$-$K$-completeness).
\end{remark}

Bao, Soubeyran and Mordukhovich applied systematically variational principles in quasi-metric spaces to various domains of human knowledge as    psychology in \cite{bao-mord15a}, to capability of wellbeing and rationality in \cite{bao-mord15b}, and to group dynamics in \cite{bao-soubey15a}.

\section{Topology and order}\label{S.top-ord}

In this section we shall discuss some results relating topology and order.
\subsection{Partially ordered sets}
 Let $X$ be a nonempty set. A \emph{preorder} on $X$ is a reflexive and transitive relation $\le\subseteq X^2$. The pair $(X,\le)$ is called a \emph{partially preordered set}. An antisymmetric preorder is called an \emph{order} on $X$. The pair $(X,\le)$ is called a \emph{partially ordered set}, or a \emph{poset} in short.

 Let $(X,\le)$ be a partially preordered set. For a nonempty subset $A$ of a partially preordered set $(X,\le)$ one puts
 \begin{equation}\label{def.up-set}
 \uparrow\!\! A:=\{y\in X : \exists x\in A,\, x\le y\}\;\;\mbox{and}\;\; \downarrow\!\! A:=\{y\in X : \exists x\in A,\, y\le x\}\,.
 \end{equation}

In particular, for $x\in X,$
$$
\uparrow\!\! x:=\uparrow\!\!\{x\}\;\;\mbox{and}\;\; \downarrow\!\! x=\downarrow\!\! \{x\}\,.
$$

The set $A$ is called \emph{upward closed} if $A=\uparrow\!\! A$ and \emph{downward closed} if $A=\downarrow\!\! A$.

A nonempty subset $A$ of $X$ is called \emph{totally ordered} (or a \emph{chain}) if any two elements of $A$ are comparable.

A nonempty subset $D$ of $X$ is called \emph{directed}   if for any two elements of $x,y\in A$ there exists $z\in A$ such that  $x\le z$ and $y\le z$.

Let $A\subseteq X$. Then

\textbullet\; an \emph{upper bound} for $A$ is an element $z\in X$ such that  $\forall x\in A,\, x\le z$;

\textbullet\; if $z$ is an upper bound for $A$ and $z\in A$ then $z$ is the \emph{greatest element} of $A$;

\textbullet\; the set $A$ is called \emph{upper-bounded} if it has at least one upper bound;

\textbullet\; the least upper bound of $A$ is called the \emph{supremum} of $A$, denoted by $\sup A$ (or by  $\bigvee\! A$);

\textbullet\; in the case of two elements $x,y\in X$ one uses the notation $x\vee y:=\bigvee\{x,y\}$;

\textbullet\; the greatest element of $X$ is called  \emph{unity} and is denoted by $\top$ (or by 1); the least element of $X$ is called \emph{zero} and is denoted by $\bot$ (or by 0);

Dually, one defines  lower bounds, infima, etc. (we have  used the notion of least element   yet). The infimum is denoted by $\inf A$ (or by $\bigwedge\! A)$. Also $x\wedge y:=\bigwedge\{x,y\}$.

\begin{defi}\label{def.posets}
 A poset $(X,\le)$ is called

\textbullet\; an \emph{upper semi-lattice} if every two elements $x,y\in X$ have a $\sup,\, x\vee y$;

\textbullet\; an \emph{lower semi-lattice} if every two elements $x,y\in X$ have an $\inf,\, x\wedge y$;

\textbullet\; if $(X,\le)$  is an upper semi-lattice and a lower semi-lattice, then it is called a \emph{lattice}, that is for every   $x,y\in X$ there exist $   x\wedge y$ and $   x\vee y$;

\textbullet\; a lattice $(X,\le)$ is called \emph{complete} if for every subset $A$ of $X$ there exist
$\sup A$ and $\inf A$.

 \textbullet\; a poset $(X,\le)$ is called \emph{directed} (\emph{chain, boundedly}) \emph{complete} if every directed (totally ordered, upper bounded) subset of $X$ has  supremum. A directed  complete partially ordered set is denoted in short by \textbf{dcpo}.
\end{defi}

\begin{remarks}\label{re1.posets} Let $(X,\le)$ be a poset.

1.\, If $X$ has a greatest element $\top$, then $\top=\sup X$. If $X$ has a least element $\bot$, then $\bot=\inf X$.

2. In the definition of a complete lattice $(X,\le) $ it suffices to request that every subset of $X$ has a supremum, because $X$ has a least element $\bot=\sup\emptyset$ and the infimum of a subset $A$ of $X$ is the supremum of the set $L_A$ of all lower bounds of $A$ (this set is nonempty because $\bot\in A$).
\end{remarks}

\textbf{Mappings between partially preordered sets}

Let $(X,\le),\, (Y,\preceq)$ be two partially preordered sets. A mapping $f:X\to Y$ is called

\textbullet\; \emph{increasing} if $\forall x,y\in X,\; x\le y\;\Rightarrow\; f(x)\preceq f(y)$;

\textbullet\; \emph{decreasing} if $\forall x,y\in X,\; x\le y\;\Rightarrow\; f(y)\preceq f(x)$;

\textbullet\; \emph{monotonic} if it is increasing or decreasing.

One says that $f$ \emph{preserves}

\textbullet\; \emph{suprema} if and only if for every $A\subseteq X$, the existence of $\sup A$ implies the existence of $\sup f(A)$ and the equality $\sup f(A)=f(\sup A)$.

\textbullet\; \emph{finite} (\emph{directed, chain}) \emph{suprema} if and only if the above condition holds for every finite (respectively directed, totally ordered) subset $A$ of $X$.

Similar definitions can be given for infima.

\begin{remark}\label{re1.suprema-pres}
  A mapping preserving finite suprema is increasing.
\end{remark}

Indeed, if $x\le y$ in $X$, then $y=\sup\{x,y\}$, and so $f(y)=\sup\{f(x),f(y)\}$, implying $f(x)\preceq f(y)$.

\subsection{Order relations in topological spaces}
The \emph{specialization order} of a topological space $(X,\tau)$ is the partial order defined by
\begin{equation}\label{def.top-ord}
x\le_\tau y\iff x\in\overline{\{y\}}\,,
\end{equation}
that is $y$ belongs to every open set containing $x$.
\begin{prop}\label{p1.top-ord}
  Let $(X,\tau)$ be a topological space. The relation defined by \eqref{def.top-ord} is a preorder.

  It is an order if and only if $X$ is $T_0$.

  If $X$ is $T_1$, then $\le_\tau$ is the equality relation in $X$.
\end{prop}\begin{proof} Since $x\in\overline{\{x\}}$ it follows $x\le_\tau x$.

The transitivity follows from the following implication
$$
x\in  \overline{\{y\}}\;\mbox{and}\; \{y\}\subseteq\overline{\{z\}}\;\Rightarrow\; x\in\overline{\{y\}}\subseteq\overline{\overline{\{z\}}}=\overline{\{z\}}\,,
$$
that is
$$
x\le_\tau y \;\mbox{and}\;y\le_\tau z \;\Rightarrow\; x\le_\tau z\,.$$

Suppose that $X$ is $T_0$ and $x,y$ are two  distinct  points in $X$. Then there exists an open set $V$ that contains exactly one of this points. If $x\in V$ and $y\notin V$, then $x\notin\overline{\{y\}}$, that is the relation $x\le_\tau y$ does not hold.
If $y\in V$ and $x\notin V$, then $y\notin\overline{\{x\}}$, that is the relation $y\le_\tau x$ does not hold.
This means that we can not have simultaneously $x\le_\tau y$  and $y\le_\tau x$ for a  pair of distinct elements in $X$.

Similar reasonings show that $X$ is $T_0$ if $\le_\tau$ is a partial order (i.e. it is antisymmetric).

The topological space  $X$ is $T_1$ if and only if  $\overline{\{x\}}=\{x\}$ for every $x\in X$. Consequently,
$$
x\le_\tau y\iff x\in \overline{\{y\}}=\{y\}\iff x=y\,.
$$
 One shows also that if the order relation $\,\le_\tau$ is equality, then $X$ is $T_1$.
  \end{proof}

In the following results the  order notions are considered with respect to  the order $\le_\tau$.
\begin{prop}\label{p2.top-ord}
  Let $(X,\tau)$ be a topological space and $A\subseteq X$.
  \begin{enumerate}
  \item If the set $A$ is  open, then it is upward closed, i.e. $\uparrow\!\! A=A$.
  \item If the set $A$ is  closed, then it is downward closed, i.e. $\downarrow\!\! A=A$.
  \end{enumerate}
\end{prop}  \begin{proof}
  (1)\; It is a direct consequence of definitions. Let  $x\in A$ and $y\in X,\, x\le_\tau y.$ Since $A$ is open, this inequality   implies $y\in A.$

  (2) \;  Let  $x\in A$ and $y\in X,\, y\le_\tau x.$ Then $y\in\overline{\{x\}}\subseteq \overline A=A$.
\end{proof}

Let us define the \emph{saturation} of a subset $A$ of $X$ as the intersection of all open subsets of $X$ containing $A$. The set $A$ is called \emph{saturated} if  equals its saturation.
\begin{prop}\label{p3.top-ord}
  Let $(X,\tau)$ be a topological space.
  \begin{enumerate}
  \item For every $x\in X$, $\downarrow\!\! x=\overline{\{x\}}$.
  \item  For any subset $A$ of $X$ the saturation of $A$ coincides with $\uparrow\!\! A$.
  \end{enumerate}
\end{prop}\begin{proof}
 (1)\; This follows from the equivalence
 $$
 y\le_\tau x\iff y\in\overline{\{x\}}\,.$$

 (2)\; Since every open set is upward closed, $U\in \tau$ and $U\supset A$ implies $U\supset\uparrow\!\! A,$
 that is
 $$
 \uparrow\!\! A\subseteq \bigcap\left\{U\in\tau : A\subseteq U\right\}\,.
 $$

 If $y\notin \uparrow\!\! A$, then for every $x\in A$ there exists $U_x\in \tau$ such that  $x\in U_x$ and  $y\notin U_x.$ It follows $y\notin V:=\bigcup\{U_x : x\in A\}\in\tau$ and $ A\subseteq V$, hence $y\notin \bigcap\{U \in \tau : A\subseteq U\}$, showing that
$$
\complement\left(\uparrow\!\!A\right)\subseteq \complement\left(\bigcap\left\{U\in\tau : A\subseteq U\right\}\right)\iff \bigcap\left\{U\in\tau : A\subseteq U\right\}\subseteq  \uparrow\!\!A\,.$$
\end{proof}

\textbf{Compactness}

We present following \cite[p. 69]{Larrecq} a result on compactness.
\begin{prop}\label{p.top-ord-comp}
Let $(X,\tau)$  be a topological space. A subset $K$ of $X$ is compact if and only if  its saturation $\uparrow\!\!K$ is   compact.
\end{prop}\begin{proof} The equivalence follows from the following remark: since every open subset of $X$ is upward closed with respect to the specialization order, the following equivalence is true
$$
K\subseteq\bigcup\{U_i : i\in I\}\iff \uparrow\!\!K\subseteq\bigcup\{U_i : i\in I\}\,,
$$
for every family $\{U_i : i\in I\}\subseteq \tau$.
\end{proof}

An order $\preceq$ on a topological space $(X,\tau)$ is said to be \emph{closed} if and only if its graph $\graph(\preceq):=\{(x,y)\in X\times X : x\preceq y\}$ is closed in $X\times X$ with respect to the product topology.
The existence of a closed order on a topological space forces the topology to be Hausdorff.
\begin{prop}[\cite{Larrecq}, P.3.9.12]\label{p4.top-ord}
If on a topological space $(X,\tau)$  there exists a closed order $\preceq$, then the topology $\tau$ is Hausdorff.
\end{prop}\begin{proof}
  Let $x,y$ be distinct points in $X$. Then the relations $x\preceq y$ and $y\preceq x$ can not both hold. Suppose, without the loose of generality,  that  $x\preceq y$ does not hold, that is $(x,y)\notin \graph(\preceq)$. Then    there exists the open neighborhoods $U,V$ of $x$ and $y$, respectively, such that   $(U\times V)\cap\graph(\preceq)=\emptyset$. The proof will be done if we show that $U\cap V=\emptyset$. Indeed supposing that  there exists $z\in W:=U\cap V$, one obtains the contradiction
  $$
   (z,z)\in (W\times W)\cap \graph(\preceq)\subseteq (U\times V)\cap \graph(\preceq)=\emptyset\,.
   $$
   \end{proof}

   For other properties of topological spaces endowed with a closed order (e.g. compactness), see \cite[Section 9.1.1]{Larrecq}

\subsection{Topologies on ordered sets - Alexandrov's, the  upper topology, Scott's, the interval topology}

Consider a partially preordered set $(X,\le)$. We are interested to define a topology $\tau$ on $X$ such that  the specialization preordering $\le_\tau$  coincides with $\le$. The answer is, in general no. For instance, on $\mathbb{R},$ with the usual ordering, we can consider several topologies (the usual, the discrete, etc), all having as specialization preordering the equality.

Let $(X,\le)$ be a partially preordered set. We shall consider three topologies on $X$ such that the corresponding specialization preorderings coincide with $\le$, as well as the interval topology and the Moore-Smith order topology. \\

\textbf{The Alexandrov topology - the finest}

This is the finest of these topologies.
\begin{prop}[Alexandov topology]\label{p.Alex-top}  Let $(X,\le)$ be a partially preordered set. Then there exists a finest  topology $\tau_a$ on $X$, called the Alexandrov topology,  such that  the specialization preordering $\,\le_{\tau_a}$ coincides with $\,\le$. This topology  is characterized by   the condition

{\rm (i)}\;\; the open sets are exactly the upward closed sets,

 or, equivalently,

  {\rm (ii)}\;\; the closed  sets are exactly the downward closed sets.

  In this case, $\uar x\in\tau_a$ for every $x\in X.$
\end{prop}
\begin{proof}
  It is clear that the upward closed subsets of $X$ form a topology $\tau_a$  on $X$. Since $\uar x$ is upward closed, it follows  $\uar x\in\tau_a.$

  We have
  \begin{align*}
    x\le_a y&\iff x\in\clos_{\tau_a}\{y\}\\
    &\iff\forall Z\in\tau_a,\; x\in Z\;\Ra\; y\in Z\,.
  \end{align*}

  If $x\le_a y,$ then $x\in\uar x$ and $\uar x\in\tau_a,$ imply   $y\in\uar x,$   that is, $x\le y.$

  If $x\le y$ and $Z\in\tau_a$ contains $x$, then $y\in Z,$ as $Z$ is upward closed, showing that $x\le_a y.$
\end{proof}

  We use the notation $X_a$ for $(X,\tau_a)$.
  \begin{remark}\label{re.Alex-top}
    Since for every upward closed subset $Z$ of $X$, $\, Z=\bigcup\{\uparrow\!\!z : z\in Z\}$ it follows that the Alexandrov topology $\tau_a$ is generated by the family of sets $\{\uparrow\!\!x : x\in X\}$.
  \end{remark}

\textbf{  The upper topology - the coarsest}

\begin{prop}\label{p.upper-top} Let $(X,\le)$ be a partially preordered set. Then there exists a coarsest topology $\tau_u$ on $X$ such that  the specialization preordering $\,\le_{\tau_u}$ coincides with $\,\le$.

A subbase of $\tau_u$ is formed by the complements of the sets $\downarrow\!\!x,\, x\in X$.

A basis  of $\tau_u$ is formed by the complements of the sets $\downarrow\!\!E$ for $E\subseteq X,\, E$ finite.
\end{prop}\begin{proof}
  It is easy to check that the sets $\downarrow\!\!E$ for $E\subseteq X,\, E$ finite, form a basis  of a topology $\tau_u$ on $X$. Denote by  $\le_u$ the specialization order determined by this topology.

 Let $x\le_u y$. By definition this is equivalent to $x\in\overline{\{y\}}.$ But
  \begin{align*}
 x\in\overline{\{y\}}\iff&  \forall z,\; [x\in X\smallsetminus \downarrow\!\!z\;\Rightarrow\;y\in X\smallsetminus \downarrow\!\!z]  \\
 \iff&  \forall z,\; [y\in \downarrow\!\!z\;\Rightarrow\;x\in\downarrow\!\!z]  \\
 \underset{(z=y)}{\Longrightarrow}\;& x\in\downarrow\!\!y\; \iff\; x\le y\,.
 \end{align*}

 Conversely, suppose $x\le y$.

If for some $z\in X$, $ y\in  \downarrow\!\!z$, then $x\le y$ and $y\le z$ would imply $x\le  z$, that is $x\in \downarrow\!\!z$. Consequently
 $$
 x\in X\smallsetminus \downarrow\!\!z\;\Rightarrow\;y\in X\smallsetminus \downarrow\!\!z\,,$$
 for all $z\in X$, which is equivalent to $x\le_u y$.
  \end{proof}

\textbf{The Scott topology}

 This is a topology between $\tau_u$ and $\tau_a$. It is defined in the following way.

Let $(X,\le)$ be a partially ordered set. A subset $U$ of $X$ is \emph{Scott open} if and only if the following two conditions hold:

(i)\; $U$ is upward closed, and

(ii)\; for every nonempty directed subset $D$ of $X$ such that  $\sup D$ exists (in $X$) and belongs to $U$, there exists $d\in D$ such that  $ d\in U$.

\begin{prop}\label{p1.Scott-top} Let $(X,\le)$ be a partially ordered set.
 \begin{enumerate}\item
 The family of Scott open subsets of $X$ forms a topology denoted by $\tau_\sigma$.
 \item A subset $F$ of $X$ is Scott closed if and only if  the following two conditions hold:

{\rm (i)}\; $F$ is downward closed, and

{\rm (ii)}\; for every nonempty directed subset $D$ of $F$ if $\sup D$ exists (in $X$), then $\sup D\in F$.

In particular, the set $\downarrow\!\!y$ is Scott closed for every $y\in Y$.
\item
 The specialization order corresponding to $\tau_\sigma$ agrees with $\le$ and

 $$ \tau_u\le \tau_\sigma\le\tau\,.$$
\item Let $(X,\tau)$ be a topological space,   $ \le_\tau$ the specialization order corresponding to $\tau$ and $\sigma=\sigma(\le_\tau)$ the Scott topology corresponding to $\le_\tau$. Then the set
    $\overline{\{x\}}^\tau$ is Scott  closed (i.e $\sigma$-closed) for every $x\in X$.
  \end{enumerate}
  \end{prop}\begin{proof}
  (1)\;  Let $U_i,\, i=1,\dots,n,$ be Scott open sets. Then $U:=\bigcap\{U_i : 1\le i\le n\}$ is upward closed. Suppose that $D$ is a directed set in $X$ such that  $\sup D$ exists and belongs to $U$. Then $\sup D\in U_i$ implies the existence of $x_i\in D\bigcap U_i,\, i=1,\dots,n.$  Since $D$ is directed there exists $x\in D$ with $ x_i\le x,\, i=1,\dots,n$. Since each $U_i$ is upward closed, $x\in U_i,\, i=1,\dots,n$, that is $x\in U,$ showing that $U$ is Scott open too. It easy to show  that the union of an arbitrary family of Scott open sets is again Scott open.

  The proof of (2) follows from the  equality $F=X\smallsetminus U$ relating open sets $U$ and closed sets $F$.

  To prove that the set $\downarrow\!\!y$ is Scott closed, let $D\subseteq \downarrow\!\!y$  be a directed set and $d_0=\sup D$. Since, $d\le y$ for all $d\in D$, it follows $d_0\le y$, i.e. $d_0\in \downarrow\!\!y$, showing that $\downarrow\!\!y$ is Scott closed.

 (3)\;  Denote by $\le_\sigma$ the specialization order corresponding to $\tau_\sigma$ and let $x\le y$.
   If $U$ is Scott open and $x\in U$, then $y\in U$, as $U$ is upward closed. Consequently $x\le_\sigma y$.

   Suppose now $x\le_\sigma y$.  The set $\downarrow\!\!y$ is Scott closed. If $x$ would belong to
    $X\smallsetminus \downarrow\!\!y$, then $y\in X\smallsetminus \downarrow\!\!y$, a contradiction, so $x$ must belong to $\downarrow\!\!y$, that is $x\le y$.

(4)\; The proof is based on (2).  Denoting by
$\le_\tau, \le_\sigma$  the specialization orders corresponding to $\tau$ and $\sigma$, respectively, then  $\le_\tau=\le_\sigma$.

Show first that the set $\overline{\{x\}}^\tau$ is $\le_\tau$-downward closed.

Indeed, let $y\in \overline{\{x\}}^\tau$ and $y'\le_\tau y$. Then $y\le_\tau x$ and $y'\le_\tau y$ imply $y'\le_\tau x$, that is $y'\in \overline{\{x\}}^\tau$.

Let us verify condition (ii) from (2).

If $\{x_i: i \in I\}$ is a directed set contained in $\overline{\{x\}}^\tau$, then $x_i\le_\tau x$ for all $i\in I,$ and so $\sup_ix_i\le_\tau x$, or equivalently, $\sup_ix_i\in \overline{\{x\}}^\tau$.
     \end{proof}

In the following proposition we characterize the continuity with respect to  the Scott topology.  We use the notation   $X_\sigma$ for $(X,\tau_\sigma)$.
\begin{prop}\label{p2.Scott-top} Let $X,Y$ be a partially ordered sets and $f:X\to Y$. The following are equivalent.
\begin{enumerate}
\item The function $f$ is continuous with respect to  the Scott topologies on $X$ and $Y$, respectively.
\item  The function $f$ satisfies the following conditions:\\
{\rm (i)}\; $f$ is increasing;\\
{\rm (ii)}\; $f$ preserves the suprema of directed sets.\end{enumerate}
\end{prop}\begin{proof}
All  closures that appear in  this proof are considered with respect to the Scott topology.

  (2)\; $\Rightarrow$\; (1). The continuity of $f$ is equivalent to each of the following conditions
  \begin{equation}\label{eq1.Scott-cont}
  \begin{aligned}
    &{\rm (i)}\;\;  &f^{-1}(Z)\; \mbox{is closed for every closed subset } Z\mbox{ of } Y;  \\
    &{\rm (ii)}\;\;  &f\big(\overline A\big)\subseteq \overline{f(A)}\;\;\mbox{for every subset } A \;\mbox{of}\; X\,.
      \end{aligned}\end{equation}

  Let $Z\subseteq Y$ be Scott closed. We shall use Proposition \ref{p1.Scott-top}.(2) to show that $f^{-1}(Z)$ is Scott closed.

If $x\in f^{-1}(Z)$ and $x'\le x$, then $ f(x)\in Z$ and $f(x')\le f(x)$. Since $Z$ is downward closed it follows  $f(x')\in Z$, which is equivalent to $x'\in f^{-1}(Z)$. Consequently $f^{-1}(Z)$ is downward closed.

  Let now $(x_i)_{i\in I}$ be a directed set contained in $f^{-1}(Z)$ such that  $x=\sup_ix_i$ exists.
  Then  $(f(x_i))_{i\in I}$   is a directed set in $Y$ contained in $Z$. By hypothesis $f(x)=\sup_i f(x_i)$ and, since $Z$ is Scott closed, $f(x)\in Z$ which is equivalent to $ x\in f^{-1}(Z)$.

   (1)\; $\Rightarrow$\; (2). Suppose that $f$ is continuous with respect to  the Scott topologies on $X$ and $Y$, respectively.

   Let $x'\le x$ in $X$. Taking into account the continuity of $f$ we have
   $$
   x'\le x\iff x'\in\overline{\{x\}} \Rightarrow f(x')\in f\big(\overline{\{x\}}\big)\subseteq \overline{f(x)}\,,$$
  which shows that  $f(x')\le f(x)$ in $Y$.

 Let now $(x_i)_{i\in I}$ be a directed set in $X$ such that  $x=\sup_ix_i$ exists. Since $f$ is increasing, it follows $f(x_i)\le f(x)$ for all $i\in I$. Let $y\in Y$ be such that  $f(x_i)\le y$ for all $i\in I$. Then
   \begin{align*}
     \forall i,\;\; f(x_i)\le y\iff  \forall i,\;\; f(x_i)\in\overline{\{y\}}\iff\forall i,\;\; x_i\in f^{-1}\big(\overline{\{y\}}\big)\,.
   \end{align*}

   Since $f^{-1}\big(\overline{\{y\}}\big)$ is Scott closed, it follows $x=\sup_ix_i\in f^{-1}\big(\overline{\{y\}}\big)$, which implies
   $f(x)\in\overline{\{y\}},$ that is $f(x)\le y$. Consequently, $f(x)$ is the least upper bound of $(f(x_i))_{i\in I}$.
  \end{proof}

  \begin{remark}\label{re.Scott-cont} A mapping satisfying  condition (i) and (ii)   from Proposition \ref{p2.Scott-top}  is called \emph{Scott continuous}.
  In fact, by Remark \ref{re1.suprema-pres}, it suffices to suppose that $f$ satisfies only the condition (ii).\end{remark}

\begin{example}[\cite{Larrecq}] A subset of $\mathbb{R}$ is compact and saturated with respect to  the Scott topology if and only if  it is the empty set or of the form $[\alpha,\infty)$ for some $\alpha \in\mathbb{R}$.
\end{example}

\textbf{The interval topology and the Moore-Smith order topology}

These topologies were  defined by Frink \cite{frink42}. By a closed interval in a  poset $(X,\le)$ one understands a set of the form
\begin{equation}
 \label{def.intervals}\begin{aligned}
&\uparrow\!\! a=\{x\in X : a\le x\},\quad  \downarrow\!\! b=\{y\in X : y\le b\},\quad\mbox{or}\\
&[a,b]=\{x\in X : a\le x\le b\}=\uparrow\!\! a\,\cap \downarrow\!\! b\,,
\end{aligned}
\end{equation}
for $a,b\in X$.  By definition, a subset $Y$ of $X$ is   closed with respect to the interval topology if it can be written as    the intersection of   finite unions of sets of the form \eqref{def.intervals}. It is shown in \cite{frink42} that the family $\mathcal F_\le$ of closed sets defined above satisfies the axioms of closed sets:
\begin{align*}
 &{\rm  (i)}\quad \emptyset,X\in \mathcal F_\le\,;\quad
 {\rm  (ii)}\quad F_i\in \mathcal F_\le,\, i\in I,\;\Rightarrow \bigcap_{i\in I}F_i\in \mathcal F_\le\, ;\\
 &{\rm  (iii)}\quad F_1,F_2\in \mathcal F_\le\;\Rightarrow\;\;  F_1\cup F_2\in \mathcal F_\le\,.
 \end{align*}

If the set $X$ is totally ordered (i.e. it is a chain), then the interval topology defined above coincides with the \emph{intrinsic topology}, which is the topology having as basis  of open sets the intervals
$$
(a,b):=\{x\in X : a< x <b\}\,,
$$
for $a,b\in X$ (see \cite[Th. 3]{frink42}). (Recall that we write $x<y$ for ``$x\le y$ and $x\ne y$").
\begin{remark}
By analogy with the upper topology one can define the \emph{lower topology} $\tau_{\,l}$ as that generated by the basis formed of the complements of the sets $\uparrow\!\! E$ for $E\subseteq X,\, E$ finite. The interval topology $\tau_\le$ is the supremum of these two topologies: $\, \tau_\le=\tau_u\vee\tau_{\,l}$.
 \end{remark}

We mention also the following result.
\begin{theo}[\cite{frink42}, Th. 9]\label{t.Frink} Every complete lattice is compact in its interval topology.
  \end{theo} \begin{proof}
    We include the simple proof of this result. Let $(X,\le)$ be a complete lattice with 0 the least and 1 the greatest element. Then $\uparrow\!\!a =[a,1]$ and $\downarrow\!\!b =[0,b]$, so that the intervals $[a,b],\, a,b\in X$, form a subbase of the interval topology. By  Alexander subbase theorem  (\cite[p. 139]{Kelley} it is sufficient to show that every family $[a_i,b_i],\, i\in I$, of intervals in $X$ having the finite intersection property has nonempty intersection. Since
    $\,[a_i,b_i]\cap [a_j,b_j]\ne\emptyset\iff a_i\vee a_j\le b_i\wedge b_j,\,$ it follows $a_i\le b_j$ for all $i,j\in I$. Hence
    $$
    a:=\sup_{i\in I}a_i\le\inf_{j\in I} b_j=:b\,,$$
    and $\emptyset \ne [a,b]\subseteq\bigcap_{i\in I} [a_i,b_i]$.
  \end{proof}

Frink \cite{frink42} considered also the Moore-Smith order topology defined in the following way. A net $(x_i:i\in I)$ in a poset $(X,\le)$ is said to converge to $x\in X$ if there exist an increasing net $(u_i:i\in I)$ and a decreasing one $(v_i:i\in I)$ such that
\bequ\label{def.o-converg-net}\begin{aligned}
{\rm (i)}\;\; &u_i\le x_i\le v_i\;\mbx{ for all }\;i\in I\;\mbx{ and}\\
{\rm(ii)}\;\; &\sup_iu_i=x=\inf_iv_i\,.\end{aligned}\eequ

 By definition, an element $x\in X$ belongs to the closure $\overline Y$ of  a subset $Y$ of $X$ if and only if there exists a net   in $Y$ that converges to $x$. This closure operation satisfies the conditions
\begin{align*}
  &{\rm (a)}\;\; \overline \emptyset=\emptyset;\quad {\rm (b)}\;\; Y\subseteq \overline Y;\\
   &{\rm (c)}\;\; \overline{Y_1\cup Y_2}=\overline{Y_1}\cup \overline{Y_2}\,,
\end{align*}
for all $Y,Y_1,Y_2\subseteq X$, but not the condition $\overline{\overline{Y}}=\overline{Y}$, so it does not generate a topology, see Kelley \cite[p. 43]{Kelley}. In spite of this fact we call it the \emph{Moore-Smith order topology}. If $(X,\le)$ is totally ordered, then it agrees with the interval topology \cite[Th. 3]{frink42}. If $(X,\le)$ is a distributive lattice, then the lattice operations $\vee$ and $\wedge$ are continuous with respect to the Moore-Smith order topology \cite[Th. 2]{frink42}.

\begin{remark}
Motivated   by  applications to computer science, mainly to   denotational semantics of functional programming languages, topological and categorical methods applied to partially ordered sets were developed. A branch of this is known under the name of continuous lattices, whose study was initiated by Dana Scott \cite{scott72} in 1971. Roughly speaking these are  complete lattices $(X,\le)$ with Scott continuous meet and join operations, which means that
$$
x\wedge \sup D=\sup\{x\wedge d : d\in D\}\quad\mbox{and}\quad x\vee \inf D=\inf\{x\vee d : d\in D\}\,,
$$
for every nonempty directed subset $D$ of $X$.

Another one is the so called domain theory. Essentially it is    concerned with the    study of lattices or of directed complete partially ordered sets (known as \textbf{dcpo}, see Definition \ref{def.posets}) equipped with a $T_0$ topology compatible with the order. A good introduction to this area is given in the book \cite{Larrecq} (which we have partially followed in our presentation), and in the paper \cite{abramsky}. For a comprehensive presentation we recommend the  monograph \cite{Keimel}, see also \cite{Stolt-Hansen}.  Notice also that a functional analysis within the context of $T_0$ topology was recently developed, see for instance    \cite{keimel08,keimel15}. It turned out that a lot of results from Hausdorff functional analysis (Hausdorff topological vector, Hausdorff  locally convex  spaces and Banach spaces) have their analogues  in some algebraic structures  -- vector spaces, cones, universal algebras, etc -- equipped with a compatible $T_0$ topology.
\end{remark}

\section{Fixed points in partially ordered sets}\label{S.fixp-ord}
In this section we shall present some fixed point results in partially ordered sets and their impact on the completeness of the underlying ordered set.
\subsection{Fixed point theorems}
These fixed point theorems bear different names in different publications. The explication is that many mathematicians  contributes to their final shape, and the authors choose one, or several of them, to name a theorem.

Recall that ``poset" is a short-hand for ``partially ordered set".
\begin{theo}[Zermelo]\label{t.fixp-ord-Zermelo} Let $(X,\le)$ be a chain-complete poset and $f:X\to X$ a mapping such that  $x\le f(x)$ for all $x\in X$.

Then $f$ has a fixed point. More precisely, for every $x\in X$ $f$ has a fixed point $y$ above $x$ (i.e. $f(y)=y$ and $x\le y$).

If, further, $f$ is  increasing, then, for every $x\in X$, $f$ has a least fixed point above $x$.
  \end{theo}

  A mapping $f:X\to X$ satisfying $x\le f(x)$ for all $x\in X$ is called \emph{progressive} in  \cite{jachym01}, \emph{inflationary} in \cite{Larrecq}, \emph{extensive} in \cite{klimes84a}.

  This theorem is attributed to Bourbaki-Witt in \cite{Larrecq} (with reference to Bourbaki \cite{bourbaki50} and  Witt \cite{witt50}), to Bourbaki-Kneser in \cite{Zeidler1}.  As it follows  from  the    discussion about this matter in the survey paper by Jachymski \cite{jachym01}, who proposed the name Zermelo FPT, this fixed point theorem appears only implicitly in  Zermelo's papers on well-ordering (from 1904 and 1908), and it was put in evidence later.  Accepting this principle (equivalent to the Axiom of Choice (AC)), the proof is immediate, but there are proofs independent of (AC), see \cite{jachym01}. A brief historical survey is given also in Blanqui  \cite{blanqui14}.
  We shall not enter into this delicate question of whether a specific result depends or not of the (AC). An exhaustive treatment is given in the monographs \cite{Howard-Rub} and \cite{Rubin}. Concerning its relevance for fixed points we recommend the papers by Taskovi\'c \cite{taskovic92,taskovic04,taskovic12} and Ma\'nka \cite{manka87,manka88a,manka88b}. Among other things, Ma\'nka has found a proof of Caristi's fixed point theorem, independent of the (AC).

  \begin{remark}\label{re1.Zermelo}
  In Bourbaki \cite{bourbaki50} Zermelo FPT is formulated for a poset in which every well-ordered subset has a supremum, an apparently stronger form. But as it was shown by Markowski \cite{markow76} these conditions are equivalent: a poset $X$ is chain complete if and only if  every well-ordered subset of $X$ has a supremum. In fact, according to  the comments before  Lemma 1.4 in \cite{smithson71}, this result can be considered as a part of the folklore,
the essential part of the proof -- that every chain contains a well-ordered
cofinal subset -- appears as  exercises in   Halmos'
\emph{Naive set theory}, and in Birkhoff's \emph{Lattice theory}.
  \end{remark}

  Another important result is the following one.
  \begin{theo}[Knaster-Tarski]\label{t.fixp-ord-Kn-Tarski} Let $(X,\le)$ be a poset and $f:X\to X$ an increasing function. If

  {\rm (i)}\; there exists $z\in X$ such that  $z\le f(z)$, and

   {\rm (ii)}\; every chain in $\uparrow\!\! z$ has a supremum,\\
then $f$ has a fixed point above $z$. Furthermore, there exists a maximal fixed point of $f$.
    \end{theo}

    In complete  lattices the above theorem takes the following form.

\begin{theo}[Birkhoff-Tarski]\label{t.Birkh-Tarski}
Let $(X,\leq)$ be a complete lattice and $f:X\to X$ an increasing
mapping. Then there exist a smallest fixed point $\underline x$ and a greatest fixed point $\overline x$
for $f$, given by $\underline x=\inf\{f^n(\top): n\in\mathbb{N}\}$ and by $\overline x=\sup\{f^n(\bot): n\in\mathbb{N}\}$, where $\bot$ denotes the least element of $X$ and $\top$ the greatest one.

Furthermore, the set of fixed points of  the mapping $f$ is a complete lattice.
\end{theo}
\begin{proof} Since $\underline x\le \top$ it follows $f(\underline x)\le f(\top)$. Also $\underline x\le f^n(\top)$ implies $f(\underline x)\le f^{n+1}(\top)$ for all $n\in\mathbb{N}.$ Consequently, $f(\underline x)\le \underline x$. By the definition of $\underline x$, $\,\underline x\le f(\top)$, so that $f(\underline x)=\underline x$.

The case of $\overline x$ can be  treated similarly.
\end{proof}

    In the following theorem one asks a kind of Scott continuity for  the mapping $f$.
\begin{theo}[Tarski-Kantorovich]\label{t.fixp-ord-Kantor-Tarski}
Let $(X,\le)$ be a poset such that  every countable chain in $X$ has a supremum and $f:X\to X$ a mapping  that preserves the suprema of countable chains. If there exists $z\in X$ such that  $z\le f(z)$, then $f$ has a fixed point. Moreover, $z_0:=\sup\{f^n(z) : n\in\mathbb{N}\}$ is the least fixed point of $f$ in $\uparrow\!\! z$.
  \end{theo}\begin{proof}
    We include the simple proof of this result following \cite{Dug-Gran}. Since $f$ preserves suprema of countable chains it follows that it is increasing. From $z\le f(z)$ follows $f(z)\le f^2(z)$ and, by induction, $f^{n-1}(z)\le f^n(z)$ for all $n\in\mathbb{N}$, showing that $\{f^n(z): n\in \mathbb{N}\}$ is a chain in $\uparrow\!\! z$. If $x_0:=\sup\{f^n(z): n\in \mathbb{N}\}$, then, by hypothesis, $f(x_0)=\sup\{f^{n+1}(z) : n\in\mathbb{N}\}=x_0$.

    Let $x_1\ge z$ be a fixed point of $f$. Then $f(z)\le f(x_1)=x_1$ and, by induction $f^n(z)\le x_1$ for all $n\in \mathbb{N},$ that is $x_1$ is an upper bound for $\{f^n(z): n\in \mathbb{N}\}$ and so $x_0\le x_1$.
  \end{proof}

  \begin{remark}
  In Theorem \ref{t.fixp-ord-Kantor-Tarski} it is sufficient to suppose  that every countable chain in $\uparrow\!\! z$  has a supremum and that $f$ preserves these suprema.
  \end{remark}

\subsection{Converse results}
Apparently, the first converse result in this area was obtained by Davis \cite{davis55}.
\begin{theo}\label{t.Davis}
  A lattice $(X,\le)$ is complete if and only if every increasing mapping $f:X\to X$ has a fixed point.
\end{theo}

By a result of Frink \cite{frink42} (see Theorem \ref{t.Frink}), a lattice $(X,\le)$ is complete if and only if it is compact with respect to the interval topology. Consequently Theorem \ref{t.Davis} admits the following reformulation.
\begin{theo}\label{t2.Davis}
  A lattice $(X,\le)$ is compact in its interval topology if and only if every increasing mapping $f:X\to X$ has a fixed point.
\end{theo}

Extensions  to lower semi-lattices of this result   as well as of Birkhoff-Tarski fixed point theorem, Theorem \ref{t.Birkh-Tarski},  were given by Ward \cite{ward57}. Recall that a lower semi-lattice (semi-lattice in short) is a poset $(X,\le)$ such that  $x\wedge y$ exists for every $x,y\in X$. It is called complete if every nonempty subset of $X$ has an infimum.
\begin{theo}[\cite{ward57}]\label{t.Ward} \hfill
\begin{enumerate}
\item A semi-lattice $(X,\le)$ is complete if and only if for every $x\in X$, $\downarrow\!\!x$ is compact with respect to the interval topology.
\item A semi-lattice $(X,\le)$ is compact with respect to the interval topology if and only if every increasing mapping  $f:X\to X$ has a fixed point.
 \end{enumerate} \end{theo}

Smithson \cite{smithson73} extended Davis' results to the case of set-valued mappings.
Wolk \cite{wolk57}  obtained also  characterizations of directed completeness of posets (called by him Dedekind completeness) in terms of fixed points of monotonic maps acting on them.

We mention also the following result of Jachymski \cite{jachym03}, connecting several properties equivalent to FPP. A \emph{periodic point} for a mapping $f:X\to X$ is an element $x_0\in X$ such that $f^k(x_0)=x_0$, for some $k\in \mathbb{N}.$ The set of periodic points is denoted by Per$(f)$ while the set of fixed points is denoted by Fix$(f)$. It is obvious that a fixed point is a periodic point with $k=1$.
 \begin{theo}\label{t.Jachym03}
 Let $X$ be a nonempty abstract set and $f$ be a self-map of $X$. The following
statements are equivalent.
\begin{enumerate}
\item {\rm Per}$(f) =$ {\rm Fix}$(f)\ne \emptyset$.
\item {\rm (Zermelo)} There exists a partial ordering $\preceq$ such that every chain in $(X,\preceq)$
has a supremum and $f$ is progressive with respect to $\preceq$ $($i.e. $x\preceq f(x),\, x\in X)$ .
\item {\rm (Caristi)} There exists a complete metric $d$ and a lower semicontinuous function
$\varphi: X \to \mathbb{R}$ such that $f$ satisfies condition \eqref{eq1.fp-Car}.
\item There exists a complete metric $d$ and a $d$-Lipschitz function $\varphi: X \to \mathbb{R}$  such
that $ f$ satisfies condition  \eqref{eq1.fp-Car} and $f$ is nonexpansive with respect to $d$; i.e.;
$$d(f(x),f(y))\le d(x,y)\mbox{ for all } x, y \in X$$.
\item {\rm  (Hicks-Rhoades)} For each $\alpha\in(0,1)$ there exists a complete metric $d $ such that
$f$ is nonexpansive with respect to $d$ and
$$d(f(x),f^2(x))\le \alpha\, d(x,f(x))\mbox{ for all } x \in X$$.
\item There exists a complete metric $d$ such that $f$ is continuous with respect to $d$
and for each $x\in X$ the sequence $(f^n(x))_{n=1}^\infty$ is convergent (the limit may depend\\
on $x$).\end{enumerate}\end{theo}

For two nonempty sets $A,B$ denote by $B^A$ the  family of  all mappings from $A$ to $B,\,$
$
B^A:=\{f: f:A\to B\}\,.
$

Let $(X,(\rho_i)_{i\in I})$ be a uniform space where    $\{\rho_i : i\in I\}$  is  a family of semi-metrics generating the uniformity of $X$.  Define a partial order $\preceq$ on   $X\times \mathbb{R}_+^I$ by
\begin{equation}\label{def.Jachym-ord}
(x,\varphi)\le (y,\psi) \iff \forall i\in I,\;\; \rho_i(x,y)\le \varphi(i)-\psi(i)\,,
\end{equation}
for $x,y\in X$ and $\varphi,\psi\in \mathbb{R}_+^I$.

If $(X,\rho)$ is a metric space (i.e. $I$ is a singleton and $\rho_1=\rho$ is a metric), then the  relation order \eqref{def.Jachym-ord} becomes
\begin{equation}\label{def2.Jachym-ord}
(x,\alpha)\le (y,\beta) \iff   \rho(x,y)\le \alpha-\beta\,,
\end{equation}
for $x,y\in X$ and $\alpha,\beta\in \mathbb{R}_+$, an order considered by Ekeland in connection with his variational principle.

An ordered set $(X,\le)$ is called \emph{Dedekind $\sigma$-complete} if every bounded increasing sequence has a supremum and  every bounded decreasing sequence has an infimum.  Also one says that a sequence $(x_n)$ in $X$ is \emph{order convergent}  (\emph{o-convergent}) to $x\in X$ if there exists the sequences $(y_n)$ and $(z_n)$ in $X$ such that
\bequ\label{def.o-converg}\begin{aligned}
 {\rm (i)}\;\;
&y_n\le y_{n+1}\le x_n\le z_{n+1}\le z_n\;\mbx{ for all $n\in\Nat$ and }\\{\rm(ii)}\;\; &\sup_ny_n=x=\inf_nz_n\,,
\end{aligned}\eequ
 (i.e., the sequential version of \eqref{def.o-converg-net}).

DeMarr \cite{demarr65} proved the following results concerning the order  \eqref{def2.Jachym-ord}.
\begin{theo}
   Let $(X,\rho)$  be a metric space and $\preceq$ the order on $X\times  \mathbb{R}_+ $ defined by \eqref{def2.Jachym-ord}. Then the following hold.
\ben
\item A  sequence $(x_n)$ in $X$ is convergent to $x\in X$ if and only if the sequence $\big((x_n,0)\big)_{n\in\Nat}$ is o-convergent to $(x,0)$ in $(X\times\Real,\preceq).$
\item The metric space $(X,\rho)$ is complete if and only if the ordered  set $(X\times\Real,\preceq)$  is Dedekind $\sigma$-complete.
\een\end{theo}

He used these results to prove the following fixed point theorem.
\begin{theo}[\cite{demarr65}]   Let $(X,\rho)$  be a complete metric space and $\preceq$ the order on $X\times  \mathbb{R}_+ $ defined by \eqref{def2.Jachym-ord}. If
$f:X\times\Real\to X\times\Real$ is an increasing map for which there exist $\xi_0,\xi_1\in X\times\Real$ such that $\xi_0\le f(\xi_0)\le f(\xi_1)\le \xi_1,$
then $f$ has a fixed point.\end{theo}

Jachymski \cite{jachym98a} extended these  results to uniform spaces and the order  \eqref{def.Jachym-ord}.
\begin{theo}\label{t1.Jachym}
Let $(X,(\rho_i)_{i\in I})$  be a uniform space and $\preceq$ the order on $X\times  \mathbb{R}_+^I$ defined by \eqref{def.Jachym-ord}. Then the following are equivalent.
\begin{enumerate}
\item Every sequence $(x_n)$ in $X$ such that   $\sum_{n=1}^\infty\rho_i(x_n,x_{n+1})<\infty$,   for all $i\in I$, is convergent.
\item Every countable chain in $(X \times\mathbb{R}_+^I,\preceq)$ has a supremum.
\item Every increasing sequence in $(X \times\mathbb{R}_+^I,\preceq)$ has a supremum.
\end{enumerate}

In particular, if the space $X$ is sequentially complete, then each of the above
 conditions holds.
 \end{theo}

 In the case of a metric space $(X,\rho)$ one obtains a characterization of completeness.

 \begin{theo}\label{t2.Jachym}
Let $(X,\rho)$  be a metric space and $\preceq$ the order on $X\times  \mathbb{R}_+ $ defined by \eqref{def2.Jachym-ord}. Then the following are equivalent.
\begin{enumerate}
\item The metric space $X$ is complete.
\item Every   chain in $(X \times\mathbb{R}_+,\preceq)$ has a supremum.
\item Every countable chain in $(X \times\mathbb{R}_+,\preceq)$ has a supremum.
\item Every increasing sequence in $(X \times\mathbb{R}_+,\preceq)$ has a supremum.
\end{enumerate} \end{theo}

Jachymski applied these results to obtain proofs of fixed point results for mappings on partially ordered sets.  In their turn, these order fixed point results  were applied to obtain simpler proofs and extensions to various fixed point results in metric and in uniform spaces, see, for instance, the papers by Jachymski \cite{jachym97,jachym98b,jachym98a,jachym01}, and the references cited therein.

Klime\v s \cite{klimes84a} has found a  common extension to Theorems \ref{t.fixp-ord-Zermelo} and \ref{t.fixp-ord-Kn-Tarski}. Let $(X,\le)$ be partially ordered. A mapping $f:X\to X$  is called \emph{partially isotone} if for all $x,y\in X$
\begin{equation}
 \label{def.part-isot}
\big(x\le y \wedge x\le f(y)\wedge f(x)\le y\big)\;\Rightarrow\; f(x)\le f(y)\,.
 \end{equation}

 It is obvious that increasing mappings, ``progressive" mappings (satisfying $x\le f(x)$) and  ``regressive" mappings (satisfying $f(x)\le x$) are partially isotone.

 The mapping $f$ is called \emph{comparable} if $x$ is comparable with $f(x)$ for every $x\in X$. The partially ordered set $X$ is called inductive  if  every chain   in $X$ has an upper bound, and \emph{semiuniform} if for every chain $C$ in $X$ the set of upper bounds of $C$ is downward directed.

 Klime\v s \cite{klimes84a} proved that:
\begin{itemize}
  \item every relatively isotone self-mapping on a complete lattice has a fixed point (Theorem 1.2);
  \item if the partially ordered set $X$ is chain complete (i.e. every chain in $X$, including the empty chain, has a supremum) then every relatively isotone self-mapping on $X$ has a fixed point (Theorem 1.6);
  \item  a lattice $X$ is complete if and only if every comparable  self-mapping on $X$ has a fixed point (Theorem 2.2);
 \item     a semiuniformly   partially ordered set $X$ is chain complete if and only if every relatively isotone self-mapping on $X$ has a fixed point (Theorem 2.3).
\end{itemize}

In \cite{klimes85} he considered \emph{ascending} maps $f:X\to X$, meaning that $f(x)\le y$ implies
$f(x)\le f(y)$ for all $x,y\in X$, and proved that the partially ordered set $X$ is inductive if and only if every ascending  self-mapping on $X$ has a fixed point. For other related results, see \cite{klimes82} and   \cite{klimes84b}.  For instance, in  \cite{klimes84b} one considers mappings $f:X\to X,\, X$ a partially ordered set, such that $x\le y$ and $x\le f(x)$ implies $f(x)\le f(y)$, called by the author extensively isotone.

\subsection{Fixed points in ordered metric spaces}
The title of this subsection could  be a little confusing - in contrast to ordered Banach spaces, or Banach lattices,    it concerns a metric space $(X,\rho)$  equipped with  an order relation $\le$ that does not have any connection with the metric structure. Fixed points are proved for  mappings $f:X\to X$ which   are monotonic (increasing or decreasing) with respect to the order    and contractive  with respect to the metric, but in a restricted manner in the following sense: there exists $0\le\alpha<1$ such that
\bequ\label{def.ord-contr}
\rho(f(x),f(y))\le\alpha\rho(x,y)\;\;\mbox{if   } x,y\in X\;\;\mbox{are comparable (i.e. }\;
 x\le  y\;\,\mbox{or}\;\, y\le x\mbox{)}\,.
\eequ

\begin{theo}\label{t.fixp-ord-m-sp}
Let $(X,\rho)$ be a complete metric space equipped with a  partial order $\le$ and    $f:X\to X$   a mapping  satisfying \eqref{def.ord-contr}. Then the following results hold.
\begin{enumerate}
  \item {\rm (\cite{nieto05})}  If  the mapping  $f$ is increasing and continuous and there exists
  $x_0\in X$        such that $x_0\le f(x_0)$, then $f$ has a fixed point.
\item {\rm (\cite{nieto05})}    Suppose that  for every increasing sequence $(x_n)$ in $X$ converging to some   $x\in X$ it holds $x_n\le x\,$ for all $n\in\mathbb{N}.$ If   $f$ is increasing and   there
 exists $x_0\in X$   such that $x_0\le f(x_0)$, then $f$ has a fixed point.
    \item {\rm (\cite{ran-reur04})} Suppose that  every pair $x,y$ of elements in $X$ has an upper bound or  a lower bound.
   If $f$ is continuous and monotone (i.e. either increasing or decreasing) and   there exists
        $x_0\in X$ such that $x_0\le f(x_0)$ or $f(x_0)\le x_0$,
           then $f$ has a unique fixed point $\overline{x}$  and for every $x\in X$ the sequence $(f^n(x))_{n\in\mathbb{N}}$ converges to $\overline{x}$.
\item {\rm (\cite{khamsi15})} Assume that the ordered set $(X,\le)$ admits a smallest  element
 $ x_0$. Then the conclusions from {\rm (3)} hold for every continuous increasing
 function $f:X\to X$ satisfying \eqref{def.ord-contr}.
\end{enumerate}
\end{theo}    \begin{proof}
  (1)\; The proof is simple. Since $f$ is increasing
  $$
  x_0\le f(x_0)\;\Rightarrow\;  f(x_0)\le f^2(x_0)   \;\Rightarrow\;  f^2(x_0)\le f^3(x_0) \;\Rightarrow\dots\,$$
  showing that the sequence $(f^n(x_0))$ is increasing. By \eqref{def.ord-contr}
$$
\rho(f^{n}(x_0),f^{n+1}(x_0))\le \alpha  \rho(f^{n-1}(x_0),f^{n}(x_0))\le\dots\le\alpha^n\rho(x_0,f(x_0))\,,
$$
for all $n\in\mathbb{N}$. But then, by the triangle inequality,
$$
\rho(f^{n}(x_0),f^{n+k}(x_0))\le (\alpha^n+\alpha^{n+1}+\dots+\alpha^{n+k-1})d(x_0,f(x_0))
\longrightarrow 0$$
as $n\to \infty$, uniformly with respect to $k\in\mathbb{N}$, which shows that   $(f^n(x_0))$
is a Cauchy sequence and so, by the completeness of the metric space $X$, it converges
to some $\overline x\in X$. By the continuity of $f$,
$$ f(\overline x)=f(\lim_nf^n(x_0))=\lim_nf^{n+1}(x_0) =\overline x\,.$$

(2)\; As in the proof of (1), the sequence $(f^n(x_0))$ is increasing and convergent to some
$\overline x\in X$. By hypothesis, it follows $f^n(x_0)\le \overline x$ for all $n\in\mathbb{N}$, so that,
$$
\rho(f^{n+1}(x_0),f(\overline x))\le\alpha\rho(f^{n}(x_0),\overline x)\to 0\;\mbox{ as }\; n\to\infty\,.$$

It follows $\rho(\overline x,f(\overline x))=0,$ that is
$f(\overline x)=\overline x$.

(3)\; Suppose that $f$ is increasing and that there   exists $x_0\in X$   such that $x_0\le f(x_0)$.
Then $(f^n(x_0))$ is an increasing sequence, convergent to some $\overline x\in X$
which is a fixed point for $f$.
The proof will be done if we show that, for every $x\in X$, the sequence $(f^n(x))$ is
convergent to $ \overline x$.

Let $x\in X$. If $x\le x_0$, then $f^n(x)\le f^n(x_0) $ so that,  by \eqref{def.ord-contr},
$$
\rho(f^n(x),f^n(x_0))\le \alpha   \rho(f^{n-1}(x),f^{n-1}(x_0))\le\dots\le \alpha^n\rho(x,x_0) \to 0\,.
$$

It follows $\lim_nf^n(x)=\lim_nf^n(x_0)=\overline x$. The situation is the same if $x\ge x_0$.

If $x\in X$ is not comparable to $x_0$, then, by hypothesis, $x$ and $x_0$ have a
lower bound  or an upper bound in $(X,\le)$.

 If they have a lower bound $x_1$, then $x_1\le x_0$ and $x_1\le x$, so that by the first part of the proof
 $$
 \overline x =\lim_nf^n(x_0)= \lim_nf^n(x_1)=\lim_nf^n(x) \,.$$

 The situation is the same if $x$ and $x_0$ have an upper bound $x_2$  in $X$.

 (4)\; In this case $x_0\le f(x_0)$ and, for every  $x\in X,\, x_0\le x,$   so we can
 proceed as in the proof of (3).
 \end{proof}

\begin{remark}
  Usually results as those from Theorem \ref{t.fixp-ord-m-sp} are called fixed point of Ran-Reurings type \cite{ran-reur04}. \end{remark}

Refinements of the above results were given in  \cite{jachym11c},  \cite{nieto07b}, \cite{nieto07a} and \cite{petrusel-rus06}.

\section{Partial metric spaces}\label{S.pm-sp}

These spaces were introduced by Matthews \cite{matths92c,matths92a,matths92b,matths94} in connection with his research on computer science. They are only $T_0$ topological spaces, a feature that fits the needs of denotational semantics of dataflow networks. In this section we shall  first present the basic notions and results following \cite{kopper-matths09}, \cite{matths92a, matths92b} and \cite{matths94} (see also the books \cite{Kirk-Shah} and \cite{Rus-PP}). Although all the  included  results on partial metric spaces can be found in the papers of Matthews or in other ones dealing with fixed point results in such spaces, we include  full proofs of the results, for reader's convenience. At  the same time, different  approaches concerning convergence of sequences and completeness notions in partial metric spaces, used by various authors, are put in a  proper light.

\subsection{Definition and topological properties}\label{Ss.pm-sp}

Let $X$ be a nonempty set.
\begin{defi}\label{def.pm}   A mapping $p:X\times X\to \mathbb{R}_+$  satisfying   the following conditions
\begin{align*}
  &{\rm (PM1)}\;\; x=y\iff p(x,x)=p(y,y)=p(x,y)\;\; \mbox{(indistancy implies equality)}; \\
&{\rm (PM2)}\;\;  0\le p(x,x)\le p(x,y)\;\; \mbox{(nonnegativity and small self-distances)};\\
&{\rm (PM3)}\;\; p(x,y)=p(y,x)\;\; \mbox{(symmetry)};\\
&{\rm (PM4)}\;\; p(x,z)\le p(x,y)+p(y,z)-p(y,y)\;\;\mbox{(triangularity)}\,,
\end{align*}
for all $x,y,z\in Z$,
is called a \emph{partial metric} on $X$. The pair $(X,p)$ is called a \emph{partial metric space}.
\end{defi}

This means that, in contrast to the metric   case, one admits the possibility that $d(x,x)>0$ for some points $x\in X,$ a property called ``self-distancy".

A point $x\in X$ is called

\textbullet\; \emph{complete} if $p(x,x)=0$,\; and

\textbullet\; \emph{partial} if $p(x,x)>0$,

\noindent giving an explanation for the term ``partial" coined by Matthews.

The following property follows from (PM2) and (PM1).

\begin{equation}\label{eq1.pm-sp}
p(x,y)=0\;\Longrightarrow\; x=y\;\; \mbox{(indistancy implies equality)}\,.
\end{equation}

The following characterization of partial metric spaces is given by  M. \& V. Anisiu \cite{anisiu16}.
\begin{theo}
  A function $p \colon X \times  X \to [0,\infty)$ is a partial metric on $X$ if and only
if there exist a metric $d$ and a nonexpansive  with respect to
$d$  function $\,\varphi\colon X \to [0,\infty)$, such that
$$p(x,y) = d(x,y) +\varphi(x) + \varphi(y) \quad\mbox{ for all }\quad x,y \in X\,.$$
Furthermore, $d$ and $\varphi$ are uniquely determined by $p$.
\end{theo}

The following two examples of partial metric spaces are related to some questions in theoretical computer science.
\begin{example}\label{ex1.pm-N}
  The function  $p:2^\mathbb{N}\times 2^\mathbb{N}\to[0,\infty)$ defined by
$$ p(x,y)=1-\sum_{n\in x\cap y}2^{-n}\,,$$
with the convention that the sum over the empty set is 0, is a partial metric on $2^\mathbb{N}$.
\end{example}

\begin{example}\label{ex2.pm-N}
  For  a nonempty set  $S$ let   $S^\infty=S^*\cup S^{\mathbb{N}} $ be the set of all finite (belonging to $S^*$) or infinite sequences (belonging to $ S^{\mathbb{N}}$). The length $\ell(x)$ of a finite sequence $x=(x_1,x_2,\dots,x_n)$ is $n$ and the length of an infinite sequence $x:\mathbb{N}\to S$ is $\infty$.
Put  $i(x,y)=\sup\{n\in\mathbb{N} : n\le\ell(x)\wedge\ell(y),\, x_j=y_j,\, \forall j< n\}$, and define
  $$p(x,y)=2^{-i(x,y)},\,\quad x,y\in S^\infty\,,$$
with the convention $2^{-\infty}=0$,   is a partial metric on $S^\infty$.

The function $p$ is a metric on $S^\mathbb{N}$, called the Baire metric, and a partial metric on $S^*\cup S^{\mathbb{N}} $, because    $p(x,x)=2^{-n}>0$ for $x=(x_1,\dots,x_n)\in S^*$.
\end{example}

We define the open balls as in the metric case:
\begin{equation}\label{def.ball-pm}
B_p(x,\varepsilon):=\{y\in X : p(x,y)<\varepsilon\}\,,
\end{equation}
for $x\in S^\infty$ and $\varepsilon >0$.

\begin{remark}\label{re2.pm-sp} In this case the possibility that $B_p(x,\varepsilon)=\emptyset$ is not excluded.

 If $p(x,x)>0$, then
$B_p(x,\varepsilon)=\emptyset$ for every $0<\varepsilon\le p(x,x)$.

If $B_p(x,\varepsilon)\ne\emptyset$, then $x\in B_p(x,\varepsilon)$.
\end{remark}

Indeed, by (PM2), $p(x,y)\ge p(x,x)\ge \varepsilon$ for every $y\in X$ implies $B_p(x,\varepsilon)=\emptyset$.
Also, if $y\in B_p(x,\varepsilon)$, then, again  by (PM2), $p(x,x)\le p(x,y)< \varepsilon,$ i.e.  $x\in B_p(x,\varepsilon)$.

Consider also the balls
 \begin{equation}\label{def.ball2-pm}
 B'_p(x,\varepsilon):=\{y\in X : p(x,y)<\varepsilon +p(x,x)\}\,,\end{equation}
 for $ x\in X$ and $ \varepsilon>0$.  It is obvious that $x\in B'_p(x,\varepsilon).$

The following proposition  contains some  properties of  these two kinds of  balls.
\begin{prop}\label{p1.pm-sp} Let $(X,p)$ be a partial metric space.
\begin{enumerate}
\item
If $y\in B_p(x,\varepsilon)$ then
$$
y\in B_p(y,\delta)\subseteq B_p(x,\varepsilon)\,,$$
where $\delta:=\varepsilon -p(x,y)+p(y,y)>0$.
\item The balls $B_p$ and $B'_p$ are related by the following equalities:
\begin{equation}\label{eq1.ball-pm-sp}
B'_p(x,\varepsilon)=B_p(x,\varepsilon+p(x,x))\,,
\end{equation}
and
$$
B_p(x,\varepsilon)=\begin{cases} B'_p(x,\varepsilon-p(x,x))\quad&\mbox{if}\quad \varepsilon>p(x,x),\\
 \emptyset   &\mbox{if}\quad 0<\varepsilon\le p(x,x).
\end{cases}$$
\item For the balls $B', \, z\in B'_p(z,r)$ for any $z\in X$ and $ r>0.$  Also if $\eps'>0$ and  $y\in B'_p(x,\eps')$, then
$$
B'_p(y,\delta')\sse B'_p(x,\eps')\,,$$
where $\delta':=\eps'+p(x,x)-p(x,y).$
\end{enumerate}
\end{prop}\begin{proof}
(1)\; Let $\delta:=\varepsilon -p(x,y)+p(y,y).$ Then $\delta >0$ (because $p(x,y)<\varepsilon$) and $p(y,y)<\delta$, so that $y\in  B_p(y,\delta)$.

 If $z\in B_p(y,\delta)$, then the inequalities
 $$
 p(y,z)< \varepsilon -p(x,y)+p(y,y)\;\;\mbox{and}\;\; p(x,z)\le p(x,y)+p(y,z)-p(y,y)\,,
 $$
 yield by addition, $p(x,z)<\varepsilon,$ that is $z\in B_p(x,\varepsilon)$ and so $B_p(y,\delta)\subseteq B_p(x,\varepsilon)$.

 The equalities from (2) are obvious by the definitions of the corresponding balls (see also Remark \ref{re2.pm-sp}).

 The inclusion from (3) follows by (1) and (2), but  it can be also proved directly.
 We have
 $$
 y\in B'_p(x,\eps')\iff p(x,y)<\eps'+p(x,x)$$
 and
 $$
 z\in B'_p(y,\delta')\iff p(y,z)<\delta'+p(y,y)\,,$$
 so that
 \begin{align*}
   p(x,z)&\le p(x,y)+p(y,z)-p(y,y)\\
   &<p(x,y)+\eps'+p(x,x)-p(x,y)+p(y,y)-p(y,y)\\
   &=\eps'+p(x,x)\,,
 \end{align*}
 which shows that $z\in B'_p(x,\eps').$
\end{proof}

Now we introduce the topology of a partial-metric space and present some of its properties.
\begin{theo}\label{t1.pm-sp}
 Let $(X,p)$ be a partial metric space.
 \begin{enumerate}
 \item The family of open balls
 \begin{equation}\label{def.basis-pm-top}
\mathcal B:=\left\{B_p(x,\varepsilon),\, x\in X,\, \varepsilon>0\right\}\,
\end{equation}
is the basis of a topology on $X,$ denoted by $\tau_p$ (sometimes by $\tau(p)$).
\item The family $\mathcal B'$ of the sets
 \begin{equation}\label{def.basis2-pm}
 \mathcal B':=\left\{B'_p(x,\varepsilon),\, x\in X,\, \varepsilon>0\right\}\,
 \end{equation}
 is also a basis for the  topology $\tau_p.$
 \item The  balls $B_p(x,\varepsilon),\, B'_p(x,\varepsilon)$ are open and for every $x\in X$ and $\eps>0.$ The family $\mathcal V_p(x)$ of neighborhoods of $x$ is given by
     \begin{equation}\label{eq.neighb-pm-sp}\begin{aligned}
            \mathcal V_p(x)&=\left\{V\subseteq X : \exists \delta >0,\; x\in B_p(x,\delta) \subseteq V\right\}\\
            &=\left\{V\subseteq X : \exists \eps >0,\; B'_p(x,\eps) \subseteq V\right\}\,.
     \end{aligned}     \end{equation}
\item The topology $\tau_p$ is $T_0$ and satisfies the first axiom of countability, i.e., every point in $X$ has a countable basis of neighborhoods..
 \end{enumerate}
\end{theo}\begin{proof} (1)\;
By \cite[Th. 11, p. 47]{Kelley} a family $\rd B\sse 2^X$ is a basis for a topology $\tau$ on a set $X$ if and only if the following two conditions are satisfied:

(B1)\quad $X=\bigcup\rd B$;

(B2)\quad for every $B_1,B_2\in\rd B$ and $x\in B_1\cap B_2$ there exists $B\in \rd B$ such that $$x\in B\sse B_1\cap B_2\,.$$

Since $x\in  B_p(x,1+p(x,x)),$ it follows $X=\bigcup\{ B_p(x,1+p(x,x)): x\in X\}\sse\bigcup\rd B$, so (B1) holds.

Also, by Proposition \ref{p1.pm-sp}.(1),   for any $z\in  B_p(x,\varepsilon_1)\cap  B_p(y,\eps_2)$,
$$\;z\in B_p(z,\delta)\subseteq  B_p(x,\varepsilon_1)\cap  B_p(y,\eps_2)\,,$$
where
 $$
 \delta:=p(z,z)+\min\{\varepsilon_1 -p(x,z),\varepsilon_2 -p(y,z)\}\,,
 $$
 so (B2) is satisfied too.

These two properties show that the family \eqref{def.basis-pm-top} forms a basis of a topology $\tau_p$ on $X$, that is, every set in $\tau_p$ can be written as a union of open balls of the form $B_p(x,\varepsilon)$.

(2)\; Let us show   that $\mathcal B'$ is also a basis. Since $x\in B'_p(x,\eps)$ it follows $X=\bigcup\{ B'_p(x,1): x\in X\}\sse\bigcup\rd B'.$

Also, by Proposition \ref{p1.pm-sp}.(3),
$$
z\in B'_p(z,\delta')\sse B'_p(x_1,\eps_1')\cap  B'_p(x_2,\eps_2')\,,$$
where $\delta'=\min\{\delta'_1,\delta'_2\}$ for $\delta'_i:=\eps'_i+p(x_i,x_i)-p(x_i,z),\,i=1,2.$

The equivalence of two bases $\rd B,\rd B'$ means that
$$
\forall B\in \rd B\;\mbx{ and }\; \forall x\in B,\; \exists B'\in \rd B'\;\mbx{ such that }\; x\in B'\sse B\,,$$
and, conversely,
$$
\forall B'\in \rd B'\;\mbx{ and }\; \forall x\in B',\; \exists B\in \rd B\;\mbx{ such that }\; x\in B\sse B'\,.$$

If  $y\in B_p(x,\eps)$,     then it is easy to check that
$$
y\in B'_p(y,\gamma)\sse B_p(x,\eps)\,,$$
where $\gamma:=\eps-p(x,y).$

Conversely, if $y\in B'_p(x,\eps)$, then $B'_p(x,\eps)=B_p(y,\eps+p(y,y)),$  so that $$y\in B_p(y,\eps+p(y,y))= B'_p(x,\eps).$$

 (3)\; By Proposition \ref{p1.pm-sp}.(1)  for every  $ y\in  B_p(x,\varepsilon)$ we have $y\in B_p(y,\delta_y)\sse B_p(x,\varepsilon),$  where
   where $\delta_y=\varepsilon -p(x,y)+p(y,y)$, showing that $B_p(x,\varepsilon)$ is open. The openness of $B'_p(x,\varepsilon)$ can be obtained in the same  way from
   Proposition \ref{p1.pm-sp}.(3).

  Since the open balls form a basis of the topology $\tau_p,$   $V\in\mathcal V_p(x) $ if and only if there exists
 $y\in X$ and $\varepsilon >0$ such that  $x\in B_p(y,\varepsilon)\subseteq V.$ Appealing again to Proposition \ref{p1.pm-sp}, it follows that  $x\in B_p(x,\delta)\subseteq B_p(y,\varepsilon)\subseteq V$, where
 $\delta =\varepsilon -p(x,y)+p(x,x)$.

  (4)\; We have to show that for any pair $x,y$ of distinct points in $X$ there exists a $\tau_p$-open set containing exactly one of them.

Let $x\ne y$ be two points in $X$. Then by (PM1) and (PM2) either $p(x,x)<p(x,y)$ or $p(y,y)<p(x,y).$

Suppose $p(x,x)<p(x,y)$ and let $\varepsilon :=(p(x,x)+p(x,y))/2.$  Then
$$
2p(x,x)<p(x,x)+p(x,y)=2\varepsilon\;\Rightarrow\; p(x,x)<\varepsilon\iff x\in B_p(x,\varepsilon)\,.$$

On the other side
$$
p(x,y)>p(x,x)=2\varepsilon-p(x,y)\;\Rightarrow\; p(x,y)>\varepsilon\;\Rightarrow\; y\notin B_p(x,\varepsilon)\,.
$$

The case $p(y,y)<p(x,y)$ can be treated similarly.

The family $\{B'_p(x,1/n) : n\in\Nat\}$ is a countable basis of neighborhoods at $x$.
\end{proof}

\begin{remark} We adopt the convention that $\,\bigcup\{A_i: i\in\emptyset\}=\emptyset$ (implying, by de Morgan rules, $\,\bigcap\{A_i: i\in\emptyset\}=X$), and so $\emptyset$ belongs to the family of arbitrary unions of sets in $\mathcal B$.  If one considers only unions over nonempty index sets, then we must say that the family $\mathcal B$ plus the empty set generates the topology $\tau_p$.
\end{remark}

\subsection{Convergent sequences, completeness  and the Contraction Principle}\label{Ss.seq-pm}
The convergence of  sequences with respect to  $\tau_p$ can be characterized in the following way.
\begin{prop}\label{p2.pm-sp}
  Let $(X,p)$ be a partial metric space. A sequence $(x_n)$ in $X$ is $\tau_p$-convergent to $x\in X$ if and only if
  \begin{equation}\label{eq1.seq-pm-sp}
  \lim_{n\to\infty}p(x,x_n)=p(x,x)\,.
  \end{equation}
\end{prop}\begin{proof}
 Suppose that $x_n\xrightarrow{\tau_p}x$. Given $\varepsilon >0$ let $n_0\in\mathbb{N}$ be such that , for all $n\ge n_0,\, x_n\in B_p(x,\varepsilon+p(x,x))\iff p(x,x_n)<\varepsilon+p(x,x)$.  Taking into account (PM2), it follows
 $$
 0\le p(x,x_n)-p(x,x)<\varepsilon\,,
 $$
 for all $n\ge n_0,$ showing that \eqref{eq1.seq-pm-sp} holds.

 Conversely, suppose that \eqref{eq1.seq-pm-sp} holds and let $V\in\mathcal V_p(x).$ Since, by Theorem \ref{t1.pm-sp}.(2), $\mathcal B'$ is also a basis for the topology $\tau_p$, there exists $\varepsilon >0$ such that  $B'_p(x,\varepsilon)\subseteq V$. Let $n_0\in \mathbb{N}$ be such that  $0\le p(x,x_n)-p(x,x)<\varepsilon$ for all $n\ge n_0$.
 Then
 $$
 0\le p(x,x_n)-p(x,x)<\varepsilon\iff p(x,x_n)<\varepsilon+p(x,x)\iff x_n\in B'_p(x,\varepsilon)\subseteq V\,,
 $$
 for all $n\ge n_0$, proving that $x_n\xrightarrow{\tau_p}x$.
 \end{proof}

\begin{remark}\label{re.pm-seq}
 Since the topology $\tau_p$ of a partial metric space is only $T_0$, a convergent sequence can have many limits. In fact, if $x_n\xrightarrow{\tau_p}x$, then $x_n\xrightarrow{\tau_p}y$ for any $y\in X$ such that  $p(x,y)=p(y,y)$.
\end{remark} Indeed
 $$
0\le  p(y,x_n)-p(y,y)\le p(y,x)+p(x,x_n)-p(x,x)-p(y,y)=p(x,x_n)-p(x,x)\longrightarrow 0\,.$$

For a nonempty subset $Y$ of a partial metric space $(X,p)$ and $x\in X$ put
\bequ\label{def.dist-partial metric space}
p(x,Y)=\inf\{p(x,y) : y\in Y\}\,.\eequ

The closure of a subset of $X$  admit the following characterization in terms of this distance function.

\begin{prop}[\cite{altun10}, Lemma 2]\label{p1.cl-dist-partial metric space}
 Let $Y$ be a nonempty subset of a partial metric space $(X,p)$ and $x\in X.$ Then
 \bequ\label{eq1.cl-dist-partial metric space}
 x\in \ov Y\iff p(x,Y)=p(x,x)\,.
\eequ\end{prop}\begin{proof}
Because
$$ p(x,x)\le p(x,y)\;\mbx{ for all }\; y\in X,$$
we have the following equivalences:
\begin{align*}
p(x,x)= p(x,Y)&\iff \forall\eps>0, \, \exists y\in Y,\; p(x,y)<\eps+p(x,x)\\
&\iff \forall\eps>0,\; Y\cap B'_p(x,\eps)\ne\ety\\
&\iff x\in \ov Y\,.\end{align*}
 \end{proof}

 To obtain uniqueness and to define a reasonable notion of completeness, a stronger notion of convergence is needed.
 \begin{defi}\label{def.prop-converg-pm-sp}
   One says that a sequence $(x_n)$ in a partial metric space \emph{converges properly} to $x\in X$ if and only if
   \begin{equation}\label{eq1.prop-converg-pm-sp}
   \lim_{n\to\infty}p(x,x_n)=p(x,x)=\lim_{n\to\infty} p(x_n,x_n)\,.
   \end{equation}
 \end{defi}

 In other words, $(x_n)$  converges properly to $x$ if and only if    $(x_n)$   converges to $x$ with respect to  $\tau_p$ and further
 \begin{equation}\label{eq2.prop-converg-pm-sp}
\lim_{n\to\infty}p(x_n,x_n)=p(x,x)\,.\end{equation}

\begin{prop}\label{p3.pm-sp}
Let $(X,p)$ be a partial metric space and  $(x_n)$  a sequence in $X$ that  converges properly to $x\in X$. Then
\begin{enumerate}
\item[{\rm (i)}]\;\; the limit is unique, and
\item[{\rm (ii)}]\;\; $\lim_{m,n\to\infty}p(x_m,x_n)=p(x,x).$\end{enumerate}
  \end{prop}\begin{proof}
    Suppose that $x,y\in X$ are such that  $(x_n)$ converges properly  to both $x$ and $y$.
    Then
$$
p(x,y)\le p(x,x_n)+p(x_n,y)-p(x_n,x_n) \longrightarrow p(y,y)   \quad\mbox{as}\; n\to \infty\,,
$$
implies $p(x,y)\le p(y,y)$.  By (PM2) this implies  $p(x,y)=p(y,y).$

 Similarly,
 $$
p(x,y)\le p(x,y_n)+p(y_n,y)-p(y_n,y_n) \longrightarrow p(x,x)   \quad\mbox{as}\; n\to \infty\,,
$$
implies $p(x,y)\le p(x,x),$ and so $p(x,y)=p(x,x).$

It follows \begin{equation}\label{eq.uniq-prop-lim}
 p(x,y)=p(x,x)= p(y,y) \,,\end{equation}
so that, by (PM1), $x=y$.

 To prove (ii) observe that
 $$ p(x_m,x_n)\le p(x_m,x)+p(x,x_n)-p(x,x)$$
implying
 $$ p(x_m,x_n)-p(x,x)\le p(x_m,x)+p(x,x_n)-2p(x,x)\longrightarrow 0\;\;\mbox{as}\; m,n\to \infty\,.$$

 Also, by (PM2) and (PM4),
$$
   p(x,x)\le p(x,x_n) \le p(x,x_m)+p(x_m,x_n)-p(x_m,x_m)
$$
so that
 $$
 p(x,x)-p(x_m,x_n)\le p(x,x_m) -p(x_m,x_m) \longrightarrow 0\;\;\mbox{as}\; m,n\to \infty\,.$$

 Consequently, $\,\lim_{m,n\to\infty}p(x_m,x_n)=p(x,x)$.
  \end{proof}
\begin{remark}\label{re.converg-seq-pm-sp}
Some authors take the condition (ii) from Proposition \ref{p3.pm-sp} in the definition of a properly
 convergent sequence. As it was shown this  is equivalent to the condition from Definition \ref{def.prop-converg-pm-sp}
\end{remark}

The definition of Cauchy sequences in partial metric spaces takes the following form.
\begin{defi}\label{def.C-seq-pm-sp}
A sequence $(x_n)$ in a partial metric space $(X,p)$ is called a \emph{Cauchy sequence} if there exists $a\ge 0$ in $\mathbb{R}$ such that  for every $\varepsilon>0$ there exists $n_\varepsilon\in\mathbb{N}$ with
$$
|p(x_n,x_m)-a|<\varepsilon\,,
$$
for all $m,n\ge n_\varepsilon$, written also as $\lim_{m,n\to \infty}p(x_n,x_m)=a$.

The partial metric space $(X,p)$ is called \emph{complete} if every   Cauchy sequence is properly convergent to some $x\in X$.
\end{defi}

A mapping $f$ on a partial metric space $(X,p)$ is called a \emph{contraction} if there exists $0\le\alpha< 1$ such that
\begin{equation}
p(f(x),f(y))\le\alpha\, p(x,y)\,,
\end{equation}
for all $x,y\in X$.

The analogue of Banach Contraction Principle holds in partial metric spaces too.
\begin{theo}[\cite{matths92a},\cite{matths94}]\label{t.Ban-fixp-pm-sp}
Let $(X,p)$ be a complete partial metric spaces. Then every contraction $f:X\to X$ has a fixed point $x_0$ such that  $p(x_0,x_0)=0$.
  \end{theo}\begin{proof} (Sketch) Let $f$ be an $\alpha$-contraction on $X$ with $0\le\alpha<1$.

   One shows first that for every $z\in X$ the sequence of iterates $(f^n(z))$ satisfies the condition
   $$
   \lim_{m,n\to\infty}p(f^n(z),f^m(z))=0\,,$$
   i.e. it is Cauchy. By the completeness of $(X,p)$ there exists $x_0\in X$ such that
   $$
   0=\lim_{n\to\infty}p(f^n(z),f^n(z))=p(x_0,x_0)=\lim_{n\to\infty}p(x_0,f^n(z))\,.$$

   But
\begin{align*}
 0&\le  p(x_0,f(x_0))\le p(x_0,f^n(x_0))+p(f^n(x_0),f(x_0))-p(f^n(x_0),f^n(x_0))\\
  &\le p(x_0,f^n(x_0))+\alpha p(f^{n-1}(x_0),x_0)-p(f^n(x_0),f^n(x_0))\longrightarrow 0\quad\mbox{as}\; n\to \infty\,.
  \end{align*}

  It follows $p(x_0,f(x_0))=0=p(x_0,x_0).$ The relations  $0\le p(f(x_0),f(x_0)) \le \alpha p(x_0,x_0)=0$ imply $p(f(x_0),f(x_0))=0,$ so that, by (PM1), $f(x_0)=x_0$.
   \end{proof}

\begin{remark}\label{re.Ban-fixp-pm-sp}
  O'Neill \cite{oneill96} considered   partial metrics that take values in $\mathbb{R}$ (not in $\mathbb{R}_+$ as in the case of Matthews' partial metric) and related them to domain theory. These kind of spaces are called by some authors \emph{dualistic partial metric space}. The extension of Banach fixed point theorem to this setting was given by Oltra and Valero  \cite{valero04} \,(see also \cite{valero05}). In this case the contraction condition is  given by
  $$
  \exists 0\le\alpha<1\;\;\mbox{such that}\;\: \forall x,y\in X,\quad |p(f(x),f(y))|\le\alpha |p(x,y)|\,.
  $$

  Extensions of various fixed point results from metric spaces to partial metric spaces were given by O. Valero in cooperation with other mathematicians, see \cite{valero12,valero14,valero13,valero15,valero08} (see also \cite{rus08}).
  \end{remark}

\subsection{Topology and order on partial metric spaces}

In this subsection we shall examine the behavior  of the specialization order \eqref{def.top-ord} with respect to  the topology $\tau(p)$ generated by a partial metric $p$.
\begin{prop}\label{p.ord-pm-sp}
Let $(X,p)$ be a partial metric space and $\le_p$ the specialization order on $X$.
\begin{enumerate}
\item The specialization order can be characterized by the following condition
\begin{equation}\label{eq1.ord-pm-sp}
x\le_p y \iff p(x,x)=p(x,y)\,.
\end{equation}
\item The open balls $B_p(x,\varepsilon)$ and $B'_p(x,\eps)$ are upward closed. Consequently every $\tau_p$-open sets is upward closed.
\item The Alexandrov  topology  $\,\tau_a(\le_p)\,$ generated  by $\,\le_p\,$ (see Proposition \ref{p.Alex-top}) is finer than $\tau(p)$.
The equality $\tau(p)=\tau_a(\le_p)$ holds if and only if
\begin{equation}\label{eq2.ord-pm-sp}
\forall x\in X,\; \exists \varepsilon_x >0,\;\; B_p(x,\varepsilon_x)=\uparrow\!\!x\,.
\end{equation} \end{enumerate}
  \end{prop}\begin{proof}
    (1)\; Suppose $x\le_p y$. By definition $x\le_p y\iff x\in\overline{\{y\}}$ and
\begin{align*}
 x\in\overline{\{y\}}&\iff\forall \varepsilon>0,\; \{y\}\cap B'_p(x,\varepsilon)\ne\emptyset\\
 &\iff \forall \varepsilon>0,\;p(x,y)<\varepsilon +p(x,x)\\
&\iff p(x,y)\le p(x,x)\\
&\iff  p(x,y)= p(x,x)\quad\mbx{(by (PM2))}\,.
\end{align*}

(2)\; Let $y\in B_p(x,\varepsilon)$ and $y\le_p z\iff p(y,z)=p(y,y)$. Then
$$
p(x,z)\le p(x,y)+p(y,z)-p(y,y)=p(x,y)<\varepsilon\,,$$
that is $z\in B_p(x,\varepsilon)$.  The proof is similar for $B'_p(x,\varepsilon)$.

If $U\subseteq X$ is $\tau_p$-open and   $x\in U$,  then there exists $\varepsilon_x>0$ such that
$x\in B_p(x,\varepsilon_x)\subseteq U$.   If  $x\le_p y$,
then, since $B_p(x,\varepsilon_x)$ is upward closed,  $y\in B_p(x,\varepsilon_x)\subseteq U$. Consequently $U$ is upward closed.

(3)\; Since the Alexandov topology is the finest such that the induced order specialization agrees with $\le_p$ (Proposition \ref{p.Alex-top}), it follows   $\tau(p)\subseteq \tau_a(\le_p)$.

Suppose now that the condition \eqref{eq2.ord-pm-sp} holds and let $Z\in \tau_a(\le_p)$. Since the $\tau_a(\le_p)$-open sets are upward closed (see Remark \ref{re.Alex-top}), it follows
$$
Z=\bigcup\left\{\uparrow\!\!x : x\in Z\right\}=\bigcup\left\{B_p(x,\varepsilon_x) : x\in Z\right\}\in\tau(p)\,.$$

Consequently, $\tau_a(\le_p)\subseteq \tau(p)$, so that $\tau_a(\le_p)= \tau(p)$.

Conversely, suppose that $\tau_a(\le_p)= \tau(p)$.  Then for every $x\in X,\, \uparrow\!\! x\in\tau(p)$, implying   the existence of $\varepsilon_x >0$ such that  $x\in B_p(x,\varepsilon_x)\subseteq \uparrow\!\!x$.

If  $y\in \uparrow\!\!x$, then $p(x,y)=p(x,x) <\varepsilon_x,$ that is $y\in B_p(x,\varepsilon_x)$, showing that $\uparrow\!\!x\sse $ $B_p(x,\varepsilon_x) $ and
so $B_p(x,\varepsilon_x)= $ $\uparrow\!\!x$.
\end{proof}
\begin{remark}
  In terms of the specialization order $\le_p$ of a partial metric space $(X,p)$,  Remark \ref{re.pm-seq} says in fact that if a sequence $(x_n)$ in $X$ converges to $x\in X$, then it converges to every $y$ with $y\le_p x$. Also, the equalities \eqref{eq.uniq-prop-lim} say that if $(x_n)$ converges properly to $x$ and $y$, then $x\le_p y$ and $y\le_p x$, and so  $x=y$.
\end{remark}

Proposition \ref{p1.top-ord} allows us to characterize the $T_1$ property of partial metric spaces.
\begin{prop}[\cite{zhang17}]\label{p1.sep-partial metric space} Let $(X,p)$ be a partial metric space.
\ben
\item The space $X$ is $T_1$ if and only if \bequ\label{eq.T1-partial metric space}
p(x,x)<p(x,y)\,,\eequ
for every $x\ne y$ in $X$.
\item The following are equivalent: \ben\item[\rm (i)] the space $X$ is $T_2$;\item[\rm (ii)]  for every $x\ne y$ in $X$
\bequ\label{eq.T2-partial metric space}
\inf\{p(x,z)-p(x,x)+p(y,z)-p(y,y) :z\in X\}>0;\eequ
\item[\rm(iii)]every convergent sequence has a unique limit.\een
\een\end{prop}\begin{proof}
  (1)\; By Proposition \ref{p1.top-ord} and the definition \eqref{eq1.ord-pm-sp} of the specialization order $\le_p$ in $(X,p)$,
  the topology $\tau_p$ is $T_1$ if and only if $p(x,y)=p(x,x)$ implies $x=y$, for all $x,y\in X$. This is equivalent to $p(x,x)<p(x,y)$ for all $x\ne y$ in $X$.

  (2)\; (i)\;$\Ra$\;(ii). Suppose that $\tau_p$ is $T_2$ and let $x\ne y$ be two points in $X$. Then there exists $\eps>0$ such that
  $$
  B'_p(x,\eps)\cap  B'_p(y,\eps)=\ety\,.$$

  This means that, for every $z\in X$,
  $$
  p(x,z)\ge \eps+p(x,x)\;\mbx{ or }\;  p(y,z)\ge \eps+p(y,y)\,,$$
  implying $$\inf\{p(x,z)-p(x,x)+p(y,z)-p(y,y) :z\in X\}\ge\eps>0.$$

   (ii)\;$\Ra$\;(iii).  We prove the equivalent implication $\neg$(iii)\;$\Ra\neg$(ii). Suppose that there exist a sequence $(x_n)$ in $X$ and $x\ne  y$ in $X$
   such that
   \begin{align*}
     x_n\to x&\iff 0\le p(x,x_n)-p(x,x)\to 0\;\mbx{ and }\;\\
     x_n\to y&\iff 0\le p(y,x_n)-p(y,y)\to 0\,.
   \end{align*}

   It follows
   \begin{align*}
   0&\le \inf\{p(x,z)-p(x,x)+p(y,z)-p(y,y) :z\in X\}\\ &\le \inf\{p(x,x_n)-p(x,x)+p(y,x_n)-p(y,y) :n\in\Nat\} =0\,,
      \end{align*}
      that is, \eqref{eq.T2-partial metric space} fails.

      (iii)\;$\Ra$\;(i).   We prove again the equivalent implication $\neg$(i)\;$\Ra\neg$(iii).

      If $X$ is not $T_2$, then there exist $x\ne y$ in $X$ such that  $ B'_p(x,1/n)\cap B'_p(y,1/n)\ne\ety$  for every $n\in\Nat.$
      Taking a point $x_n$ in each of these sets, it follows that, for every $n\in\Nat,$
      \begin{align*}
        0&\le p(x,x_n-p(x,x)<\f 1n\; \mbx{ and }\\
        0&\le p(y,x_n-p(y,y)<\f 1n\,,
                     \end{align*}
which shows that  $x_n\to x$ and $x_n\to y.$
\end{proof}

\begin{remark}
  The paper \cite{zhang17} also contains a discussion on the validity of $T_3$ and $T_4$ separation axioms as well as on the second countability and separability of partial metric space.
\end{remark}

\subsection{The specialization order in quasi-metric spaces}
In this subsection we shall describe the specialization order in a quasi-metric space.
\begin{prop}\label{p1.order-qm}
Let $(X,q)$ be a quasi-metric space.
\begin{enumerate}
  \item The specialization order $\le_q$ corresponding to $q$ is given by
  \begin{equation}\label{eq1.ord-qm}
  x\le_q y\iff q(x,y)=0\,.\end{equation}
  \item Every open set is upward closed.
\end{enumerate}
\end{prop}\begin{proof}
(1)\;For  $x,y\in X$,
\begin{align*}
x\le_q y\iff& x\in\overline{\{y\}}
\iff \forall \varepsilon >0,\; y\in B_q(x,\varepsilon)\\
\iff& \forall \varepsilon >0,\; q(x,y)<\varepsilon
\iff q(x,y)=0\,.
\end{align*}

(2)\; Let us show first that an open ball $B_q(x,\varepsilon)$  is upward closed.
Indeed, $y\in B_q(x,\varepsilon)$ and $y\le_q z$ imply
$$
q(x,z)\le q(x,y)+q(y,z)=q(x,y)<\varepsilon\,.$$

Now if $U\subseteq X$ is $\tau_q$-open, then for every $x\in U$ there exists $\varepsilon_x>0$ such that  $B_q(x,\varepsilon_x)\subseteq U$. If $x\le_q y$, then $y\in B_q(x,\varepsilon_x)\subseteq U$.\end{proof}
\begin{remark}
  If $q$ is only a quasi-semimetric (see Definition \ref{def.qm}), then \eqref{eq1.ord-qm} defines only a preorder $\le_q$, which is an order if and only if $q$ is a quasi-metric.
\end{remark}

Indeed
\begin{align*}
  (x\le_q y\wedge y\le_q x)\iff& \big(q(x,y)=0\wedge q(y,x)=0\big)\\\iff& x=y\,.
\end{align*}

 A contraction principle holds in this case too. A mapping $f$ on a quasi-metric space $(X,q)$ is called
 a contraction if there exists $\alpha\in[0,1)$ such that
 \begin{equation}\label{def.contr-qm}
 q(f(x),f(y))\le\alpha\, q(x,y)\,,
 \end{equation}
 for all $x,y\in X$.
 \begin{theo}[Contraction Principle in quasi-metric spaces, \cite{matths92b}]\label{t.contr-qm} Let $(X,q)$  be a quasi-metric space such that  the associated metric space $(X,q^s)$ is complete. Then every contraction on $(X,q)$ has a fixed point.
 \end{theo}

\subsection{Partial metrics and quasi-metrics}
In this subsection we put in evidence some relations between partial metrics and quasi-metrics.
\begin{prop}\label{p1.pm-qm}
Let $(X,p)$ be a partial metric space. Then the mapping $q:X^2\to\mathbb{R}_+$ given by
\begin{equation}\label{eq1.pm-qm}
q(x,y)=p(x,y)-p(x,x),\quad x,y\in X\,,
\end{equation}
is a quasi-metric on $X$. The topology $\tau(p)$ generated by $p$ agrees with the topology $\tau(q)$ generated by $q$ and  the corresponding specialization orders $\le_p$ and $\le_q$ coincide as well.
  \end{prop}\begin{proof}
     It is a routine verification to show that the mapping $q$ defined by \eqref{eq1.pm-qm} is a quasi-metric on $X$.

For $0<\varepsilon\le p(x,x),\, B_p(x,\varepsilon)=\emptyset\in\tau(q)$. If $ \varepsilon>p(x,x)$, then
$B_p(x,\varepsilon)=B_q(x,\varepsilon-p(x,x))\in\tau(q),$ relations that imply $\tau(p)\subseteq \tau(q)$.

Since, for every $\varepsilon >0,\, B_q(x,\varepsilon)=B_p(x,\varepsilon+p(x,x))\in\tau(p),$  it follows   $\tau(q)\subseteq \tau(p)$.

Taking into account  \eqref{eq1.ord-pm-sp}
\begin{align*}
x\le_py\iff& p(x,y)=p(x,x)\iff q(x,y)=0\iff \forall \varepsilon >0,\; y\in B_q(x,\varepsilon)\\
\iff& x\in\overline{\{y\}}^q\iff x\le_q y\,.
  \end{align*}
  \end{proof}
 \begin{remark}
  It follows that
  \begin{equation}\label{def.m-pm}
  q^s(x,y)=q(x,y)+q(y,x)=2p(x,y)-p(x,x)-p(y,y),\, x,y\in X\,,
  \end{equation}
  is a metric on $X$, called the associate metric to the partial metric $p$.
 \end{remark}

Indeed, writing $p^s$ as
\begin{equation}\label{def2.m-pm}
  q^s(x,y)=p(x,y)-p(x,x)+p(x,y)-p(y,y),\, x,y\in X\,,
\end{equation}
and taking into account (PM2), it follows $q^s(x,y)\ge 0.$  Also, $q^s(x,y)=0$ implies $p(x,x)=p(x,y)=p(y,y),$ so that, by (PM1), $x=y$.
The symmetry of $q^s$ is obvious and the triangle inequality  $q^s(x,z)\le q^s(x,y)+q^s(y,z)$ follows from (PM4).

 The next result shows that the completeness of the partial metric space $(X,p)$ is equivalent to the completeness of the   associate metric space $(X,q^s)$.
 \begin{prop}\label{p2.pm-qm} Let $(X,p)$ be a partial metric space and $q^s$ the associated metric to $p$ given by \eqref{def.m-pm}.
 \begin{enumerate}
 \item \;{\rm(\cite{oneill95})} \; The following inequality
\begin{equation}\label{eq1.p2.pm-qm}
 |p(x,y)-p(z,w)|\le q^s(x,z)+q^s(y,w)\,,
  \end{equation}
is true for all $x,y,z,w\in X$.
 \item The convergence and completeness properties of the spaces $(X,p)$ and $(X,q^s)$ are related in the following way:

 {\rm (i)}\;\;\; a sequence $(x_n)$ in $X$ is properly convergent to $x\in X$ if and only if  $x_n\xrightarrow{q^s}x$;

 {\rm (ii)}\;\; a sequence $(x_n)$ in $X$ is $p$-Cauchy if and only if  it is $q^s$-Cauchy;

  {\rm (iii)}\, the partial metric space $(X,p)$ is complete if and only if  the associated  metric space $(X,q^s)$ is complete.

\item The mapping   $p(\cdot,\cdot)$ is continuous on $X\times X$ with respect to   the metric $d((x,y),(z,w))=q^s(x,z)+q^s(y,w)$ generating the product topology $\tau(q^s)\times\tau(q^s)$ on $X\times X$.

In particular, the mapping $\beta:X\to [0,\infty)$, given by $\beta(x)=p(x,x),\, x\in X,$ is $q^s$-continuous.
  \end{enumerate}
   \end{prop}\begin{proof} (1)\; By the triangle inequality (PM4)
\begin{align*}
  p(x,y)&\le p(x,z)+p(z,y)-p(z,z)\\&\le p(x,z)+p(z,w)+p(w,y)-p(w,w)-p(z,z)\,,\\
 \end{align*}
so that
\begin{align*}
 p(x,y)- p(z,w)&\le p(x,z)+\underbrace{p(x,z)-p(x,x)}_{\ge 0}-p(z,z)+\\
 &+p(y,w)+\underbrace{p(y,w)-p(y,y)}_{\ge 0}-p(w,w)\\
 &=q^s(x,z)+q^s(y,w)\,.
\end{align*}

Because
$$  p(z,w)-p(x,y)\le q^s(z,x)+q^s(w,y)=q^s(x,z)+q^s(y,w)\,,$$
  the inequality \eqref{eq1.p2.pm-qm} follows.\smallskip

   (2).(i)\; By definition
 $$
 x_n\xrightarrow{q^s}x\iff p(x_n,x)-p(x,x)+p(x_n,x)-p(x_n,x_n)\longrightarrow 0.$$

Since $p(x_n,x)-p(x,x)\ge 0$ and $p(x_n,x)-p(x_n,x_n)\ge 0$ the last condition from above is equivalent to
$$
\begin{cases}
 p(x_n,x)\longrightarrow p(x,x)\quad \\
  p(x_n,x)-p(x_n,x_n)\longrightarrow 0
\end{cases}\iff\;\;
\begin{cases}
 p(x_n,x)\longrightarrow p(x,x) \\
  p(x_n,x_n)\longrightarrow p(x,x)
\end{cases}$$
that is, to the fact that $(x_n)$ converges properly to $x$.
\smallskip

(ii)\; I.\;\emph{Any} $p$-\emph{Cauchy sequence is } $q^s$-\emph{Cauchy}.
\smallskip

   Let $(x_n)$ be a $p$-Cauchy sequence in $X$, that is
   $$
   \lim_{m,n\to \infty}p(x_m,x_n)=a\,,
      $$
       for some $a\in\mathbb{R}_+\,.$
Then $\lim_{k\to\infty}p(x_k,x_k)=a$, so that
$$
q^s(x_m,x_n)=2\,p(x_m,x_n) -p(x_n,x_n)  -p(x_m,x_m)\longrightarrow 0\quad \mbox{as}\;\; m,n\to\infty\,,
$$
which shows that the sequence $(x_n)$ is $q^s$-Cauchy.

 \vspace{2mm}

 II.\; \emph{Any} $q^s$-\emph{Cauchy sequence is } $p$-\emph{Cauchy}.

Let $(x_n)$ be a $q^s$-Cauchy sequence in $X$. We   show now that the net $(p(x_m,x_n))_{(m,n)\in \mathbb{N}^2}$ is Cauchy in $\mathbb{R}_+\,$. By \eqref{eq1.p2.pm-qm},
\bequ\label{ineq.p2.pm-qm}
|p(x_m,x_n)-p(x_{m'},x_{n'})|\le q^s(x_m,x_{m'})+q^s(x_n,x_{n'})\,.\eequ

Given  $\varepsilon >0$ let $n_0\in \Nat$ be such that
\beqs
  q^s(x_m,x_{m'})<\eps/2\quad\mbx{and}\quad q^s(x_n,x_{n'})<\eps/2\,,
\eeqs
for all $m,m'\ge n_0$ and all $n,n'\ge n_0.$  But then, the inequality \eqref{ineq.p2.pm-qm} yields
$$
|p(x_m,x_n)-p(x_{m'},x_{n'})|<\eps\,,$$
for all $m,m',n,n'\ge n_0$.

Since  the net $(p(x_m,x_n))_{(m,n)\in \mathbb{N}^2}$ is Cauchy in $\mathbb{R}_+\,,$  it converges to some $a\in\mathbb{R}_+$, which means that the sequence $(x_n)$ is $p$-Cauchy.
\vspace{2mm}

(iii)\;  This follows from the definition of the completeness of the partial metric space $(X,p)$ and from (i) and (ii).

(3)\; If $q^s(x_n,x)\to 0$ and $q^s(y_n,y)\to 0$, then, by \eqref{eq1.p2.pm-qm},
$$
|p(x_n,y_n)-p(x,y)|\le q^s(x_n,x)+q^s(y_n,y)\longrightarrow 0\,,$$
proving the $q^s$-continuity of $p(\cdot,\cdot)$ at $(x,y)\in X\times X.$
\end{proof}

\begin{remark}\label{re.C-seq-pm-sp}
Definition  \ref{def.C-seq-pm-sp} of a Cauchy sequence in a partial metric space is taken from \cite{matths94} (see also  \cite{kopper-matths09}). In \cite{matths92a} the following equivalent definition is proposed: a sequence $(x_n)$ in a partial metric space $(X,p)$ is called a Cauchy sequence if  for every $\varepsilon>0$ there exists $n_\varepsilon\in\mathbb{N}$ such that
$$
0\le p(x_n,x_m)-p(x_m,x_m)<\varepsilon\,,
$$
for all $m,n\ge n_\varepsilon$.
\end{remark}

 Indeed, the equality \eqref{def2.m-pm} shows that this is equivalent to the fact that $(x_n)$ is  $q^s$-Cauchy, which in its turn is equivalent to the fact that $(x_n)$ is $p$-Cauchy.

 \begin{remark} Another metric on a partial metric space $(X,p)$ is given by $d(x,y)=0$ if $x=y$ and $d(x,y)=p(x,y)$ for $x\ne y$. In this case $\tau_{q^s}\subseteq \tau_d$ and  the metric space $(X,d)$ is complete if and only if
 the partial metric space $(X,p)$ is complete. This result can be used to show that some fixed points results in partial metric spaces can be obtained directly from their analogues in the metric case, see \cite{rezap13}.
A similar situation  occurs  in the case of the so called cone-metric spaces,    see, for instance, the survey paper \cite{radenov11}.
\end{remark}

 \begin{example}\label{ex3.pm-N} Some topological properties of the partial metric spaces   $2^\mathbb{N}$ and $S^\infty$ from  Examples \ref{ex1.pm-N} and \ref{ex2.pm-N} are examined in \cite{oneill95}.

 In the case of the partial metric $(2^\mathbb{N},p)$, the associated quasi-metric is
   $$q(x,y)=\sum_{k\in x\setminus y}2^{-k}\,,$$
 and the associated metric is
   $$q^s(x,y)=\sum_{k\in x\Delta y}2^{-k}\,,$$
   where $x\Delta y=(x\setminus y)\cup (y\setminus x)$ is the symmetric difference of the sets $x,y\subseteq\mathbb{N}$.

   It is shown in \cite{oneill95} that:
   \begin{itemize}
     \item The spaces $2^\mathbb{N}$ and $S^\infty$ are complete.
     \item The associated metric space $(2^\mathbb{N},q^s)$ is compact (and so separable).
 \item The associated metric space $(S^\infty,q^s)$ is  separable if and only if the set $S$ is countable.
 \item The associated metric space $(S^\infty,q^s)$ is  compact if and only if the set $S$ is finite.
   \end{itemize}
 \end{example}

\subsection{The existence of suprema in partial metric spaces}
In this subsection we shall prove that every increasing sequence in a partial metric space has a supremum and it is properly convergent to its supremum. We agree to call a mapping $f:(X_1,p_1)\to (X_2,p_2)$ \emph{properly continuous} if $(f(x_n))$ properly converges to $f(x)$ for every sequence $(x_n)$ in $X_1$ properly convergent to $x$.
\begin{prop}\label{p.sup-pm}
Let $(X,p)$ be a partial metric space and $\le_p$ the specialization order corresponding to $p$.
\begin{enumerate}
\item If $(X,p)$ is complete, then every increasing sequence $x_1\le_px_2\le_p\dots$ in $X$ has a supremum $x$ and the sequence $(x_n)$ converges properly to $x$.
\item    Let   $ (X_1,p_1), (X_2,p_2)$ be complete partial metric spaces with the specialization orders $\le_1,\,\le_2,$  respectively, and $f:(X_1,p_1)\to (X_2,p_2)$ a mapping. If $f$ is properly continuous and monotonic, then $f$ preserves suprema of increasing sequences, that is, $\sup_nf(x_n)= f(x)$  for every increasing sequence $x_1\le_{1}  x_2\le_{1}\dots$ in $X_1$ with $\sup_nx_n=x$
  \end{enumerate}  \end{prop}\begin{proof}
      (1)\; We show first that the sequence $(x_n)$ is Cauchy. Indeed,
      $$
      x_n\le_px_{n+k}\iff p(x_n,x_{n+k})-p(x_n,x_n)=0\,,
      $$
      so that, taking into account Remark \ref{re.C-seq-pm-sp}, it follows that $(x_n)$ is Cauchy.   The completeness hypothesis implies the existence of $x\in X$  such that  the sequence $(x_n)$ is  properly convergent to $x$, that is
      \begin{equation}\label{eq1.sup-pm}
      \lim_np(x,x_n)=p(x,x)=\lim_np(x_n,x_n)\,.
      \end{equation}

      We show that $x=\sup_nx_n$, that is
      \begin{equation}\label{eq2.sup-pm}
      \begin{aligned}
      &{\rm (i)}\quad x_n\le x\;\mbox{for all  } n\in\mathbb{N};\\
      &{\rm (ii)}\quad \mbox{if}\;\;  x_n\le y\;\mbox{for all  } n\in\mathbb{N},\; \mbox{then}\; x\le y\,.
      \end{aligned}\end{equation}

We have for all $n,k\in\mathbb{N}$
\begin{align*}
  p(x_n,x)\le &p(x_n,x_{n+k})+p(x_{n+k},x)-p(x_{n+k},x_{n+k})\\
 = &p(x_n,x_{n})+p(x_{n+k},x)-p(x_{n+k},x_{n+k})\,.
  \end{align*}

  Letting $k\to \infty$ and taking into account \eqref{eq1.sup-pm}, one obtains $p(x_n,x)\le p(x_n,x_n)$, so that, by (PM2) from Definition \ref{def.pm},   $p(x_n,x)= p(x_n,x_n)$,   that is $x_n\le_p x$.

  Suppose now that $x_n\le_p y$  for all $n\in\mathbb{N}$. Then
  \begin{align*}
    p(x,y)\le &p(x,x_n)+p(x_n,y)-p(x_n,x_n)\\
    = &p(x_n,y)=p(x_n,x_n)\,,
      \end{align*}
for all $n\in\mathbb{N}$. Letting $n\to \infty $ one obtains (by \eqref{eq1.sup-pm}, $p(x,y)\le p(x,x)$.
It follows $p(x,y) = p(x,x)$, that is $x\le_p y$. Consequently, both conditions (i) and (ii) from \eqref{eq2.sup-pm} hold.

(2)\; Let  $x_1\le_{1}  x_2\le_{1}\dots$ be an increasing sequence in $X$ with $ \sup_nx_n= x$. Then $(x_n)$ is $p_1$-properly convergent to $x$. Then the sequence $(f(x_n))$ is $\le_2$-increasing and properly convergent to $f(x)$. By (1), this implies that $\sup_nf(x_n)=f(x)$.
\end{proof}

\begin{remark}
  It is possible that the property from the first statement of Proposition \ref{p.sup-pm} characterizes the completeness of the partial metric space $(X,p)$ (like in Theorem \ref{t2.Jachym}). Concerning the second statement, I don't know whether the Scott continuity is equivalent to the continuity of the mapping $f$.
\end{remark}

 \subsection{Caristi's fixed point theorem and   completeness in partial metric spaces}
 In this subsection we shall present, following Romaguera \cite{romag10} the equivalence of Caristi's fixed point theorem   to the completeness of the underlying partial metric space.

Let $(X,p)$ be a partial metric space. Recall the Caristi condition for a mapping $f:X\to X$:
\begin{equation}\tag{Car$_\varphi$}\label{def.Caristi-pm}
p(x,f(x))\le \varphi(x)-\varphi(f(x))\,,
\end{equation}
for all $x\in X$. Here $\varphi $ is a function $\varphi:X\to \mathbb{R}$. According to the continuity properties of the function $\varphi$ we distinct two kinds of Caristi conditions. One says that the mapping  $f$ is\\

\textbullet\; $p$-Caristi if \eqref{def.Caristi-pm} holds for some $p$-lsc bounded from below function $\varphi:X\to\mathbb{R}$;

\textbullet\; $q^s$-Caristi if \eqref{def.Caristi-pm} holds for some $q^s$-lsc bounded from below function $\varphi:X\to\mathbb{R}$,\\

\noindent where $q^s$ is the metric associated to $p$ by \eqref{def.m-pm}.

As it was shown in \cite{romag10} the completeness of a partial metric space $(X,p)$  cannot be characterized by the existence of fixed points of   $p$-Caristi mappings.
\begin{example}
Consider  the  set $\mathbb{N}$ with the partial metric $p(m,n)=\max\{m^{-1},n^{-1}\}$. The associated metric $q^s$  is given by $q^s(m,n)=|m^{-1}-n^{-1}|,\, m,n\in\mathbb{N}$. If $0<\varepsilon<\left[n(n+1)\right]^{-1}$, then $B_{q^s}(n,\varepsilon)=\{n\}$, that is the topology $\tau(q^s)$ is the discrete metric on $\mathbb{N}$, and so the only convergent sequences are the ultimately constant ones. The space $(\mathbb{N},q^s)$ is not complete because the sequence $x_n=n,\,n\in\mathbb{N},$ is $q^s$-Cauchy and not $q^s$-convergent. On the other side there are no $p$-Caristi maps on $\mathbb{N}$.
\end{example}

To obtain a characterization of this kind, another notion is needed.
\begin{defi}\label{def.0-compl-pm}
Let $(X,p)$ be a partial metric space. A sequence $(x_n)$ in $X$ is called 0-\emph{Cauchy} if and only if $\lim_{m,n\to\infty}p(x_m,x_n)=0$. The partial metric space $(X,p)$ is called 0-\emph{complete} if every 0-Cauchy sequence $(x_n)$  is convergent with respect to  $\tau_p$ to some $x\in X$ such that  $p(x,x)=0.$
 \end{defi}

 \begin{remark}\label{re.0-compl-pm}
   The above definition is given in  \cite{romag10}. Taking into account Proposition \ref{p2.pm-qm}, the following  assertions  hold:
   \begin{align*}
      &\begin{cases}
     \lim_{m,n\to\infty} p(x_n,x_m)=0,\\
      \lim_{n\to\infty} p(x,x_n)=p(x,x),\\
      p(x,x)=0,
   \end{cases}\iff
     \begin{cases}
     \lim_{m,n\to\infty} p(x_n,x_m)=0,\\
      \lim_{n\to\infty} p(x,x_n)=p(x,x),\\
      p(x,x)=0=\lim_{n\to\infty}p(x_n,x_n),
   \end{cases}\\
   &\Rightarrow \begin{cases}
     (x_n)\;\mbox{is } q^s\mbox{-Cauchy, and}\\
   x_n\xrightarrow{q^s} x\,.  \end{cases}
   \end{align*}

   Consequently,  a partial metric space $(X,p)$ is 0-complete if and only if  every 0-Cauchy sequence is properly convergent if and only if
   every 0-Cauchy sequence is $q^s$-convergent.
   \end{remark}
   \begin{remark}\label{re2.0-compl-pm}
   It is obvious that a complete partial metric space is 0-complete, but the converse is not true (see \cite{romag10}).\end{remark}

   Notice also the following property.
   \begin{remark}[\cite{karapin11b}]\label{re.0-converg-pm}
   Let $(X,p)$ be a partial metric space, $(x_n)$  a sequence in $X$ and $x\in X$. If $\lim_{n\to \infty}p(x_n,x)=0 $, then    $\lim_{n\to \infty}p(x_n,y)=p(x,y)$ for every $y\in Y$.
   \end{remark}

Indeed,
$$
p(x_n,y)\le p(x_n,x)+p(x,y)-p(x,x)\le p(x_n,x)+p(x,y)\,,
$$
implies $p(x_n,y)-p(x,y)\le p(x_n,x)$,
while
$$
p(x,y)\le p(x,x_n)+p(x_n,y)-p(x_n,x_n)\le p(x,x_n)+p(x_n,y)\,,
$$
implies $p(x,y)-p(x_n,y)\le p(x,x_n)$.

Consequently
$$|p(x,y)-p(x_n,y)|\le p(x,x_n)\longrightarrow 0\,.$$

 The characterization result is the following one.
 \begin{theo}[\cite{romag10}, \cite{Kirk-Shah}]\label{t1.Caristi-pm}
 Let $(X,p)$ be a partial metric space. Then $(X,p)$ is $0$-complete if and only if  every $q^s$-Caristi mapping on $X$ has a fixed point.
 \end{theo}\begin{proof}
 Suppose that $(X,p)$ is 0-complete and let $f:X\to X$ be a $q^s$-Caristi mapping for some $q^s$-lsc bounded for below function $\varphi:X\to\mathbb{R}$. For $x\in X$ let
 $$
 A_x:=\{y\in X : p(x,y)+\varphi(y)\le\varphi(x)\}\,.$$

 Then, by \eqref{def.Caristi-pm}, $f(x)\in A_x$ and $A_x$ is $q^s$-closed because, taking into account Proposition \ref{p2.pm-qm}, the mapping $p(x,\cdot)+\varphi(\cdot)$ is $q^s$-lsc.

 Starting with an arbitrary $x_0\in X$ we shall construct inductively a sequence of $q^s$-closed sets $A_{x_n}$ such that, for all $k\in\mathbb{N},$
 \begin{equation}\label{eq1.Caristi-pm}\begin{aligned}
  {\rm (i)}\;\;  &x_k\in A_{x_{k-1}}\quad\mbox{and}\quad  A_{x_{k}}\subseteq A_{x_{k-1}} \\
 {\rm (ii)}\;\;  &p(x_k,x)<\frac1{2^k}\;\;\mbox{for all}\;\; x\in A_{x_{k}}\,.
 \end{aligned} \end{equation}

 Suppose that $x_k $ and $A_{x_{k}},\, k=0,1,\dots,n,$  satisfy the conditions \eqref{eq1.Caristi-pm}. Choose $x_{n+1}\in A_{x_{n}}$ such that
 $$
 \varphi(x_{n+1})<\inf\varphi\left(A_{x_{n}}\right)+\frac1{2^{n+1}}\,.$$

 If $y\in A_{x_{n+1}},$ then
\begin{align*}
 p(x_n,y) \le&p(x_n,x_{n+1}) +p(x_{n+1},y)-p(x_{n+1},x_{n+1})\\
 \le&\varphi(x_n)-\varphi(x_{n+1})+\varphi(x_{n+1})-\varphi(y)-p(x_{n+1},x_{n+1})\le \varphi(x_n)-\varphi(y)\,,
\end{align*}
which shows that $y\in A_{x_n}$, and so  $A_{x_{n+1}}\subseteq A_{x_{n}}$.

For $x\in A_{x_{n+1}}\subseteq A_{x_{n}},$
\begin{align*}
  p(x_{n+1},x)\le& \varphi(x_{n+1})-\varphi(x)\le \inf\varphi\left(A_{x_n}\right)+\frac1{2^{n+1}}-\varphi(x)\\
  \le&\varphi(x)+\frac1{2^{n+1}}-\varphi(x)=\frac1{2^{n+1}}\,.
\end{align*}

For $m>n,\, x_m\in A_{x_{m-1}}\subseteq A_{x_n},$ so that $p(x_n,x_m)<1/2^n,$ showing that the sequence $(x_n)$ is 0-Cauchy. It follows that there exists $z\in X$ with $p(z,z)=0$ such that
$$
\lim_n p(x_n,z)=0\,.$$

By Remark  \ref{re.0-compl-pm}, $x_n\xrightarrow{q^s}z$. Since each set $A_{x_n}$ is $q^s$-closed and $x_{n+k}\in A_{x_{n+k-1}}\subseteq A_{x_n}$ for all $k\in \mathbb{N},$ it follows $z\in A_{x_n}$, for all $n\in \mathbb{N}.$

Also, the inequalities
\begin{align*}
  p(x_n,f(z))\le& p(x_n,z)+p(z,f(z))\le \varphi(x_n)-\varphi(z)+\varphi(z)-\varphi(f(z))\\
  \le&\varphi(x_n)-\varphi(f(z))\,,
\end{align*}
show that $f(z)\in \bigcap_{n=1}^\infty A_{x_n}$. Consequently, $p(x_n,f(z))<1/2^n$ and, by the $q^s$-lsc of $p(\cdot,f(z)),$
$$
0\le p(z,f(z))\le \liminf _n p(x_n,f(z))\le \lim_n1/2^n=0\,,
$$
so that $p(z,f(z))=0$. From
$$
p(f(z),f(z))\le p(f(z),z)+p(z,f(z))-p(z,z)=0\,,$$
follows
$$
p(z,f(z))=p(z,z)=p(f(z),f(z))=0\,,$$
which implies $f(z)=z$.

To prove the converse, suppose that the partial metric space $(X,p)$ is not 0-complete. Then there exists a 0-Cauchy sequence $(x_n)_{n=0}^\infty$ that is not properly convergent in $(X,p)$. Passing, if necessary, to  a subsequence we can suppose further that the points $x_n$ are pairwise distinct and
\begin{equation}\label{eq2.Caristi-pm}
p(x_n,x_{n+1})<\frac1{2^{n+1}}\quad\mbox{for all   } n\in\mathbb{N}_0:=\mathbb{N}\cup\{0\}\,.
\end{equation}

Let
$$
A:=\{x_n: n\in\mathbb{N}_0\}\,.
$$

By Proposition \ref{p2.pm-qm} the sequence $(x_n)$ is $q^s$-Cauchy and not $q^s$-convergent, so it has no limit points, implying  that the set $A$ is $q^s$-closed.

Consider the functions $f:X\to X$ and  $\varphi:X\to[0,\infty)$ given by
\begin{equation}\label{eq.Caristi-pm}
f(x)=
\begin{cases}
  x_0 \;\; &\mbox{for  } x\in X\smallsetminus A,\\
  x_{n+1}  &\mbox{for  } x=x_n,\, n\in\mathbb{N}_0\,,
\end{cases}\quad
\mbox{and}\quad
\varphi(x)=
\begin{cases}
  p(x_0,x)+1 \;\; &\mbox{for  } x\in X\smallsetminus A,\\
  1/2^n &\mbox{for  } x=x_n,\, n\in\mathbb{N}_0\,.
\end{cases}\end{equation}

It is obvious that $f$ has no fixed points.

I.\; \emph{The function $\varphi$ is $q^s$-lsc.}

Let $(y_n)$ be a sequence in $X$ $q^s$-convergent to some $y\in X$.

If $y\in X\smallsetminus A,$ then there exists $n_0\in\mathbb{N}$ such that  $y_n\in X\smallsetminus A$ for all $n\ge n_0$ (because the set $X\setminus A$ is $q^s$-open). Since $p(x_0,\cdot)$ is $q^s$-continuous  (Proposition \ref{p2.pm-qm}), it follows $\varphi(y)=\lim_n\varphi(y_n)$.

Suppose now that  $y=x_k$ for some $k\in \mathbb{N}_0$ and
denote by  $(y_{m_j})_{j\in\mathbb{N}},\; m_1<m_2<\dots,$ the  terms of the sequence $(y_n)$ that  belong to $A$.  If the set $\{m_j:j\in \mathbb{N}\}$  is infinite,  then we must have $y_{m_j}=x_k,\, j\ge j_0,$
 for some  $j_0\in\mathbb{N}$ (because $(x_n)$ has no convergent subsequences). Since $\varphi(x)\ge 1\ge 2^{-k}$ for $x\in X\setminus A$, it follows   $\inf\{\varphi(x_i) : i\ge n\}=\varphi(x_k)$ for all n, and so $\liminf_n\varphi(y_{n})=\sup_n\inf\{\varphi(x_i) : i\ge n\}=\varphi(x_k)$.

If the set $\{m_j : j\in\mathbb{N}\}$ is finite, then there exists $n_0\in\mathbb{N}$ such that $y_n\in X\setminus A$ for all $n\ge n_0$. This implies
\begin{equation}
\varphi(x_k)=\frac1{2^k}\le 1\le \lim_n[p(x_0,y_{n})+1]=\lim_n\varphi(y_{n})\,.
\end{equation}

Consequently $\varphi(y)\le\liminf_n\varphi(y_n)$ in both cases.

II. \; \emph{$f$ is a Caristi mapping with respect to  $\varphi$}.

Indeed, if $x\in X\smallsetminus A$, then $f(x)=x_0$ and
\begin{align*}
  p(x,f(x))= p(x,x_0)
  =\varphi(x)-1=\varphi(x)-\varphi(f(x))\,.
\end{align*}

If $x=x_k$ for some $k\in\mathbb{N}_0,$ then $f(x_k)=x_{k+1}$ and  , by \eqref{eq2.Caristi-pm},
\begin{align*}
  p(x_k,f(x_k))=& p(x_k,x_{k+1})<\frac1{2^{k+1}}\\=&\frac1{2^{k}}-\frac1{2^{k+1}}=\varphi(x_k)-\varphi(f(x_k))\,.
\end{align*}

Consequently, $f$ is a $q^s$-Caristi mapping without fixed points.
\end{proof}

\begin{remark}
Caristi-type fixed point theorems in complete partial metric spaces were also  proved by Karapinar \emph{et al.} in \cite{karapin13} and \cite{karapin11}. Since a complete partial metric space is 0-complete, but the converse is not true (see \cite{romag10}), these results follow from those proved by Romaguera \cite{romag10}
\end{remark}

Another definition of Caristi condition in partial metric spaces was given by Acar,  Altun and Romaguera \cite{altun-romag13}. A mapping $f:X\to X$ is called \emph{AR-Caristi} if

\begin{equation}\tag{AR-Car$_\varphi$}\label{def.AR-Car-pm}
p(x,f(x))\le p(x,x)+\varphi(x)-\varphi(f(x))\,,
\end{equation}
 for some $q^s$-lsc bounded from below function $\varphi:X\to\mathbb{R}.$
 \begin{theo}
 [Acar,  Altun and Romaguera \cite{altun-romag13}] \label{t2.AR-Car-pm} A partial metric space $(X,p)$ is complete if and only if every AR-Caristi mapping on X has a fixed point.
 \end{theo}\begin{proof}
   Suppose that $(X,p)$ is complete. Let $f:X\to X$ be a mapping satisfying the condition \eqref{def.AR-Car-pm} for some $q^s$-lsc bounded from below function $\varphi:X\to\mathbb{R}$.
 By Proposition \ref{p2.pm-qm} the function  $\beta:X\to [0,\infty)$ given by $\beta(x)=p(x,x),\, x\in X$, is $q^s$-continuous, so that the function $\psi:=\beta+2\varphi$ is $q^s$-lsc and bounded from below (by $\,2\inf \varphi(X)$).

Putting  $\varphi=2^{-1}(\psi-\beta)$ in \eqref{def.AR-Car-pm} and  taking into account the definition \eqref{def.m-pm} of the metric $q^s$ associated to the partial metric $p$, one obtains
 \begin{equation}\label{eq.AR-Car-pm}
   q^s(x,f(x))\le \psi(x)-\psi(f(x))\,.
\end{equation}

   Since, by Proposition \ref{p2.pm-qm} the metric space $(X,q^s)$ is complete, we can apply Caristi's fixed point theorem (Theorem \ref{t.caristi1}) to the mapping $f$ and the $q^s$-lsc function $\psi$ to conclude that $f$ has a fixed point.

  The proof of the converse follows the same line as that of the corresponding implication in Theorem \ref{t1.Caristi-pm}.

   Suppose that $(x_n)_{n\in\mathbb{N}_0}$ ($\mathbb{N}_0=\{0,1,2,\dots\}$) is a Cauchy sequence in $(X,p)$ which is not convergent. Passing to a subsequence, if necessary,   we can suppose further that
   \begin{equation}\label{eq1.AR-Car-pm}
   p(x_n,x_{n+1})- p(x_n,x_{n})<\frac1{2^{n+1}}\,,
   \end{equation}
   for all $n\in\mathbb{N}_0$ (see Remark \ref{re.C-seq-pm-sp}).  It follows that the set
   $$
A:=\{x_n: n\in\mathbb{N}_0\}\,.
$$
is $q^s$-closed in $(X,q^s)$.

   Define the mappings $f:X\to X$ and $\varphi:X\to[0,\infty)$ by the formulae \eqref{eq.Caristi-pm}. Then $\varphi$ is $q^s$-lsc. It is obvious that the mapping $f$ has no  fixed points, so it remains to show that  it satisfies the condition \eqref{def.AR-Car-pm}.

   For $ x\in X\smallsetminus A,$
   \begin{align*}
     p(x,f(x))=&\,p(x,x_0)=\varphi(x)-\varphi(f(x))\\
     \le& \, p(x,x)+ \varphi(x)-\varphi(f(x))\,,
   \end{align*}
while for $x=x_n\in A$,
\begin{align*}
     p(x_n,f(x_n))= &\, p(x_{n},x_{n+1})<p(x_{n},x_{n+1})+\frac1{2^{n+1}}\\
     = &\,  p(x_n,x_n)+ \varphi(x_n)-\varphi(f(x_n))\,.
   \end{align*}
   \end{proof}

  \begin{remark} One can think to use the relations
  $$
  \psi=\beta+2\varphi \iff \varphi=\frac12(\psi-\beta)\,,
  $$
  in the proof of the converse. Indeed if $(X,p)$ is not complete, then $(X,q^s)$ is not complete (see Proposition \ref{p2.pm-qm}), so, by Corollary  \ref{c.Car-compl}, there exists a mapping $f:X\to X$ without fixed points which satisfies \eqref{eq.AR-Car-pm} for some $q^s$-lsc bounded from below function $\psi:X\to\mathbb{R}$.
  The function  $\varphi=\frac12(\psi-\beta)$ is $q^s$-lsc (because $\beta $ is $q^s$-continuous) and replacing $\psi$ by $ \beta+2\varphi$ in
  \eqref{eq.AR-Car-pm} one obtains \eqref{def.AR-Car-pm}.

  Unfortunately, it is not sure that the  function $\varphi=\frac12(\psi-\beta)$ is bounded form below, in order to obtain a contradiction.
 \end{remark}

 \begin{remark} Caristi's Fixed Point Theorem for set-valued mappings on partial metric spaces is discussed in a recent paper by Alsiary and Latif \cite{latif15}.
\end{remark}

\subsection{Ekeland Variational Principle (EkVP) in partial metric spaces}

In this subsection we shall show that in partial metric spaces Caristi's FPT is also equivalent to weak Ekeland principle.

\begin{theo}[Ekeland Variational Principle - weak form (wEk)]\label{t.wEk-pm} Let $(X,p)$ be a $0$-complete partial metric space and $\varphi:X\to\mathbb{R}\cup\{+\infty\}$ a $q^s$-lsc   bounded below proper function. Then for every
$\varepsilon >0$ there exists $x_\varepsilon\in X$ such that
\begin{equation}\label{eq1.wEk-pm}
\forall x\in X\smallsetminus\{x_\varepsilon\}, \;\; \varphi(x_\varepsilon)<\varphi(x)+\varepsilon p(x,x_\varepsilon)\,.
\end{equation}
\end{theo}\begin{proof}
  Suppose on the contrary that there exists $\varepsilon >0$ such that
  \begin{equation}\label{eq2.wEk-pm}
\forall x\in X,\; \exists y_x\in X\smallsetminus\{x\} \;\;\mbox{with}\;\;   \varphi(x)\ge\varphi(y_x)+\varepsilon p(x,y_x)\,.
\end{equation}

Consider a point   $x_0\in X$ such that  $\varphi(x_0)\le\inf\varphi(X)+\varepsilon$ and let
\begin{equation}\label{eq2a.wEk-pm}
Y:=\{x\in X : \varphi(x)+\varepsilon p(x_0,x)\le \varphi(x_0)+\varepsilon\, p(x_0,x_0)\}\,.
\end{equation}

By  Proposition \ref{p2.pm-qm}, the function $p(x_0,\cdot)$ is $q^s$-continuous, hence  the function $\varphi(\cdot)+\varepsilon\, p(x_0,\cdot)$  is $q^s$-lsc. Consequently,  the set $Y$ is $p^s$-closed, and so 0-complete. Indeed if $(x_n)$ is a 0-Cauchy sequence in $Y$, then it has a $\tau_p$-limit $x\in X$ such that  $p(x,x)=0$. But this implies $x_n\xrightarrow{q^s}x$ (see Remark \ref{re.0-compl-pm}) and so $x\in Y$. Also $Y\ne\emptyset$ because $x_0\in Y$ and $\varphi$ is finite on $Y$ (i.e. $\varphi(x)\in \mathbb{R}$ for all $x\in Y$).

Observe that the element $y_x$ given by \eqref{eq2.wEk-pm} belongs to $Y$ for every $x\in Y$.
Indeed, if $x\in Y$, then
\begin{align*}
 \varphi(y_x)+\varepsilon p(x_0,y_x)&\le\varphi(x)-\varepsilon p(x,y_x)+\varepsilon p(x_0,y_x)\\
 &\le\varphi(x_0)+\varepsilon \,p(x_0,x_0)+\varepsilon[p(x_0,y_x)-p(x_0,x)-p(x,y_x)]\\
 &\le\varphi(x_0)+\varepsilon \,p(x_0,x_0)\,,
 \end{align*}
 because $\,p(x_0,y_x)-p(x_0,x)-p(x,y_x)\le 0$. This last inequality follows from
$$
p(x_0,y_x)\le p(x_0,x)+p(x,y_x)-p(x,x)\le p(x_0,x)+p(x,y_x)\,.
$$

Put now $\tilde \varphi:=\varepsilon^{-1}\varphi|_Y:Y\to\mathbb{R}$ and let $f:Y\to Y$ be  defined  by
   $f(x)=y_x$, where, for $x\in Y$, $y_x\ne x$ is the element of $Y$ satisfying \eqref{eq2.wEk-pm}.

   Then the  inequality \eqref{eq2.wEk-pm} is equivalent to
   $$
   p(x,f(x))\le\tilde\varphi(x)-\tilde\varphi(f(x))\;\; x\in Y\,,
   $$
   which shows that $f$ is a Caristi mapping with respect to  $\tilde \varphi$. Since $f$ has no fixed points, this is in contradiction to Caristi's fixed point theorem (Theorem \ref{t1.Caristi-pm})
 \end{proof}

We show now that the converse implication also holds.

\begin{prop}\label{p.Caristi-Ek-pm}   Ekeland's Variational Principle in its weak form (Theorem \ref{t.wEk-pm}) implies Caristi's Fixed Point Theorem (Theorem \ref{t1.Caristi-pm}).
  \end{prop}\begin{proof}
Let $(X,p)$ be a 0-complete partial metric space, $\varphi:X\to\mathbb{R}$ a $q^s$-lsc bounded from below function and $f:X\to X$ a Caristi mapping with respect to  $\varphi.$  By  Theorem \ref{t.wEk-pm} applied to $\varphi$ for  $\varepsilon =1$   there exists  a point $x_1\in X$ such that
$$
\varphi(x_1)<\varphi(x)+p(x_1,x) \,,
$$
for all $x\in X\smallsetminus\{x_1\}$. Supposing $f(x_1)\ne x_1,$ we can take $x=f(x_1)$ in the above inequality  to obtain
$$
p(x_1,f(x_1))>\varphi(x_1)-\varphi(f(x_1))\,,
$$
in contradiction to the inequality \eqref{def.Caristi-pm} satisfied by $f$.

Consequently  $f(x_1)=x_1$, that is $x_1$ is  a fixed point of $f$.
  \end{proof}

  \begin{remark}\label{re.wEk-0-compl-pm}
   It follows that the validity of  Ekeland's Variational Principle in its weak form, as given in Theorem \ref{t.wEk-pm}, is also equivalent to
   the 0-completeness of the partial metric space $(X,p)$.
  \end{remark}

We shall present now  the version of Ekeland Variational Principle that can be obtained from Theorem \ref{t2.AR-Car-pm}.

\begin{theo}[Ekeland Variational Principle 2 - weak form (wEk2)]\label{t.AR-wEk-pm} Let $(X,p)$ be a complete partial metric space and $\varphi:X\to\mathbb{R}\cup\{\infty\}$ a  $q^s$-lsc   bounded below proper function. Then for every
$\varepsilon >0$ there exists $x_\varepsilon\in X$ such that
\begin{equation}\label{eq1.AR-wEk-pm}
\forall x\in X\smallsetminus\{x_\varepsilon\}, \;\; \varphi(x_\varepsilon)+\varepsilon p(x_\varepsilon,x_\varepsilon)< \varphi(x)+\varepsilon p(x,x_\varepsilon)\,.
\end{equation}
\end{theo}\begin{proof} Suppose,   by contradiction, that there exists an $\varepsilon>0$ such that
\begin{equation}\label{eq2.AR-wEk-pm}
\forall x\in X,\; \exists y_x\in X\smallsetminus\{x\} \;\;\mbox{with}\;\;   \varphi(x)+\varepsilon p(x,x)\ge\varphi(y_x)+\varepsilon p(x,y_x) \,,
\end{equation}
and let $x_0\in X$ be such that  $\varphi(x_0)\le \varepsilon+\inf \varphi(X)$.

To get rid of the points where $\varphi$ takes the value $+\infty$, consider again the set $Y$ given by
\eqref{eq2a.wEk-pm}. Then $Y$ is nonempty ($x_0\in Y$) and $q^s$-closed and so complete with respect to  the partial metric $p$. Indeed if $(x_n)$ is a Cauchy sequence in $(X,p)$ then, by the definition of the completeness,  it converges properly to some $x\in X$.
By Proposition \ref{p2.pm-qm}, $(x_n)$ is $q^s$-convergent to $x$ and so  $x\in Y$.

Observe that $x\in Y$ implies that the element $y_x $ given by \eqref{eq2.AR-wEk-pm} also belongs
to $Y$. Indeed, if $x\in Y$, then
\begin{align*}
 \varphi(y_x)+\varepsilon p(x_0,y_x)&\le\varphi(x)+\varepsilon [p(x_0,y_x)- p(x,y_x)+p(x,x)]\\
 &\le\varphi(x_0)+\varepsilon [p(x_0,x_0)-p(x_0,x)+p(x_0,y_x)-p(x,y_x)+p(x,x)]\\
 &\le\varphi(x_0)+\varepsilon \,p(x_0,x_0)\,,
 \end{align*}
 because
 $$p(x_0,y_x)-p(x_0,x)-p(x,y_x)+p(x,x)\le 0 \iff
p(x_0,y_x)+p(x,x)\le p(x_0,x)+p(x,y_x) \,,
$$
and the last inequality is true, by the triangle inequality (PM4) from Definition \ref{def.pm}.

Taking again $\tilde \varphi=\varepsilon^{-1}\varphi|_Y$ and $f:Y\to Y$ defined by $f(x)=y_x,$ where for $\, x\in Y$ the element $y_x\in Y$ is given by \eqref{eq2.AR-wEk-pm}, the function $\tilde \varphi $ is $q^s$-lsc and  $f $ is a mapping on $Y$ without fixed points, satisfying \eqref{def.AR-Car-pm} for $ \varphi=\tilde \varphi$.
\end{proof}

The converse implication holds in this case too. The proof is similar to that of Proposition \ref{p.Caristi-Ek-pm}.
\begin{prop}\label{p.AR-Caristi-Ek2-pm}   Ekeland's Variational Principle in its weak form, as given in Theorem  \ref{t.AR-wEk-pm}, implies Caristi's Fixed Point Theorem, as given in Theorem \ref{t2.AR-Car-pm}.
  \end{prop}\begin{proof}
Let $(X,p)$ be a  complete partial metric space, $\varphi:X\to\mathbb{R}$ a $q^s$-lsc bounded from below function and $f:X\to X$ a mapping satisfying \eqref{def.AR-Car-pm}.  Applying Theorem \ref{t.wEk-pm} to $\varphi$ for  $\varepsilon =1$ it follows the existence of a point $x_1\in X$ such that
$$
\varphi(x_1)+p(x_1,x_1)<\varphi(x)+p(x_1,x) \,,
$$
for all $x\in X\smallsetminus\{x_1\}$. Supposing $f(x_1)\ne x_1,$ we can take $x=f(x_1)$ in the above inequality  to obtain
$$
p(x_1,f(x_1))>p(x_1,x_1)+\varphi(x_1)-\varphi(f(x_1))\,,
$$
in contradiction to the inequality \eqref{def.AR-Car-pm} satisfied by $f$.

Consequently  $f(x_1)=x_1$, that is $x_1$ is  a fixed point of $f$.
  \end{proof}

   \begin{remark}\label{re1.wEk2-compl-pm}
   It follows that the validity of  Ekeland's Variational Principle in its weak form, as given in Theorem \ref{t.AR-wEk-pm},   is  equivalent to
   the completeness of the partial metric space $(X,p)$.
  \end{remark}

  \begin{remark}\label{re2.wEk2-compl-pm}
  A version of Ekeland Variational Principle in  partial metric spaces was proved by Aydi, Karapinar and Vetro \cite{karapin15}.
  \end{remark}

  \subsection{Dislocated metric spaces}

This class of spaces was considered by Hitzler and Seda \cite{hitzler00} in connection with some problems in logic programming.  A \emph{dislocated metric} on a set $X$ is a function $\rho:X\times X\to\mathbb{R}_+$ satisfying the conditions:
\begin{align*}
{\rm (DM1)}&\quad \rho(x,y)=0\; \Rightarrow\; x=y;\\
{\rm (DM2)}&\quad \rho(x,y)=\rho(y,x);\\
{\rm (DM3)}&\quad \rho(x,y)\le\rho(x,z)+\rho(z,y),
\end{align*}
for all $x,y,z\in X$. If $\rho$ satisfies only (DM1) and (DM3), then it is called a \emph{dislocated quasi-metric}. The pair $(X,\rho)$ is called a \emph{dislocated metric } (resp. a \emph{dislocated quasi-metric}) \emph{space}.

These spaces are close to partial metric spaces (in this case it is also  possible that $\rho(x,x)>0$ for some $x\in X$), with the exception that a dislocated metric satisfies the usual triangle inequality (DM3) instead of the inequality (PM4) from Definition \ref{def.pm}. In fact, any partial metric is a dislocated metric.

For $x\in X$ and $r>0$ the open ball $B(x,r)$ is defined by $B(x,r)=\{y\in X : \rho(x,y)<r\}$.

Hitzler and Seda \cite{hitzler00}  defined a kind of topology on a dislocated metric space $(X,\rho)$ in the following way. Instead of the membership relation $\in\,$ they considered   a relation $\prec$ in $X\times2^X$,  defined for $(x,A)\in X\times2^X$ by
$$
x \prec A \iff \exists \epsic>0\;\mbx{ such that }\; B(x,\epsic)\subseteq A\,.$$

The d-\emph{neighborhood system }$\mathcal V(x)$ of a point $x\in X$ is defined by the condition
$$
V\in\mathcal V(x) \iff V\subseteq X\quad\mbx{ and }\quad x\prec V\,.$$

(Here ``d-" comes from ``dislocated-").

The neighborhood axioms are satisfied with the relation $\prec$ instead of $\in$.
\begin{align*}
  &{\rm (V1)}\;\; V\in\mathcal V(x)\;\Longrightarrow\; x\prec V;\\
   &{\rm (V2)}\;\; V\in\mathcal V(x)\;\mbx{ and }\; V\subseteq U\; \Longrightarrow\; U\in\mathcal V(x);\\
 &{\rm (V3)}\;\; U,V\in\mathcal V(x)\; \Longrightarrow\; U\cap V\in\mathcal V(x);  \\
  &{\rm (V4)}\;\; V\in\mathcal V(x)\; \Longrightarrow\;\exists  W\in \mathcal V(x),\, W\subseteq V\;\mbx{ such that }\;   V\in\mathcal V(y)\;\mbx{ for all }\; y\prec W.
\end{align*}

It is easy to check the validity of these properties. As a sample, let us check (V4). For  $V\in \mathcal V(x)$ let $\epsic >0$ be such that $B(x,\epsic)\subseteq V$. If $y\prec B(x,\epsic)$, then there exists $\epsic'>0$ such that
$B(y,\epsic')\subseteq B(x,\epsic)\subseteq V,$ so that $V\in\mathcal V(y)$. It follows that we can take $W=B(x,\epsic)$.

The so defined ``neighborhood system" is not a proper neighborhood system (i.e., with respect to the relation $\in$), because the relation $x\in V$ is not always satisfied -- it is not sure that $x\in B(x,\epsic)$ and further, the ball $B(x,\epsic)$ could be empty for some $\epsic$.
\begin{example}
Let $X$ be a set of cardinality at least 2. Define $\rho(x,x)=1$ and $\rho(x,y)=2$ if $x\ne y$, for all $x,y\in X$. Then $B(x,\epsic)=\emptyset$ for $0<\epsic\le 1$, implying that every subset of $X$ (including the empty set) is a d-neighborhood of $x$.
\end{example}

In fact the following properties hold.
\begin{prop}[\cite{hitzler00}, Proposition 3.2]\label{p1.hitz}
Let $(X,\rho)$  be  a dislocated metric space.
\begin{enumerate}
\item The following conditions are equivalent:

{\rm (i)}\; \; $\rho$ is a metric.

{\rm (ii)}\;\: $\rho(x,x)=0$\; for all \; $x\in X$.

{\rm (iii)}\; $B(x,\epsic)\ne\emptyset$\; for all \; $x\in X$ and $\epsic >0$.
\item The subset $\ker \rho :=\{x\in X : \rho(x,x)=0\}$ is a metric space with respect to $\rho$.
\end{enumerate}\end{prop}

A sequence $(x_n)$ in $X$ is called d-\emph{convergent} to $x\in X$ if
$$
\forall V\in \mathcal V(x),\; \exists n_0\in\mathbb{N}\;\mbx{ such that }\; x_n\in V\;\mbx{ for all }\; n\ge  n_0\,.$$

The sequence $(x_n)$ in $X$ is called $\rho$-\emph{convergent} to $x$ if $\lim_{n\to\infty}\rho(x,x_n)=0$.
\begin{remark}
Note again that this type of convergence is not a proper convergence. For instance, if $\rho(x,x)>0$, then the constant sequence $x_n=x,\, n\in\mathbb{N}$, is not $\rho$-convergent to $x$.
\end{remark}

The following property holds.
\begin{prop}[\cite{hitzler00}, Proposition 3.9]\label{p2.hitz}
Let  $(X,\rho)$ be    a dislocated metric space. A sequence $(x_n)$ in $X$ is $\rho$-convergent to $x\in X$ if and only if it is d-convergent to $x$.
\end{prop}\begin{proof}
  Suppose that $(x_n)$ is d-convergent to $x$. For   $\epsic >0,\, B(x,\epsic)$ is a d-neighborhood of $x$, so there exists $n_0\in \mathbb{N}$ such that $x_n\in B(x,\epsic)\iff \rho(x,x_n)<\epsic$, for all $n\ge n_0$, showing that $\rho(x,x_n)\to 0$.

  Suppose now that    $\rho(x,x_n)\to 0$. For    $V\in \mathcal V(x)$ let $\epsic >0$ be such that $B(x,\epsic)\subseteq V$. By hypothesis there exists $n_0\in \mathbb{N}$ such that $\rho(x,x_n)<\epsic$ for all $n\ge n_0$. It follows $x_n\in B(x,\epsic)\subseteq V$ for all $n\ge n_0$.
\end{proof}

A sequence $(x_n)$ in $X$ is called $\rho$-\emph{Cauchy} if for every $\epsic >0$ there exists $n_0\in \mathbb{N}$ such that $\rho(x_n,x_m)<\epsic$ for all $m,n\ge n_0$. The dislocated metric space $(X,\rho)$ is called \emph{complete} if every Cauchy sequence is $\rho$-convergent. Hitzler and Seda \cite[Theorem 2.7]{hitzler00} proved   that  Banach's  contraction principle holds in complete dislocated metric spaces.

Pasicki \cite{pasicki15} defined  a topology $\tau_\rho$ on a dislocated metric space $(X,\rho)$ in the following way.
The family of subsets $\{B(x,r) : x\in X,\, r>0\}$ satisfies  $X=\bigcup\{B(x,r) : x\in X,\, r>0\}$, so it is a subbase for a topology $\tau_\rho$ on $X$ (see Kelley \cite[Theorem 12, p. 47]{Kelley}).

It follows that a subset $U$ of $X$ is a neighborhood of $x\in X$ if and only if
\begin{equation}
 \exists  n\in \mathbb{N},\; \exists y_1,\dots,y_n\in  X,\;\exists r_1,\dots,r_n>0, \;\mbx{such that } x\in B(y_1,r_1)\cap\dots\cap B(y_n,r_n)\subseteq U\,.\end{equation}

 Denote by  $\mathcal U_\rho(x)$ the neighborhood system  of a point $x\in X$ with respect to $\tau_\rho$.

\begin{remark} Let $(x_n)$ be a sequence in a dislocated metric space $(X,\rho)$ and $x\in X$.
  If $\lim_n\rho(x,x_n)=0$, then the sequence $(x_n)$ is $\tau_\rho$-convergent to $x\in X$.
  \end{remark}

  Indeed, for any $\tau_\rho$-neighborhood $U$ of $x$ there exists $y\in X$ and $\epsic >0$ such that $x\in B(y,\epsic)\subseteq U$.  Then $\epsic-\rho(x,y)>0$ so that, by hypothesis, there exists $n_0\in \mathbb{N}$ such that $ \rho(x,x_n)<\epsic-\rho(x,y)$ for all $n\ge n_0$. It follows
$$
\rho(y,x_n)\le\rho(y,x)+\rho(x,x_n)<\rho(x,y)+\epsic-\rho(x,y)=\epsic\,,$$
that is $x_n\in B(y,\epsic)\subseteq U$ for all $n\ge n_0$, showing  that $(x_n)$ is $\tau_\rho$-convergent to $x$.

\begin{remark} I don't know a characterization of the $\tau_\rho$-convergence in  terms of the sequence $(\rho(x,x_n))_{n\in\mathbb{N}}$.
  \end{remark}

  Apparently unaware of Hitzler and Seda paper \cite{hitzler00}, Amini-Harandi \cite{harandi12}
  defined dislocated metric spaces calling them metric-like spaces. He  defined the balls by analogy with partial metric spaces:
  $$
  \tilde B(x,\epsic)=\{y\in X :|\rho(x,y)-\rho(x,x)|<\epsic\}\, .$$
  The family of balls $\tilde B(x,r),\, x\in X,\, \epsic >0\,$ form the base of a topology $\tilde \tau_\rho$ on the dislocated metric space $(X,\rho)$.

  A sequence $(x_n)$ in $X$ is $\tilde \tau_\rho$-convergent to $x\in X$ if and only if $\lim_{n\to\infty}\rho(x,x_n)=\rho(x,x)$.

  A sequence $(x_n)$ in $X$ is called \emph{Cauchy} if there exists the limit $\lim_{m,n\to \infty}\rho(x_n,x_m)\in\mathbb{R}$.  The space $(X,\rho)$ is called complete if for every Cauchy sequence $(x_n)$ in $X$ there exists $x\in X$ such that
  $$
  \lim_{n\to\infty}\rho(x,x_n)=\rho(x,x)=\lim_{m,n\to\infty}\rho(x_m,x_n)\,,$$
(compare with Subsection \ref{Ss.seq-pm}, Definition \ref{def.prop-converg-pm-sp}).

  The paper \cite{hitzler00} contains some fixed point theorems in complete dislocated metric spaces.
  The same approach  is adopted in the paper \cite{karapin13b} (and possibly in other papers).

  \begin{remark}
     In fact, in a preliminary version of the paper \cite{pasicki15}, Pasicki  called these spaces near metric spaces. After the reviewer draw his attention to Hitzler and Seda paper, he changed    to dislocated metric spaces. There are a lot of papers dealing with fixed point results in dislocated metric spaces (or in metric like spaces), as it can be seen by a simple search on   MathSciNet, ZbMATH or ScholarGoogle. I don't know if there are some converse results -- i.e. completeness implied by the validity of some fixed point results.  \end{remark}

\subsection{Other generalized metric spaces}
In this subsection we  shall present some completeness results in other classes of generalized metric spaces:  dislocated metric spaces, $w$-spaces and $\tau$-spaces. Good surveys of various generalizations of metric spaces are given in the papers by Ansari \cite{ansari11},  Berinde and Choban \cite{berinde13}, and in the books \cite{Deza}, \cite{Kirk-Shah} and \cite{Rus-PP}.\\

$w$-\textbf{distances}

This notion was introduced by Kada \emph{et al}. \cite{suzuki-taka96}. Let $(X,\rho)$ be a metric space.
A mapping $p:X\times X\to\mathbb{R}_+$ is called a $w$-\emph{distance} if, for all $x,y,z\in X$,
\begin{itemize}
\item[{\rm(w1)}]\quad $p(x,y)\le p(x,z)+p(z,y);$
\item[{\rm(w2)}]\quad $p(x,\cdot)$ is $\rho$-lsc;
\item[{\rm(w3)}]\quad $\forall \varepsilon >0,\; \exists\delta >0$\; such that  \; $p(x,y)<\delta$ \; and\; $p(x,z)<\delta$\; implies\; $p(y,z)<\varepsilon$.
\end{itemize}\vspace{2mm}

$\tau$-\textbf{distances}

A more involved notion was introduced by Suzuki \cite{suzuki01}. Let $(X,\rho)$ be a metric space and $\eta:X\times\mathbb{R}_+\to\mathbb{R}_+$.
A mapping $p:X\times X\to\mathbb{R}_+$ is called a $\tau$-\emph{distance} if

\begin{itemize}
\item[{\rm ($\tau$1)}]\;\;  $p(x,y)\le p(x,z)+p(z,y)$\; for all\; $x,y,z\in X$;
\item[{\rm($\tau$2)}]\;\; for every $x\in X$ the function $\eta(x,\cdot)$ is concave and continuous, $\eta(x,0)=0$ and $\eta(x,t)\ge t$ for all $(x,t)\in X\times \mathbb{R}_+$;
\item[{\rm($\tau$3)}]\;\; $\lim_nx_n=x$ and $\lim_n\left(\sup_{m\ge n}\eta(z_n,p(z_n,x_m))\right)=0$\; imply\; $p(w,x)\le \liminf_np(w,x_n)$ for all $w\in X$;
\item[{\rm($\tau$4)}]\;\; $\lim_n\left(\sup_{m\ge n}p(x_n,y_m)\right)=0\;$ and\;$\lim_n\eta(x_n,t_n)=0$\; imply\; $\lim_n\eta(y_n,t_n)=0$;
\item[{\rm($\tau$5)}]\;\; $\lim_n \eta(z_n,p(z_n,x_n))=0$\; and \; $\lim_n \eta(z_n,p(z_n,y_n))=0$\; imply\; $\lim_n\rho(x_n,y_n)=0$\,.
\end{itemize}

\begin{remark}
  It was shown in \cite{suzuki01} that ($\tau$2) can be replaced by \\

 ($\tau$2$'$)\quad for every $x\in X$ the function $\eta(x,\cdot)$ is increasing and $\inf_{t>0}\eta(x,t)=0$\,.
  \end{remark}

  Lin and Du \cite{lin-du06,lin-du08} propose a slightly  simplified version of a $\tau$-function.

  Let $(X,\rho)$ be a metric space.
A mapping $p:X\times X\to\mathbb{R}_+$ is called a (LD$\tau$)-\emph{distance} if
\begin{itemize}
\item[{\rm (LD-$\tau$1)}]\;\;  $p(x,y)\le p(x,z)+p(z,y)$\; for all\; $x,y,z\in X$;
\item[{\rm(LD-$\tau$2)}]\;\; for every $x\in X$  and every  sequence $(y_n)$  in $X$ converging to some $y\in X$, if  for some $M>0,\, p(x,y_n)\le M,$ for all $n$, then $p(x,y)\le M$;
\item[{\rm(LD-$\tau$3)}]\;\; if $(x_n$) and  $(y_n)$ are sequences in $X$ such that
$\lim_n\left(\sup_{m\ge n} p(x_n,x_m)\right)=0$ and  $\lim_np(x_n,y_n)=0$, then
    $\lim_n\rho(x_n,y_n)=0$;
\item[{\rm(LD-$\tau$4)}]\;\; for all $x,y,z\in X,\,  p(x,y)=p(x,z)=0$ implies $y=z$ \,.
\end{itemize}

\begin{remark}\hfill

1. If, for every $x\in X$, $p(x,\cdot)$ is lsc, then condition (LD-$\tau2$) is satisfied.

2. If $p$ satisfies (LD-$\tau3$), then every sequence $(x_n)$ in $X$ satisfying
  $$\lim_n\left(\sup_{m\ge n} p(x_n,x_m)\right)=0$$ is a Cauchy sequence.
\end{remark}

 Lin and Du proved in  \cite{lin-du06,lin-du08} variational principles of Ekeland type for this kind of function, and for the $w$-distance in \cite{lin-du07}.\\

  \textbf{Tataru distance}

  This was defined by Tataru \cite{tataru92} in the following way. Let $X$ be a subset of a Banach space $E$. A family $\{T(t) : t\in\mathbb{R}_+\}$ of mappings on $X$ is called a
\emph{strongly continuous semigroup of nonexpansive mappings} on $X$ if
\begin{itemize}
\item[{\rm (Sg1)}]\;\; for every $t\in \mathbb{R}_+,\, T(t)$ is a nonexpansive mapping on $X$;
\item[{\rm (Sg2)}]\;\; $T(0)x=x$\; for all\;  $x\in X$;
\item[{\rm (Sg3)}]\;\; $T(s+t)=T(s)T(t)$ \; for all\; $s,t\in\mathbb{R}_+$;
\item[{\rm (Sg4)}]\;\; for each $x\in X$ the mapping $T(\cdot)x:\mathbb{R}_+\to X$ is continuous.
\end{itemize}

 The \emph{Tataru distance} corresponding to a strongly continuous semigroup $\{T(t) : t\in\mathbb{R}_+\}$ of nonexpansive mappings on $X$  is  defined for $x,y\in X$ by
\begin{equation}\label{def.Tataru-dist}
p(x,y)=\inf\{t+\|T(t)x-y\| : t\in\mathbb{R}_+\}\,.\end{equation}

It was shown by Suzuki, \cite{suzuki01} and \cite{suzuki08c}, that any $w$-distance is a $\tau$-distance, but the converse does not hold -- for instance, the Tataru distance is a  $w$-distance but not a $\tau$-distance. The paper \cite{suzuki08c} contains many examples of $w$-distances  and $\tau$-distances, other $\tau$-distances which are not $w$-distances, and   conditions under which the  Tataru distance is a $\tau$-distance.

Various fixed point results, Ekeland-type  principles  and completeness   for $\tau$-distances were proved by Suzuki in \cite{suzuki01,suzuki04,suzuki05,suzuki06,suzuki09,suzuki09b,suzuki10}.

Fixed points for contractions and completeness results in quasi-metric spaces endowed with a $w$-distance were proved by Alegre \emph{et al}. \cite{alegre-romag14}, for single-valued maps, and by Mar\'{\i}n \emph{et al}. \cite{marin-romag11}, for set-valued ones. Similar results in the case of partial metric spaces were obtained by Altun and Romaguera \cite{altun-romag12}.

A mapping $f$  on a metric space $(X,\rho)$ for which there exist a $w$-distance $p$ on $X$ and a number $\alpha\in[0,1)$ such that
\begin{equation}\label{def.w-contr}
p(f(x),f(x'))\le \alpha p(x,y)\quad\mbox{for all   } x,x'\in X\,,
\end{equation}
is called \emph{weakly contractive} (or a \emph{weak contraction}). In the case of a set-valued mapping
$F:X\rightrightarrows X$, the condition \eqref{def.w-contr} is replaced by
\begin{equation}\label{def2.w-contr}
\forall x,x'\in X,\; \exists y\in F(x),\, y'\in F(x'),\; \mbox{ such that  }\; p(y,y')\le \alpha p(x,y)\,.
\end{equation}

Direct and converse fixed point results involving   completeness for weakly contractive mappings and for other types of mappings (e.g. Kannan maps) on metric spaces endowed with a $w$-distance were proved in \cite{taka12,taka10,suzuki-taka96,taka98,suzuki-taka96b,taka99,taka13} (see also the books \cite{Kirk-Shah} and \cite{Taka}). For instance, in \cite{suzuki-taka96b} it is proved that a metric space $X$ is complete if and only if  every weakly contractive mapping on $X$ has a fixed point. Also, the result of Borwein \cite{borw83} (see Corollary \ref{c.Borw-compl}), on the completeness of convex subsets of normed spaces on which every contraction has a fixed point, is rediscovered.\\

\textbf{Branciari's distance -- generalized metric spaces}

Branciari \cite{branciari00} (see \cite{samet10} for some corrections) introduced a new class of
spaces, called \emph{generalized metric spaces}, in the following way. A function $d:X\times X\to\mathbb{R}_+$, where $X$ is a nonempty set, is called a \emph{generalized metric} if the following conditions hold
\begin{itemize}
  \item[{\rm (GM1)}]\;\; $d(x,y)= 0\iff x=y$;
  \item[{\rm (GM2)}]\;\; $d(x,y)= d(y,x)$;
  \item[{\rm (GM3)}] \;\; $d(x,y)\le d(x,u)+d(u,v)+d(v,y)$\,,
  \end{itemize}
  for all  pairwise distinct points  $x,y,u,v\in X$ . The generalized triangle inequality (GM3) causes several troubles concerning the topology of these spaces (it is not always Hausdorff and  the distance function $d(\cdot,\cdot)$ is continuous only under a supplementary condition, see \cite[Ch. 13]{Kirk-Shah}) and the completeness.  Branciari loc. cit. proved a Banach Contraction Principle within this context (some flaws in the original proof are corrected in   \cite[Ch. 13]{Kirk-Shah}).

  Ghosh and Deb Ray \cite{ghosh12} considered Suzuki's generalized contractions for these spaces and proved direct fixed point results as well as  converse completeness results.\\

\textbf{Probabilistic metric spaces}

Completeness as well as  relations between  completeness   and  fixed point results in probabilistic metric spaces are explored in the papers \cite{razani08}, \cite{saadat09}, \cite{grab-cho08}, \cite{saadat05}. We do not enter into  the details of this matter.

\section*{Appendix -- A pessimistic  conclusion}
  In conclusion we quote from the review of the paper \cite{park86}.\\

\textbf{MR835839 (87m:54125) }
Park, Sehie; Rhoades, B. E.
\emph{Comments on characterizations for metric completeness}.
Math. Japon. 31 (1986), no. 1, 95--97.\\

\begin{quotation}
There are many papers in which the completeness of a metric space is characterized by using a
fixed point theorem. In the present paper the authors prove two very simple and general theorems
which ``\textbf{encompass some previous as well as future theorems of this type}".\end{quotation}
\hfill (Reviewed by J. Matkowski)\\

Under these circumstances, it seems that the  best we can hope  to do in this domain is to prove some particular cases of these very general results.\\

\textbf{Acknowledgement.} This is an expanded version of a talk delivered at the \emph{International Conference on Nonlinear Operators, Differential Equations and Applications} (ICNODEA 2015), Cluj-Napoca, Romania, July 14-17, 2015.

\providecommand{\bysame}{\leavevmode\hbox to3em{\hrulefill}\thinspace}
\providecommand{\MR}{\relax\ifhmode\unskip\space\fi MR }

\end{document}